\documentclass[12pt,leqno,twoside]{article}
\usepackage{amssymb}
\usepackage{amsmath}
\usepackage{amsthm,amscd}
\usepackage{a4,indentfirst,latexsym}
\usepackage[french]{babel}   % For French text
\usepackage{graphics}
 \usepackage[T1]{fontenc}
\usepackage{mathrsfs}
\usepackage{cite,enumitem,graphicx}
\usepackage[colorlinks=true,urlcolor=black,
citecolor=black,linkcolor=black,linktocpage,pdfpagelabels,
bookmarksnumbered,bookmarksopen]{hyperref}
\usepackage[colorinlistoftodos]{todonotes}
\usepackage{color}
\usepackage{cite}
\usepackage{graphicx}
\usepackage{amssymb,bbm,bm}
\usepackage[titletoc]{appendix}
\usepackage{blindtext}
\usepackage{geometry}
\usepackage[colorinlistoftodos]{todonotes}
\usepackage{appendix}
\newtheorem{theorem}{Theorem}[section]
\newtheorem{lemma}[theorem]{Lemma}
\newtheorem{proposition}[theorem]{Proposition}

\newtheorem{df}{Definition}[section]
\newtheorem{cor}[theorem]{Corollary}

\newtheorem{rem}{\it Remark}

\newtheorem{claim}{\it {Claim}}[section]

\newenvironment{altproof}[1]
{\noindent%\addvspace{0.3cm}
	{\em Proof of {#1}}.}
{\nopagebreak\mbox{}\hfill $\Box$\par\addvspace{0.5cm}}
\numberwithin{equation}{section}

\def\i{\mathrm{i}}

\def\R{\mathbb{R}}
\def\dd{\operatorname{d}}

\newcommand{\la}{\langle}
\newcommand{\ra}{\rangle}

\newcommand{\NN}{\mathbb{N}}

\newcommand{\RR}{\mathbb{R}}
\newcommand{\cA}{\mathcal{A}}
\newcommand{\cB}{\mathcal{B}}
\newcommand{\cF}{\mathcal{F}}

\newcommand{\intsigma}{\mathring\Sigma}

\def\deg{\mathrm{deg}}

\def\B{\mathbb{B}}

\date{\today}
\begin{document}
	\title{{Blow-up solutions for  mean field equations with Neumann boundary conditions  on Riemann surfaces  }}
	\author{Mohameden Ahmedou   \and Thomas Bartsch \footnote{Supported by DFG grant BA 1009/19-1.} \and Zhengni Hu \footnote{Supported by CSC No. 202106010046.}}
	
	\date{4 Jan. 2025}
	\maketitle
	%\begin{center}
	%{\sl }
	%\end{center}
	
	%\begin{abstract}
	
	%\end{abstract}
	
	%{\bf MSC 2010:} Primary: ; Secondary: 
	
	%{\bf Key words:} 
	
	%%%
	\noindent{\bf Abstract:} 
	On a compact Riemann surface $(\Sigma, g)$ with a smooth boundary $\partial \Sigma$ and the interior  $\intsigma=\Sigma\setminus\partial\Sigma$, we consider the following mean field equations with Neumann boundary conditions:
	\begin{equation*}
		\left\{\begin{aligned}
			-\Delta_g u &= \lambda \left(\frac{Ve^u}{\int_{\Sigma} Ve^u \, dv_g} - \frac{1}{|\Sigma|_g}\right) & & \text{in } \intsigma\\
			\partial_{\nu_g} u &= 0 & & \text{on } \partial \Sigma
		\end{aligned}\right.,
	\end{equation*}
	We find conditions on potential functions $V$ such that solutions exist for the parameter $\lambda$ when it is in a small right (or left) neighborhood of a critical value $4\pi(m+k)$ for  integers $0\leq k \leq m \in \NN_+$ and blow up as $\lambda$ approaches the critical value parameter. Moreover, the blow-ups  occur exactly at $k$ points in the interior of $\Sigma$ and $(m-k)$ points on the boundary $\partial \Sigma$.\\
	%%%%%%%%%%%%%%%%
	
	\newpage
	%constent
	\tableofcontents
	\newpage
	% main parts of the article 
	
	%	\cleardoublepage
	%\include{Hu_intro}
	\section{Introduction}\label{sec:introduction}
	Let $(\Sigma, g)$ be a compact  Riemann surface with smooth boundary $\partial \Sigma$. We are interested in the existence of blow-up solutions of the following mean field equations with the Neumann boundary conditions:
	\begin{equation}~\label{eq:main_eq}
		\left\{\begin{aligned}
			-\Delta_g u &=\lambda\left(\frac{Ve^u}{\int_{\Sigma}  Ve^u d v_g}-\frac{1}{|\Sigma|_g}\right) & & \text { in } \intsigma,\\
			\partial_{ \nu_g } u&=0 & & \text { on } \partial \Sigma.
		\end{aligned}\right.
	\end{equation} 
	Here $\intsigma=\Sigma\setminus\partial\Sigma$ denotes the interior of $\Sigma$, $\Delta_g$ is the Laplace-Beltrami operator, $dv_g$ is the volume element in $(\Sigma, g)$, $|\Sigma|_g=\int_{\Sigma} dv_g$, the potential function $V:\Sigma\to\RR$ is smooth and positive, and $\nu_g$ is the unit outward normal of $\partial \Sigma$. 
	The mean field equation~\eqref{eq:main_eq} arises in many problems:  the Kazdan-Warner problem~\cite{kazdan1974curvature,Ding1997TheDE,nolasco1998sharp}; the prescribed Gauss curvature problem~\cite{chang1987prescribing,chen1987scalar,Chang1993,chang1988conformal,ching1993nirenberg,moser1973nonlinear}; the Chern-Simons-Higgs gauge theory \cite{nolasco1999double,tarantello1996multiple,ding1999multiplicity,ding2001self,caffarelli1995vortex}; statistical mechanics \cite{caglioti1992special,caglioti1995special2,li1999existence,chanillo1994rotational,kiessling1993statistical};
	and the Keller-Segel system for chemotaxis collapse \cite{Senba2000some,Wang2002SteadySS,keller1970,childress1984chemotactic,Battaglia2018}.

	Due to the blow-up analysis in~\cite{lisunyang2023} for Riemann surfaces with boundary, the compactness and concentration phenomenon is observed for~\eqref{eq:main_eq} alternatively. More precisely, we have the following lemma:
	\begin{lemma}[Lemma 3.1 of \cite{lisunyang2023}]\label{thm:cp}
		Assume that $\lambda>0$ and $(u_n)$ be a sequence of solutions of~\eqref{eq:main_eq} with $\lambda_n\rightarrow \lambda$. Up to a subsequence, one of the following alternatives holds:
		\begin{itemize}
			\item 	[(i)] There exists a constant $C_\lambda>0$  such that  $\|u_n\|\leq C_{\lambda} \text{ for any } n;$ 
			\item [(ii)]  there exists a finite singular set $\mathcal{S}=\left\{p_1, \cdots, p_m\right\} \subset \Sigma$ such that for any $1 \leqslant j \leqslant m$, there is a sequence of points $\left\{p_{j, n}\right\} \subset \Sigma$ satisfying $p_{j, n} \rightarrow p_j, u_n\left(p_{j, n}\right) \rightarrow+\infty$. Moreover, if $\mathcal{S}$ has $k$ points in $\Sigma$ and $(m-k)$ points on $\partial \Sigma$, then
			$
			\lambda_n  \rightarrow 4(m+k) \pi.
			$
		\end{itemize}
	\end{lemma} 
	If $(ii)$ of Lemma~\ref{thm:cp} holds, we say $\{u_n\}$ is a sequence of blow-up solutions of~\eqref{eq:main_eq} with parameter $\lambda$. The blow-up phenomenon for~\eqref{eq:main_eq} can only appear when the parameter $\lambda\in 4\pi \NN_+$, which is the so-called critical value set. The existence of solutions for $\lambda\notin 4\pi \NN_+$  has been fully addressed in~\cite{lisunyang2023} for general Riemann surfaces with smooth boundary via the min-max scheme introduced by Djadli and Malchiodi in \cite{djadli2008existenceq_curvature}. 
	When $\lambda$ lies in the critical value set, it is more complicated to understand the solutions of mean field equations.
	In~\cite{ahmedou_resonant_2017}, for  Dirichlet boundary conditions on bounded domains of $\RR^2$ the authors establish sufficient conditions for the existence of solutions at the critical value $\lambda$ using the approach of ``critical points at infinity." They compute the topological degree $\dd_\lambda$ under the assumption that a quantity related to the Green's function does not vanish, thereby identifying conditions that ensure $\dd_\lambda \neq 0$. Subsequently, \cite{Ahmedou_Ben_Morse2023} extends their results on Riemann surfaces with boundary. On the other hand, the study of blow-up solutions of mean field equations  has garnered significant interest, too. 
	The blow-up solutions are constructed in  \cite{brezis_uniform_1991, del_pino_singular_2005, pino_collapsing_2006, Esposito2005, ma_convergence_2001, Nagasaki1990AsymptoticAF, suzuki_two_1992} for the mean field equation on domains, and in  \cite{Bartolucci2020, Esposito2014singular, figueroa2022bubbling} for the mean field equations on a compact surface without boundary, and the references therein. 
	
	This paper aims to construct solutions of the mean field equations~\eqref{eq:main_eq} with Neumann boundary conditions for $\lambda$ close to one of the critical parameter values $\lambda_{k,m}:= 4\pi(m+k) $ for $ 0\leq  k\leq m\in \mathbb{N}_+$, which blow up at $k$ points in the interior and $(m-k)$ points on the boundary as $\lambda$ approaches $\lambda_{k,m}$.

	As in \cite{Esposito2005,Esposito2014singular}, our approach to finding blow-up solutions of \eqref{eq:main_eq} is based on variational methods combined with a Lyapunov-Schmidt reduction, which is the so-called ``localized energy method'' (see \cite{pino_collapsing_2006}). The main difficulty in the Neumann boundary conditions scenario is that the blow-up point may appear on the boundary. The estimates for the interior case generally cannot be applied directly to the boundary case. The paper modifies the approximation solutions presented in \cite{Esposito2014singular} for Riemann surfaces and applies  Lyapunov-Schmidt reduction and the variational method to obtain a sufficient condition for the existence of blow-up solutions. Furthermore, depending on the sign of quantities related to the potential function, Gaussian curvature of the surface and geodesic curvature of the boundary, we can construct bubbling solutions for $\lambda$ in a small left (or right) neighborhood of $\lambda_{k,m}$. 
	
	We begin by defining ``stable'' critical points  to state the main results. Specifically, we adopt the following definition of stability for critical points (see~\cite{del_pino_singular_2005,Esposito2005,Li1997OnAS}):
	\begin{df}  Let $D \subset \intsigma^k \times (\partial\Sigma)^{m-k}\setminus \Delta$, where $0\leq k\leq m\in \mathbb{N}_+$ and
		\begin{equation}
			\label{def:thick_d} \Delta:=\{\xi=(\xi_1,\cdots,\xi_m)\in \Sigma^m: \xi_i=\xi_j \text{ for some } i\neq j\}.
		\end{equation} 
		{ And $F: D \rightarrow \mathbb{R}$ be a $C^{1}$-function and  $x^*$ be  {a critical point} of $F$.
			We say $x^*$ is $C^1$-stable if for any closed neighborhood $U$ of $x$ in $ \intsigma^k \times(\partial\Sigma)^{m-k}\setminus \Delta$, there exists $\varepsilon>0$
			such that, if $G: D \rightarrow \mathbb{R}$ is a $C^{1}$-function with $\|F-G\|_{C^1(U)}<\varepsilon$, then G has at least one critical point in $U$.}
	\end{df}
	\begin{rem}
		If one of the following conditions holds:
		\begin{itemize}
			\item [i)] $x^*$ is a non-degenerate critical point of $F$,
			\item [ii)] $x^*$ is a strict local maximizer or minimizer,
			\item [iii)] for any $\varepsilon>0$ the Brouwer degree   $\deg(\nabla F, B_{\varepsilon}(x^*), 0)\neq 0$, where $B_{\varepsilon}(x^*):=\{ x: \text{dist}(x,x^*)<\varepsilon\}$,
		\end{itemize}
		then, $x^*$ is a $C^1$-stable critical point of $F.$
	\end{rem}
	Let $K_g$ be the Gaussian curvature of $\Sigma$ and $k_g$ be the geodesic curvature of the boundary $\partial\Sigma.$
	
	For the one-point blow-up case, we establish the following theorem concerning interior blow-ups:
	\begin{theorem}
		\label{thm:1_point_in}
		Let $\xi$ be a $C^1$-stable critical point of the Robin's function $R^g(\xi)$ (defined in Section~\ref{sec_green}) on $\intsigma$. If 
		$\Delta_g \log V(\xi^*) -2 K_g (\xi^*)+\frac{8\pi}{|\Sigma|_g}>0 \quad (<0 \text{ \it resp.})$ where $K_g$ is the Gaussian curvature of $(\Sigma, g)$
		then for all $\lambda$ in a small right (left \text{ \it resp.}) neighborhood of $8\pi$, there is a  solution $u_{\lambda}$ of
		\eqref{eq:main_eq}.
		As $\lambda\rightarrow 8\pi^+$ ($\lambda\rightarrow 8\pi^-$ resp.), 
		$u_\lambda$ blow up precisely at  $\xi^*$ with 
		\begin{equation*}
			\frac{\lambda V e^{u_{\lambda}}} { \int_{\Sigma} V e^{u_{\lambda}  dv_g}} \rightarrow 8\pi\delta_{\xi^*}, 
		\end{equation*}
		which is  convergent as measures in $\Sigma$ and $\delta_{\xi}$ is the Dirac mass on $\Sigma$ concentrated at point $\xi$. 
	\end{theorem}
	We further consider boundary blow-up scenarios with a single blow-up point, as in the following theorem:
	\begin{theorem}
		\label{thm:1_point_bo}
		Let $\xi^*$ be a $C^1$-stable critical point of the Robin's function $R^g(\xi)$ (defined in Section~\ref{sec_green}) on $\partial\Sigma$. If  in a small closed neighborhood of $\xi^*$: 
	$$  \partial_{\nu_g}\log V +2k_g= 0, \quad \Delta_g \log V -2 K_g +4k_g+\frac{4\pi}{|\Sigma|_g}>0 \quad ( <0\text{ \it resp.}),$$ 
then for all $\lambda$ in a small right (left \text{ \it resp.}) neighborhood of $4\pi$, there is a  solution $u_{\lambda}$ of
		\eqref{eq:main_eq}.
		As $\lambda\rightarrow 4\pi^+$ ($\lambda\rightarrow 4\pi^-$ resp.), 
		$u_\lambda$ blow up precisely at points $\xi^*$  (up to a subsequence), 
		\begin{equation*}
			\frac{\lambda V e^{u_{\lambda}}} { \int_{\Sigma} V e^{u_{\lambda}  dv_g}} \rightarrow 4\pi\delta_{\xi^*}, 
		\end{equation*}
		which is convergent as measures on  $\Sigma$.
	\end{theorem}
	
	For scenarios of multiple blow-ups, we have the following theorem:
	\begin{theorem}\label{cor:refine_mean_2}
		Let  $m\geq k\geq 0$  and $\partial_{\nu_g} \log V=-2k_g$ on $\partial\Sigma$ satisfying that 
		\[ m+k>\max\left\{\sup_{\Sigma} (-\Delta_g\log V+2K_g), \sup_{\partial \Sigma} (-\Delta_g\log V+2K_g+4k_g^2)\right\}, \]
		\[ \left( m+k< \min \left\{\inf_{\Sigma} (-\Delta_g\log V+2K_g), \inf_{\partial \Sigma} (-\Delta_g\log V+2K_g+4k^2_g)\right\}  \text{ resp.}\right).\]
		If there exists a $C^1$-stable critical point of $\cF_{k,m}$, denoted by $\xi^*= (\xi^*_1,\cdots,\xi^*_m)$, then 
		there exists a solution $u_{\lambda}$ of~\eqref{eq:main_eq} for $\lambda$ a small right (left resp.) neighborhood of $\lambda_{k,m}:=4\pi(m+k)$  blowing up at points $\xi^*_1,\cdots,\xi^*_m$.   Moreover, there exists a family of blow-up solutions $u_\lambda$ of~\eqref{eq:main_eq} with $\lambda\rightarrow \lambda_{k,m}$  such that 
		\[ \lambda\frac{V e^{u_\lambda}}{\int_\Sigma V e^{u_\lambda}}\rightarrow \sum_{i=1}^k 8\pi \delta_{\xi^*_i}+ \sum_{i=k+1}^m 4\pi \delta_{\xi^*_i}, \]
		which is convergent as measures in $\Sigma.$
	\end{theorem}
	The paper is organized as follows:\\ Section~\ref{sec:introduction} provides a concise overview of the background and the key conclusions of this study. Section~\ref{sec:2} details the main theorems and analyzes the conditions applied in different settings. Section~\ref{prelim} is dedicated to constructing approximate solutions using Liouville-type equations and providing their asymptotic estimates. Section~\ref{reduction} discusses applying the Lyapunov-Schmidt reduction method, which transforms our problem into a finite-dimensional one. The asymptotic behavior of the reduced energy functional is analyzed in Section~\ref{expansion}, highlighting its connection to the reduced function $\cF_{k,m}$. The proofs of our main theorems are presented in Section~\ref{proof_main_thm}. Finally, Section~\ref{appendix} includes technical estimates as an appendix.

	\section{Main results}\label{sec:2}
	Let $K_g$ be the Gaussian curvature of the Riemann surface $(\Sigma,g)$ and $k_g$ is the geodesic curvature of the boundary $\partial\Sigma$,  $\lambda_{k,m}=4\pi(m+k)$ and $\varrho(x)=
	8\pi \text{ if }\xi\in \intsigma \text{ and }
	4\pi  \text{ if }\xi\in \partial\Sigma$, through the whole paper.

	In this section, we discuss  definitions of the quantities $\cA_1$, $\cA_2$, and $\cB$, whose signs decide whether the solutions we constructed lie in the left or right side of the critical value $\lambda_{k,m}$ and illustrate the main theorems that we obtain. 
	
	Given any integers $0\leq k\leq m\in \NN_+$,  we define a scaling function
	\begin{equation*}
		\tau_i(x)=V(x)e^{\varrho(\xi_i)H^g(x,\xi_i)+\sum_{l\neq i}\varrho(\xi_l)G^g(x,\xi_l)},
	\end{equation*}
	where  $H^g(,\xi_i)$ is the regular part of the Green's function $G^g(\cdot,\xi_i)$ for $i=1,\cdots,m$ (see Section~\ref{sec_green}).

	For any $\xi=(\xi_1,\cdots,\xi_m)\in \intsigma^k\times(\partial\Sigma)^{m-k}\setminus\Delta$, we define 
	\begin{equation}\label{eq:A_xi_1}
		\mathcal{A}_1(\xi)=-\sum_{i=k+1}^m4\pi \sqrt{\tau_i} (\partial_{ \nu_g}\log V+ 2k_g)|_{\xi_i},
	\end{equation}
	\begin{equation}
		\label{eq:A_xi_2}	\mathcal{A}_2(\xi)=\sum_{i=1}^m 4\pi (\Delta_g\tau_i-2K_g\tau_i)|_{\xi_i}+ \sum_{i=k+1}^m 8\pi  k_g(\partial_{\nu_g
		}\log V +k_g)\tau_i|_{\xi_i},
	\end{equation}
	and 
	\begin{equation}
		\label{eq:B_xi}
		\begin{array}{lcl}
			&&\mathcal{B}(\xi)=-\sum_{i=1}^m 2\pi \log(\tau_i) (\Delta_g\tau_i-2K_g\tau_i)|_{\xi_i} - \sum_{i=k+1}^{m} 4\pi k_g(\partial_{\nu_g}\log V+k_g)\tau_i\log\tau_i |_{\xi_i}\\
			&&-\frac 1 2\mathcal{A}_2(\xi)
			+\lim_{r\rightarrow 0}\left( 8\int_{\Sigma\setminus \cup_{i=1}^mU_{r}(\xi_i)}Ve^{\sum_{i=1}^m \varrho(\xi_i)G^g(x,\xi_i)} dv_g- \right.\\	&&\left. \frac{1}{r^2}\left(\sum_{i=1}^m \varrho(\xi_i)\tau_i|_{\xi_i}-8 \sum_{i=k+1}^m  (\partial_{\nu_g}\log V+2k_g )\tau_i |_{\xi_i} \right)-\mathcal{A}_2(\xi)\log \frac 1 r\right). 
		\end{array}
	\end{equation}  
	
	The reduced function
	$\cF_{k,m}:  \Xi_{k,m}:=\intsigma^{k}\times (\partial \Sigma)^{m-k}\setminus \Delta\rightarrow \mathbb{R}$ is defined as
	\begin{equation*}
		{ \cF_{k,m}(\xi_1,\cdots,\xi_m)=\sum_{i=1}^m\varrho^2(\xi_i) R^g(\xi_i)+\sum^m _{ 
				\begin{array}{l}
					i,j=1\\
					i\neq j
			\end{array} }  \varrho(\xi_i)\varrho(\xi_j) G^g(\xi_i,\xi_j)}+  \sum_{i=1}^m 2\varrho(\xi_i)\log V(\xi_i),
	\end{equation*}
	where $R^g$ is the Robin's function and $G^g$ is the Green's function (refer to  Section~\ref{sec_green}) and the thick diagonal $\Delta$ is defined by~\eqref{def:thick_d}.

	The following theorem shows the existence of a family of blow-up solutions to~\eqref{eq:main_eq}, under conditions involving the potential functions, the Gaussian curvature, and the geodesic curvature of the boundary.
	\begin{theorem}~\label{main_thm}
		Given integers $0\leq k\leq m\in \NN_+$. Let  $\xi^* \in   \intsigma^k \times (\partial \Sigma)^{m-k}\setminus \Delta$ be a $C^1$-stable critical
		point of $\cF_{k,m}$.  $\mathcal{A}_1,\mathcal{A}_2,\mathcal{B}$  are defined by~\eqref{eq:A_xi_1},~\eqref{eq:A_xi_2} and~\eqref{eq:B_xi}, respectively. If one of the following conditions is satisfied in a closed neighborhood of $\xi^*$: \begin{itemize}
			\item 
			[$\mathfrak{a}1$.] $\mathcal{A}_1=0$ and $\mathcal{A}_2> 0(< 0\text{ resp.})$ 
			\item[$\mathfrak{a}2$.]$\mathcal{A}_2=0$, $\mathcal{A}_1=0$ and $\mathcal{B}>0(<0\text{ resp.})$
		\end{itemize}
		then for all $\lambda$ in a small right (left \text{ \it resp.}) neighborhood of $\lambda_{k,m}$, there is a  solution $u_{\lambda}$ of
		\eqref{eq:main_eq}.
		As $\lambda\rightarrow \lambda_{k,m}^+$ ($\lambda\rightarrow \lambda_{k,m}^-$ resp.), 
		$u_\lambda$ blow up precisely at points $\xi_{1}, \cdots, \xi_{m}$, such that $\left(\xi_{1}, \cdots, \xi_{m}\right) \in K $ (up to a subsequence), 
		\begin{equation*}
			\frac{\lambda V e^{u_{\lambda}}} { \int_{\Sigma} V e^{u_{\lambda}  dv_g}} \rightarrow \sum_{ i=1}^m \varrho(\xi_i)\delta_{\xi_i}, 
		\end{equation*}
		which is  convergent as measures in $\Sigma$.
	\end{theorem}
	
	It is clear  that Theorem~\ref{thm:1_point_in} to Theorem~\ref{cor:refine_mean_2} are the special cases of Theorem~\ref{main_thm}. 
	\begin{rem}
		The quantities $\cA_2$ and $\cB$
		originate from  \cite{Esposito2014singular}, where they are used to construct blow-up solutions for singular mean-field equations on closed Riemann surfaces. In our paper, we largely adhere to the idea proposed in \cite{Esposito2014singular}, with additional estimates for cases of boundary blow-up. Specifically, the quantity $\cA_1$ appears to deal with the case  where blow-up points occur on the boundary. Notably, when $\Sigma$ is a closed surface, $\cA_1 \equiv 0$, and $\cA_2$ and $\cB$ simplify to the constructions given in~\cite{Esposito2014singular}. 
	\end{rem}
	The quantity $\cA_2$ can be re-written as follows: 
	\begin{eqnarray*}
		\cA_2(\xi):= \sum_{i=1}^m \frac{\varrho(\xi_i)}{2}\tau_i\left( \Delta_g \log \tau_i+|\nabla \log \tau_i|^2 -2K_g\right)|_{x=\xi_i}+ \sum_{i=k+1}^m 2\varrho(\xi_i) k_g( \partial_{\nu_g}\log V+k_g) \tau_i|_{\xi_i}.
	\end{eqnarray*}
	For $\xi=(\xi_1,\cdots, \xi_m)$ as a critical point of $\cF_{k,m}$, we observe that $\nabla\log \tau_i|_{x=\xi_i}=0$ for $i=1,\cdots,m$. It follows that 
	\begin{eqnarray*}
		\cA_2(\xi)&=& \sum_{i=1}^k \frac{\varrho(\xi_i)}{2}\tau_i\left. \left(\Delta_g\log V-2K_g+\frac{4\pi(k+m)}{|\Sigma|_g} \right)\right|_{x=\xi_i}\\
		&&+\sum_{i=k+1}^m \frac{\varrho(\xi_i)}{2}\tau_i\left. \left(\Delta_g\log V-2K_g+ 4k_g(\partial_{\nu_g}\log V+k_g)+\frac{4\pi(k+m)}{|\Sigma|_g} \right)\right|_{x=\xi_i} 
	\end{eqnarray*}
	for any $\xi$ is a critical point of $\cF_{k,m}.$
	This observation leads us to Theorem~\ref{cor:refine_mean_2}.

	More detailed estimates are imperative when $\cA_1\equiv \cA_2\equiv 0$ in a small neighborhood of stable critical point. Initially, $\cB(\xi)$ was introduced in \cite{Chang_chen_lin2003} to analyze mean field equations on bounded domains. This concept was further applied by \cite{Esposito2014singular}, where $\cB(\xi)$ was utilized to develop blow-up solutions in cases where $\cA_1(\xi) \text{ and }\cA_2(\xi)$ vanish on compact Riemann surfaces. As for the examples involving the unit sphere and flat torus, we refer to \cite[Section 1]{Esposito2014singular}.

	\section{Preliminaries }~\label{prelim}
	\subsection{Isothermal Coordinate}
	Riemann surfaces are locally conformal flat and one can find isothermal coordinates where the metric is conformal to the Euclidean metric (see~\cite{chern1955,Hartman1955,Vekua1955,bers1957riemann}, for instance).  In our paper, we construct a family of isothermal coordinates based on the isothermal coordinates applied in~\cite{Esposito2014singular} and\cite{yang2021125440}, which maps an open neighborhood in $\Sigma$ onto an open or half-disk in $\RR^2.$ Set
	$$
	\mathbb{B}_r(y^*)=\left\{y=\left(y_1, y_2\right) \in \mathbb{R}^2: (y_1-y^*_1)^2+(y_2-y^*_2)^2<r\right\}, \B_r:=\B_r(0)$$
	$$
	\mathbb{B}_r^{+}=\mathbb{B}_r\bigcap\left\{y=\left(y_1, y_2\right) \in \mathbb{R}^2: y_2\geq0\right\},  \text{ and  }$$
	$$
	\mathbb{R}^{2}_+=\left\{y=\left(y_1, y_2\right) \in \mathbb{R}^2: y_2 \geq 0\right\}.$$
	
	For any $\xi\in\intsigma$, there exists an isothermal coordinate system $\left(U(\xi), y_{\xi}\right)$ such that $y_{\xi}$  maps  $U(\xi)$ around $\xi$ in $\Sigma$ onto an open disk $B^{\xi}:= \B_{2r_\xi}$ and  transforms $g$ to $e^{\varphi_\xi}\langle\cdot,\cdot\rangle_{\RR^2}$ with  $y_{\xi}(\xi)=0,$  
	$\overline{U(\xi)}\subset \intsigma$. 
	
	For $\xi\in \partial\Sigma$ there exists an isothermal coordinate system $\left(U(\xi), y_{\xi}\right)$ around $\xi$ such that the image of $y_{\xi}$ is a half disk $\B_{2r_{\xi}}^{+}$,  $y_{\xi}\left(U(\xi)\cap \partial \Sigma\right)= {\B_{2r_{\xi}}^{+}} \cap \partial \RR^{2}_+$ and  transforming $g$ to $e^{\varphi_\xi}\langle\cdot,\cdot\rangle_{\RR^2}$.
	In this case, we take $B^\xi = {\B}_{2r_{\xi}}^{+}$. 
	For $\xi\in \Sigma$ and $0<r\le 2r_\xi$ we set
	\[
	B_r^\xi := B^\xi \cap \{ y\in\RR^2: |y|< r\}\quad \text{and}\quad U_{r}(\xi):=y_\xi^{-1}(B_{r}^{\xi}).
	\]
	Let 
	$K_g$ be the Gaussian curvature of $\Sigma$ and $k_g$ be the geodesic curvature of the boundary $ \partial\Sigma$. Then,  for $\xi\in \Sigma$
	\begin{equation}
		\label{eq:Gauss}
		-\Delta \varphi_\xi(y) = 2K_g\big(y^{-1}_\xi(y)\big) e^{\varphi_\xi(y)} \quad\text{for all } y\in B^\xi. 
	\end{equation}
	and for $\xi\in \partial\Sigma$,
	\begin{equation}\label{eq:b_restrict}
		\frac{\partial}{\partial y_2}  \varphi_{\xi}(y) =- 2k_g(y_{\xi}^{-1}(y)) e^{ \frac{\varphi_{\xi}(y)}{2}} \quad \text{ for all } y\in B^{\xi}\cap \{ y_2=0\}.
	\end{equation}

	Additionally, the isothermal coordinates reserve the Neumann boundary conditions in following sense:   for any $x \in$ $y_{\xi}^{-1}\left({{\B}_{2r_{\xi}}^{+}} \cap \partial \mathbb{R}^{2}_+\right)$, 
	\begin{equation}\label{eq:out_normal_derivatives}
		\left(y_{\xi}\right)_*(\nu_g(x))=\left. -\exp\left( -\frac{\varphi_{\xi}(y)}2\right) \frac {\partial} { \partial y_2 }\right|_{	y=y_{\xi}(x)}.
	\end{equation}

	For technical reasons, we expect that the isothermal coordinates depend smoothly 
	on the parameter $\xi$. Locally smooth dependence can be achieved by using appropriate conformal maps. 
	\begin{claim}\label{claim:0}
		For any fixed $\zeta\in \Sigma$, starting from its isothermal coordinate, we can construct a family of refined isothermal coordinates  $( y_{\xi}, U(\xi))$ in a small neighborhood of $\zeta$ in which both $y_{\zeta}$ and $\varphi_{\zeta}$ smoothly depend  on the parameter 
		$\xi$ in $U_{r_{\zeta}}(\zeta)$. Moreover, 
		the Riemann metric takes the form in $U(\xi)$: 
		\[ g= e^{\varphi_{\xi}(y)}  ( dy_1^2+ dy_2^2)\]
		with $\varphi_{\xi}(0)=0$ and $\nabla \varphi_{\xi}(0)=0$ when $\xi\in \intsigma$ and $\nabla \varphi_{\xi}(0)=(0, -2k_g(\xi))$ when $\xi\in \partial\Sigma.$
	\end{claim}
	\begin{proof}
		Let $ \mathfrak{R}_{\theta}=\begin{bmatrix}
			\cos\theta & \sin\theta\\
			-\sin\theta & \cos\theta
		\end{bmatrix}$ be the rotation matrix. 
		Define  $T_{\theta,b,c}(y)= b  \mathfrak{R}_{\theta}y + c ( y_1^2-y_2^2, 2 y_1 y_2)^T$, where $b\neq 0  $  and $c\in \R$ which will be chosen later.
		It follows that 
		\[  DT_{\theta,b,c}(y)= b \mathfrak{R}_{\theta}+  \begin{bmatrix}
			2cy_1  & -2 cy_2 \\
			2c y_2  & 2cy_1
		\end{bmatrix}.\]
		Since $ DT_{\theta,b,c}(0)=b\mathfrak{R}_{\theta}  $ is invertible,  $T_{\theta,b,c}$ is  locally conformal around $0$ with  $d(y^*_1)^2+ d(y^*_2)^2=  \exp\{  \log ( (b\cos\theta +2cy_1 )^2+(b\sin \theta -2c y_2)^2) \} ( dy^2_1+dy^2_2)$. Denote $f_{\theta,b, c}(y)= \log ( (b\cos\theta +2cy_1 )^2+(b\sin \theta -2c y_2)^2)$.
		
		{\it Case I: $\xi\in \partial\Sigma$.}
		Take $T_1:=T_{\theta_1,b_1,c_1},  y_{\xi}(x)= T_1( y_{\zeta}(x)- y_{\zeta}(\xi))$ with $  \theta_1=0, b_1=e^{\varphi_{\zeta}\circ y_{\zeta}(\xi)/2}, c_1= \frac{ e^{\frac 1 2 \varphi_{\zeta}(y_{\zeta}(\xi))} \partial_{y_1} \varphi_{\zeta}(y)|_{y=y_{\zeta}(\xi)}}{4}$ and $r_{\xi}>0$ sufficiently small such that $  T^{-1}_1( \B^+_{2r_{\xi}})\subset  y_{\zeta} (U(\zeta))$ and $T^{-1}_1$ is a diffeomorphism for $\B^+_{2r_{\xi}}$ onto its image. Then, the Riemann metric has the form in $ U(\xi):= y^{-1}_{\xi}( \B^+_{2r_{\xi}})$
		$$g= e^{\varphi_{\zeta}( y_{\zeta}(\xi)+ T^{-1}_1( y^*)) - f_{0,b_1,c_1}(T^{-1}_1( y^*))} d(y^*_1)^2+ d(y^*_2)^2. $$
		Let $\varphi_{\xi}(y^*)= \varphi_{\zeta}( y_{\zeta}(\xi)+ T^{-1}_1( y^*)) - f_{0,b_1,c_1}( T^{-1}_1( y^*))$.
		The conformal factor $\varphi_{\xi}$  has the following properties: 
		\[ \varphi_{\xi}(0)= 0  \text{ and } \partial_{y^*_1} \varphi_{\xi}(0)= 0. \]
		Since $(y_{\xi}, U(\xi))$ is conformal coordinate, \eqref{eq:b_restrict} implies that  $\partial_{y^*_2} \varphi_{\xi}(0)= -2k_g(\xi)$.
		
		{\it Case II. $\xi\in \intsigma$.}
		Take $T_2:= T_{\theta_2, b_2, c_2}$ with $\theta_2=\frac \pi 2, b_2=1 $ and $c_2= -\frac{e^{-\frac 1 2 \varphi_{\zeta}(y_{\zeta}(\xi))} \partial_{y_2} \varphi_{\zeta}(y)|_{y=y_\zeta(\xi)}}{4}$. Let 
		\[ y_{\xi}(x)= T_2\circ T_1 ( y_{\zeta}(x)-y_{\zeta}(\xi))\]
		and $r_{\xi}>0$  be sufficiently small such that 
		$U(\xi):= y_{\xi}^{-1}(\B_{2r_{\xi}}) \subset U(\zeta)$ with $U(\xi)\cap \partial\Sigma=\emptyset$ and $ y_{\xi}(x)$ is a diffeomorphism from $U(\xi)$ onto $\B_{2r_{\xi}}$. 
		For the Riemann metric  in $U(\xi)$ we have the form:
		\[ g =  e^{\varphi_{\zeta}( y_{\zeta}(\xi)+ T^{-1}_1\circ T^{-1}_2( y^*)) - f_{0,b_1,c_1}(T^{-1}_1\circ T^{-1}_2( y^*))- f_{\frac \pi 2, 1, c_2}(T^{-1}_2(y^*))} d(y^*_1)^2+ d(y^*_2)^2.\]
		Let $\varphi_{\xi}(y^*):= \varphi_{\zeta}( y_{\zeta}(\xi)+ T^{-1}_1\circ T^{-1}_2( y^*)) - f_{0,b_1,c_1}(T^{-1}_1\circ T^{-1}_2( y^*))- f_{\frac \pi 2, 1, c_2}(T^{-1}_2(y^*))$. The conformal factor $\varphi_{\xi}$  has the following properties: 
		\begin{eqnarray*}
			\varphi_{\xi}(0)= 0 \text{ and } \nabla \varphi_{\xi}(0)=0.
		\end{eqnarray*}
		From our construction, it is clear that for any $\xi\in U_{r_{\zeta}}(\zeta)$, $y_{\xi}$ and $\varphi_{\xi}$ smoothly depend on $\xi$.
	\end{proof}
	Unless  specified, we will use these refined isothermal coordinates from Claim~\ref{claim:0} throughout the paper.

	\subsection{Green's functions}\label{sec_green} 
	For any $\xi\in \Sigma$, we define the Green's function for~\eqref{eq:main_eq} by following equations:
	\begin{equation}~\label{eq:green}
		\begin{cases}
			-\Delta_g G^g(x,\xi)  =\delta_{\xi} -\frac{1}{|\Sigma|_g} &  x\in \intsigma\\
			\partial_{ \nu_g } G^g(x,\xi) =0 & x\in \partial \Sigma\\
			\int_{\Sigma} G^g(x,\xi) dv_g(x) =0
		\end{cases}. 
	\end{equation}
	\begin{rem}\label{rk:G^g}
		We give several important properties of Green's functions (see  \cite[Lemma 6]{yang2021125440}):
		\begin{itemize}
			\item[a.] there exists a unique Green function $G^g( \cdot,\xi) \in L^1(\Sigma)$ solves~\eqref{eq:green} 
			in the distributional sense;
			\item [b.] for any  distinct points $x, \xi \in {\Sigma}, G^g(x, \xi)=G^g(\xi, x)$;
			\item [c.] the representation formula holds, for any $h\in C^2(\Sigma)$, 
			\begin{equation}
				\label{eq:representation}
				h(x)- \frac 1 {|\Sigma|_g}\int_{\Sigma} h dv_g=- \int_{\Sigma}G^g(\cdot,x)\Delta_gh dv_g+ \int_{\partial \Sigma} G^g(\cdot,x) \partial_{\nu_g}h ds_g, 
			\end{equation}
			where $ds_g$ is the line element of the boundary $\partial\Sigma$
			\item [d.]for any  distinct points $x, \xi \in {\Sigma}$, there exists a constant $C>0$ such that $$
			|G^g(x, \xi)| \leq C\left(1+\left|\log \dd_g(x, \xi)\right|\right), \quad\left|\nabla_{g, \xi} G^g(x, \xi)\right| \leq C\dd^{-1}_g(x, \xi),
			$$
		\end{itemize}
		where $\dd_g(x, \xi)$ denotes the geodesic distance between $x$ and $\xi$.
	\end{rem}
	Define the Robin's function as follows: 
	\[ R^g(\zeta):=\lim_{x\to\zeta}\left(G^g(x,\zeta)+\frac4{\varrho(\zeta)}\log \dd_g(x,\zeta)\right). \]
	Observe that for $\zeta\in U(\xi)$, 
	$
	\lim_{x\to\zeta}\frac{d_g(x,\zeta)}{\big|y_\xi(x)-y_\xi(\zeta)\big|} = e^{\frac12\varphi_\xi\circ y_{\xi}(\zeta)}.
	$
	It follows  
	\begin{equation}\label{eq:robin}
		R^g(\zeta) = \lim_{x\to\zeta}\left(G^g(x,\zeta)+\frac4{\varrho(\zeta)}\log\big|y_\xi(x)-y_\xi(\zeta)\big|\right)
		+ \frac 2 {\varrho(\zeta)}\varphi_\xi\big(y_\xi(\zeta)\big).
	\end{equation}
	In particular, using the assumption  $\varphi_\xi\big(y_\xi(\xi)\big)=\varphi_\xi(0)=0$, we obtain that 
	$$
	R^g(\xi) = \lim_{x\to\xi}\left(G^g(x,\xi)+\frac4{\varrho(\xi)}\log\big|y_\xi(x)\big|\right) . 
	$$
	
	Next, we will construct the regular part $H^g(x,\xi)$ of the Green's function $G^g(x,\xi)$ such that $H^g(\xi,\xi)=R^g(\xi)$.
	Let $\chi$ be a radial cut-off function in  $C^{\infty}(\mathbb{R}, [0,1])$ such that
	\begin{equation*}
		\chi(s)=	\begin{cases}
			1 & |s|\leq 1\\
			0& |s|\geq 2 
		\end{cases}. 
	\end{equation*}
	Setting $\chi_{\xi}(x)=\chi(4|y_{\xi}(x)|/\bar{r}_{\xi})$, we  define 
	\begin{eqnarray*}\Gamma_{\xi}^g(x):=\Gamma^g(x,\xi)=- \frac{4}{\varrho(\xi)} \chi_{\xi}(x)\log{|y_{\xi}(x)|}= \left\{ \begin{array}{ll}	-\frac{1}{2\pi} \chi_{\xi}(x)\log{|y_{\xi}(x)|}& \text{ if }\xi\in \intsigma \\
			-{\frac{1}{\pi}}\chi_{\xi}(x)\log{|y_{\xi}(x)|}	& \text{ if } \xi\in\partial\Sigma \end{array}\right., 
	\end{eqnarray*}   
	where $\bar{r}_{\xi}< r_{\xi}$ will be chosen later. 
	Then, for $\xi\in\Sigma$ the function $H^g_{\xi}:=H^g(\cdot,\xi):\Sigma\to\R$ is defined to be the unique solution of the Neumann problem 
	\begin{equation}\label{eq:eqR}
		\quad\left\{\begin{aligned}
			\Delta_g H_{\xi}^g  =& \frac{4}{\varrho(\xi)}\left(\Delta_g \chi_{\xi}\right) \log |y_{\xi}|+\frac{8}{\varrho(\xi)}\left\langle\nabla \chi_{\xi}, \nabla \log |y_{\xi}|\right\rangle_g+\frac{1}{|\Sigma|_g} &\text { in } \intsigma\\ 
			{\partial_{\nu_g} H_{\xi}^g}=&\frac{4}{\varrho(\xi)}\left(\partial_{\nu_g} \chi_{\xi}\right) \log |y_{\xi}|+\frac{4}{\varrho(\xi)} \chi_{\xi} \partial_{\nu_g} \log |y_{\xi}| &\text { on } \partial \Sigma . \\
			\int_{\Sigma} H_{\xi}^g d v_g  =&\frac{4}{\varrho(\xi)} \int_{\Sigma} \chi_{\xi} \log |y_{\xi}| d v_g& 
		\end{aligned}\right.
	\end{equation}
	\begin{rem}
		For the regular part $H^g(x,\xi)$ we will list some important properties that we have to use (see Lemma~\ref{lem:re_green}). 
		\begin{itemize}
			\item [a.]	$
			G^g(x,\xi) = \Gamma_{\xi}^g(x)+ H^g(x,\xi)
			\text{ and } H^g(\xi,\xi)=R^g(\xi);$
			\item[b.] for any fixed $\xi\in\Sigma$ and $\alpha\in (0,1)$, 
			$H^g(x,\xi)$ is $C^{2,\alpha}$ in $\Sigma$ with respect to  $x$. Moreover, $ H^g(x,\xi)$ is uniformly bounded in $C^{2,\alpha}$ for any $\xi$ in any compact subsect of $\intsigma$ or $\partial\Sigma$.
		\end{itemize}
	\end{rem}
	\subsection{ The reduced function $\cF_{k,m}$} \label{introf}
	Denote \[\Delta=\{\xi:=(\xi_1,\cdots,\xi_m)\in  \intsigma^{k}\times (\partial \Sigma)^{m-k} : \xi_i=\xi_j \text{ for some } i\neq j\},\] 
	as the thick diagonal of $\intsigma^{k}\times (\partial \Sigma)^{m-k}.$
	Define a  Kirchhoff-Routh type function on $\Xi_{k,m}$ (see~\cite{lin1941motion,Bartsch2017TheMP,Ahmedou-Bartsch-Fiernkranz:2023,BartschHuSubmitted})
	\begin{equation*}
		\begin{array}{ccc}
			&  \cF_{k,m} : \intsigma^{k}\times (\partial \Sigma)^{m-k}\setminus \Delta \rightarrow \mathbb{R},&  \\
			&&\\
			&\cF_{k,m}(\xi)=\sum_{i=1}^m 2\varrho(\xi_i)\log V(\xi_i) 
			+\sum_{i=1}^m\varrho(\xi_i)^2  R^g(\xi_i) 
			+\sum^m _{
				\begin{array}{l}
					i,j=1\\
					i\neq j
			\end{array} }  \varrho(\xi_i)\varrho(\xi_j) G^g(\xi_i,\xi_j), & 
		\end{array}.
	\end{equation*} 
	
	Observe that for any $\alpha\in (0,1)$,  $G^g(x,\xi) \in C^{2,\alpha}(\Sigma\setminus \{ \xi \})$ and $H^g_{\xi}$  is $C^{2,\alpha}(\Sigma)$, too. Thus $ \cF_{k,m}$ is of  $C^{2,\alpha}$ class. In particular, $\cF_{k,m}\in C^1( \intsigma^{k}\times (\partial \Sigma)^{m-k}\setminus \Delta).$
	
	\subsection{Approximation of Solutions}\label{sec:approx}
	It is well known that 
	\[ u_{\tau,\eta}(y)=\log \frac{ 8\tau^2 }{(\tau^2 +|y-\eta|^2)^2}\] for $(\tau, \eta)\in (0,\infty)\times  \mathbb{R}^2 $ are all the solutions to 
	the Liouville-type equation, 
	\begin{equation}
		\left\{\begin{aligned}
			-\Delta u= e^u  \text{ in } \mathbb{R}^2,\\
			\int_{\mathbb{R}^2}  e^u <\infty. 
		\end{aligned}\right.
	\end{equation}
	For all $x\in U(\xi)$, denote that 
	\[U_{\tau,\xi}(x)= u_{\tau, 0}(y_{\xi}(x))=\log\frac{8\tau^2}{(\tau^2+|y_{\xi}(x)|^2)^2}, \]
	for any $\tau>0$. 
	We try to construct approximate solutions to \eqref{eq:main_eq} applying $U_{\tau,\xi}$ with proper choices of $\tau$ and $\xi$. And for any $f\in L^1(\Sigma)$, we denote $\overline{f}=\frac{1}{|\Sigma|_g}\int_{\Sigma} fdv_g$. 
	
	Next we will introduce a projection operator $P$ to project $U_{\tau,\xi}$ into space $\overline{\mathrm{H}}^1$. 
	$P U_{\tau,\xi}$ is defined as the solution of the following problem:
	\begin{equation}~\label{eq:projection}
		\left\{\begin{array}{ll}
			-\Delta_g P U_{\tau,\xi}(x)= -\chi_{\xi} \Delta_g U_{\tau,\xi} +                      \overline{\chi_{\xi} \Delta_gU_{\tau,\xi}}  , &  x\in \intsigma,\\
			\partial_{ \nu_g } P U_{\tau,\xi}=0, &x\in \partial \Sigma, \\
			\int_{\Sigma} P U_{\tau, \xi} dv_g=0, &
		\end{array}\right..
	\end{equation}
	The solution of~\eqref{eq:projection} exists uniquely in $\overline{\mathrm{H}}^1$  and $ PU_{\tau,\xi}\in C^{\infty}({\Sigma})$ by the elliptic standard estimate. So $PU_{\tau,\xi}$ is well-defined.
	
	Through this paper, we denote  $\i(\xi)$ to be $ 
	2 \text{ if }\xi\in\intsigma \text{ and }
	1  \text{ if }\xi \in \partial \Sigma$. 	Given any $k\leq m\in \mathbb{N}_+$ and 
	$\xi^*:=(\xi_1^*,\cdots, \xi_m^*)\in \intsigma^k\times (\partial\Sigma)^{m-k}\setminus \Delta$. Then, we define 	the configuration space
	\begin{equation}
		\label{M_delta}
		M_{\xi^*,\delta}=
		\left\{ \xi=(\xi_1,\cdots,\xi_m)\in  \intsigma^{k}\times (\partial \Sigma)^{m-k}\setminus \Delta: 
		\xi_i \in \overline{U_{\delta}(\xi_i^*)} \text{ for } i=1,\cdots, m
		\right\},\end{equation} for any $\delta>0$ sufficiently small. 
	We fix a $\delta\in (0, \frac 1 2 \min\left\{r_{\xi^*_i}: i=1,\cdots,m\right\})$. 
	Consequently, for any $\xi\in M_{\xi^*,\delta}$, we have refined isothermal coordinates $(y_{\xi_i}, U(\xi_i))$ around $\xi_i$ with a uniform radius $ r_{\xi_i}$ for any $i=1,\cdots,m$, denoted by $4r_0$. We assume that $r_0>0$ is sufficiently small such that $U_{2r_0}(\xi_i)\cap U_{2r_0}(\xi_j)=\emptyset$ if $i\neq j$ and $U_{4r_0}(\xi_i)\cap \partial\Sigma=\emptyset$ for $i=1,\cdots, k.$
	We try to construction  approximate solutions to~\eqref{eq:main_eq} with 
	\begin{equation}\label{xi}
		(\rho, \xi)\in (0,+\infty)\times M_{\xi^*,\delta}.
	\end{equation} 
	For technical reasons, we define a scaling function
	\begin{equation}~\label{taui}
		\tau_i(x)=V(x)e^{\varrho(\xi_i)H^g(x,\xi_i)+\sum_{l\neq i}\varrho(\xi_l)G^g(x,\xi_l)}.
	\end{equation}
	and $\rho_i>0$ with 
	\begin{eqnarray} \label{rhoi}
		\rho^2_i=\rho^2\tau_i(\xi_i),  \text{ for any } i=1,2,\cdots,m.
	\end{eqnarray}
	We denote $\tau_i=\tau_i(\xi_i), \chi_i=\chi(4|y_{\xi_i}|/r_{\xi_i})$, $\varphi_i= \varphi_{\xi_i}$, $U_i=U_{\rho_i,\xi_i}$, $W_i=PU_i$ and $W=\sum_{i=1}^m W_i$. To make $W$ be a good approximation to the solution of~\eqref{eq:main_eq}, we assume that  for same constant $C>0$
	\begin{equation}~\label{lamda1}
		|\lambda-\lambda_{k,m}|\leq C\rho^2 |\log \rho|.
	\end{equation}
	The solution $u$ can decompose into two parts: one part is $\sum_{i=1}^m PU_i$ on the manifold $\mathcal{M}$; another part is the remainder term $\phi^{\rho}_{\xi}$, i.e. 
	\[u= \sum_{i=1}^m PU_i+\phi^{\rho}_{\xi},  \]
	where  $ \phi^{\rho}_{\xi}$ goes to $0$ as $\rho\rightarrow 0$.

	Define for $\xi\in M_{\xi^*,\delta}$
	\begin{equation*}
		\mathcal{A}_1(\xi)=-\sum_{i=k+1}^m\varrho(\xi_i)\sqrt{\tau_i}( \partial_{ \nu_g}\log V+2k_g)|_{\xi_i}, 
	\end{equation*}
	\begin{equation*}
		\mathcal{A}_2(\xi)=\sum_{i=1}^m \frac{\varrho(\xi_i)}2(\Delta_g\tau_i-2K_g\tau_i)|_{\xi_i}+\sum_{i=k+1}^m  2\varrho(\xi_i) k_g( \partial_{\nu_g}\log V +k_g )\tau_i|_{\xi_i},
	\end{equation*}
	and 
	\begin{equation*}
		\begin{array}{lcl}
			&&\mathcal{B}(\xi)=-\sum_{i=1}^m \frac{\varrho(\xi_i)} 4\log(\tau_i) (\Delta_g\tau_i-2K_g\tau_i)|_{\xi_i} \\
			&&-\sum_{i=k+1}^m 2\varrho(\xi_i) k_g(  \partial_{\nu_g}\log V +k_g) \log(\tau_i)\tau_i|_{\xi_i}
			\\
			&&-\frac 1 2\mathcal{A}_2(\xi)+\lim_{r\rightarrow 0}\left( 8\int_{\Sigma\setminus \cup_{i=1}^mU_{r}(\xi_i)}Ve^{\sum_{i=1}^m \varrho(\xi_i)G^g(x,\xi_i)} dv_g \right.\\	&&\left. -\frac{1}{r^2}(\sum_{i=1}^m \varrho(\xi_i)\tau_i-8 \sum_{i=k+1}^m \tau_i (\partial_{\nu_g}\log V+2k_g) |_{\xi_i})-\mathcal{A}_2(\xi)\log \frac 1 r\right). 
		\end{array}
	\end{equation*}   
	We note that the definition of $H^g_{\xi_i}$ depends on the choice coordinate system, while $H^g(\xi_i,\xi_i)$ and $\Delta_g H^g_{\xi_i}|_{\xi_i}= \frac 1 {|\Sigma|_g}$ are invariant under changes of isothermal coordinates. Hence, the quantities $\cA_1,\cA_2$ and $\cB$ are independent with the choice of isothermal coordinates.

	Now we can expand the energy functional for approximate solutions $$ J_{\lambda}(W)= \frac 1 2 \int_{\Sigma} |\nabla W|^2_g dv_g - \lambda \log\left(\int_{\Sigma} V e^W dv_g \right).$$
	\begin{proposition}
		Assume that~\eqref{xi}-\eqref{lamda1}, there exists $\rho_0>0$ such that  we have the expansion 
		\begin{equation}\label{expansion_JW}
			\begin{array}{lcl}
				J_{\lambda}(W)
				&=&-\lambda_{k,m}-\lambda\log\left(\frac{\lambda_{k,m}}{8}\right)-\frac 1 2 \cF_{k,m}(\xi)+2(\lambda-\lambda_{k,m})\log \rho\\
				&& -\mathcal{A}_1(\xi)\rho+\mathcal{A}_2(\xi)\rho^2\log\rho-\mathcal{B}(\xi)\rho^2+ \frac{\mathcal{A}^2_1(\xi)\rho^2}{2\lambda_{k,m}}+o(\rho^2).
			\end{array}
		\end{equation}
		which is  $C((0,\rho_0)\times M_{\xi^*,\delta})$ for $(\rho,\xi)$ as $\rho\rightarrow 0$. 
	\end{proposition}
	\begin{proof} Considering that for all $l=1,2,\cdots,m$ we have 
		$\int_{\Sigma} W_l dv_g=0$.
		\begin{equation*}
			\begin{array}{lcl}
				\int_{\Sigma}|\nabla W|_g^2 dv_g &=&\int_{\Sigma} (-\Delta_g W) W dv_g =\sum_{i,l=1}^m \int_{\Sigma} \chi_ie^{-\varphi_i\circ y_{\xi_i}^{-1}} e^{U_i} W_l dv_g, 
			\end{array}
		\end{equation*}
		
		By the representation formula, for any $i,l=1,\cdots,m$
		\begin{equation}~\label{rep_PU}
			\begin{array}{lcl}
				PU_i(\xi_l)&=& \frac 1 {|\Sigma|_g} \int_{\Sigma} PU_i dv_g +\int_{\Sigma}G^g(\cdot,\xi_l)(-\Delta_g)PU_i dv_g + \int_{\partial\Sigma} G^g(\cdot,\xi_l) \partial_{\nu_g}PU_i ds_g\\
				&=& \int_{\Sigma} \chi_i e^{-\varphi_i\circ y_{\xi_i}^{-1}} e^{U_i} G^g(\cdot,\xi_l) dv_g,
			\end{array}
		\end{equation}
		where $ds_g$ is the line element of $\partial\Sigma.$
		By a direct calculation, we have 
		\begin{equation}\label{eq:int_log_i}
			\int_{B^{\xi_i}_{r_0}} \frac{8\rho_i^2\log\left( 1+\frac{\rho_i^2}{|y|^2}\right)}{(\rho_i^2+|y|^2)^2} dy =\varrho(\xi_i)+\mathcal{O}(\rho^4). 
		\end{equation}
		Lemma~\ref{lem0}, Corollary~\ref{cor0} and Lemma~\ref{lemb1} yield that 
		\begin{equation*} \begin{array}{lcl}
				&&	\int_{\Sigma}\chi_i e^{-\varphi_i\circ y_{\xi_i}^{-1}}e^{U_i} PU_i  dv_g \\
				&=& \int_{ \Sigma} \chi_ie^{-\varphi_i\circ y_{\xi_i}^{-1}} \frac{8\rho_i^2}{ (\rho_i^2+|y_{\xi_i}(x)|^2)^2 } \left(\chi_i\left(U_{i}(x)-\log \left(8 \rho_i^2\right)\right)\right. \\
				&&\left. +\varrho(\xi_i) H^g(x, \xi_i)+h_{\rho_i,\xi_i}-2\rho_i^2 F_{\xi_i}+\mathcal{O}(\rho^3|\log \rho|)\right) dv_g(x) \\
				&=&\int_{\Sigma} \chi_ie^{-\varphi_i\circ y_{\xi_i}^{-1}} \frac{8\rho_i^2}{ (\rho_i^2+|y_{\xi_i}|^2)^2 }   \left(\chi_i\log \frac { |y_{\xi_i}|^4}{(\rho_i^2 +|y_{\xi_i}|^2)^2 } +\varrho(\xi_i) G^g(x,\xi_i) +h_{\rho_i,\xi_i}\right.\\
				&&\left. -2\rho_i^2 F_{\xi}+\mathcal{O}(\rho^3|\log \rho|) \right) dv_g\\
				&\stackrel{\eqref{rep_PU}}{=}& 
				-2\int_{B^{\xi_i}_{2r_0}}\chi^2(|y|/r_0) \frac{8\rho_i^2\log\left( 1+\frac{\rho_i^2}{|y|^2}\right)}{(\rho_i^2+|y|^2)^2} dy + \varrho(\xi_i)PU_i(\xi_i)\\
				&&+\varrho(\xi_i)h_{\rho_i,\xi_i}-2\rho_i^2\varrho(\xi_i) F_{\xi_i}(\xi_i)+o(\rho^2)\\
				&\stackrel{\eqref{eq:int_log_i}}{=} & -2\varrho(\xi_i) + \varrho(\xi_i)PU_i(\xi_i)+ \varrho(\xi_i)h_{\rho_i,\xi_i}-2\rho_i^2\varrho(\xi_i) F_{\xi_i}(\xi_i)+o(\rho^2)\\
				&=& -2\varrho(\xi_i) -4\varrho(\xi_i)\log\rho_i+
				\varrho^2(\xi_i) H^g(\xi_i,\xi_i)+ 2\varrho(\xi_i)h_{\rho_i,\xi_i}\\
				&&-4\rho_i^2\varrho(\xi_i) F_{\xi_i}(\xi_i)+o(\rho^2).
		\end{array}\end{equation*}
		For any $i\neq l$, by Lemma~\ref{lem0}, Corollary~\ref{cor0} and  Lemma~\ref{lemb1}, we have 
		\begin{equation*} \begin{array}{lcl}
				&&\int_{\Sigma} \chi_i e^{-\varphi_i\circ y_{\xi_i}^{-1}} e^{U_i} PU_l dv_g\\
				&=& \int_{ \Sigma} \chi_ie^{-\varphi_i\circ y_{\xi_i}^{-1}} \frac{8\rho_i^2}{ (\rho_i^2+|y_{\xi_i}(x)|^2)^2 } \left(\chi_l\left(U_{l}(x)-\log \left(8 \rho_l^2\right)\right)\right. \\
				&&\left. +\varrho(\xi_l) H^g(x, \xi_l)+h_{\rho_l,\xi_l}-2\rho_l^2 F_{\xi_l}+\mathcal{O}(\rho^3|\log \rho|)\right) dv_g(x) \\
				&=&\int_{\Sigma} \chi_ie^{-\varphi_i\circ y_{\xi_i}^{-1}} \frac{8\rho_i^2}{ (\rho_i^2+|y_{\xi_i}|^2)^2 }   \left(\chi_l\log \frac { |y_{\xi_l}|^4}{(\rho_l^2 +|y_{\xi_l}|^2)^2 } +\varrho(\xi_l) G^g(x,\xi_l) +h_{\rho_l,\xi_l}\right.\\
				&&\left. -2\rho_l^2 F_{\xi_l}+\mathcal{O}(\rho^3|\log \rho|) \right) dv_g\\
				&\stackrel{\eqref{rep_PU}}{=} &\varrho(\xi_l) PU_i(\xi_l)+ \varrho(\xi_i)h_{\rho_l,\xi_l}-2\rho_l^2\varrho(\xi_i)F_{\xi_l}(\xi_i)+ o(\rho^2)
				\\
				&=&\varrho(\xi_i)\varrho(\xi_l) G^g(\xi_l,\xi_i) +\varrho(\xi_l)h_{\rho_i,\xi_i}+ \varrho(\xi_i)h_{\rho_l,\xi_l}-2\rho_i^2\varrho(\xi_l)F_{\xi_i}(\xi_l)\\
				&&-2\rho_l^2\varrho(\xi_i)F_{\xi_l}(\xi_i)+ o(\rho^2).
		\end{array}\end{equation*}
		Since $\log \tau_i= \log V(\xi_i)+ \varrho(\xi_i)H^g(\xi_i,\xi_i)+ \sum_{l\neq i}\varrho(\xi_l)G^g(\xi_i,\xi_l)$, we derive that 
		\begin{equation}\label{eq:int_nablaW}
			\begin{array}{lll}
				\int_{\Sigma} |\nabla W|_g^2 dv_g&=& -\lambda_{k,m}(2+4\log \rho) -\cF_{k,m}(\xi)+2\lambda_{k,m}c_{\rho,\xi}-4\rho^2 F_{\rho,\xi}(\xi)+o(\rho^2),
			\end{array}
		\end{equation}
		where $c_{\rho,\xi}=\sum_{l=1}^m h_{\rho_l,\xi_l} ,\text{ and } F_{\rho,\xi}(\xi)=\sum_{i,l=1}^m \tau_i\varrho(\xi_l)F_{\xi_i}(\xi_l)$.
		Lemma~\ref{lemb1}  implies that
		\begin{equation}
			%~\label{intvew}
			\begin{array}{lcl}
				&&\int_{U_{r_0}(\xi_i)} Ve^W dv_g
				\\
				&=& \int_{U_{r_0}(\xi_i)} V(x) \exp\left\{U_i-\log(8\rho_i^2)+\varrho(\xi_i)H^g(x,\xi_i)+ \sum_{l\neq i} \varrho (\xi_l) G^g(x,\xi_l)\right. \\
				&&\left. + \sum_{l=1}^m  h_{\rho_l,\xi_l}-2\sum_{l=1}^m\rho_l^2 F_{\xi_l}\right\}\left(1+\mathcal{O}(\rho^3|\log \rho|)\right)dv_g  \\
				&=& \frac 1 {8\rho^2\tau_i}\int_{U_{r_0}(\xi_i)} e^{U_i}\tau_i(x) e^{\sum_{l=1}^m ( h_{\rho_l,\xi_l}-2 \rho_l^2 F_{\xi_l})} (1+\mathcal{O}(\rho^3|\log \rho|)) dv_g \\
				&=&  \left(\frac 1 {8\rho^2\tau_i}\int_{\Sigma} \chi_i  e^{U_i}\tau_i(x) e^{ \sum_{l=1}^m ( h_{\rho_l,\xi_l}-2 \rho_l^2 F_{\xi_l})} dv_g - \int_{A_{2r_0}(\xi_i)}\chi_i\frac{\tau_i(x)}{|y_{\xi}|^4}dv_g \right) +o(1),
			\end{array}
		\end{equation}
		where $A_{2r_0}(\xi_i)=U_{2r_0}(\xi_i)\setminus U_{r_0}(\xi_i).$
		
		Applying  Corollary~\ref{cor0} for $f(x)= \tau_i(x)e^{\varphi_i\circ y_{\xi_i}^{-1}(x)} e^{ \sum_{l=1}^m ( h_{\rho_l,\xi_l}-2 \rho_l^2 F_{\xi_l}(x))} $, it follows that 
		\begin{equation*}
			\begin{array}{lll}
				&&	\int_{\Sigma} \chi_i  e^{U_i}\tau_i(x) e^{ \sum_{l=1}^m ( h_{\rho_l,\xi_l}-2 \rho_l^2 F_{\xi_l})} dv_g= \varrho(\xi_i)\tau_i e^{\sum_{l=1}^m (h_{\rho_l,\xi_l}-2 \rho_l^2 F_{\xi_l}(\xi_i))}+ 	\varrho(\xi_i)\rho_i \mathfrak{f}_2(\xi_i)\\
				&&+\left( 4\tau_i e^{\sum_{l=1}^m ( h_{\rho_l,\xi_l}-2 \rho_l^2 F_{\xi_l}(\xi_i))}\int_{\Sigma}\frac 1 {r_0} \chi^{\prime}(\frac {|y_{\xi_i}|}{r_0})e^{-\varphi_{\xi_i}\circ y_{\xi_i}^{-1}}\frac {1} {|y_{\xi_i}|^3}dv_g\right.\\
				&&\left. +  4\mathfrak{f}_2(\xi_i) \int_{\Sigma} \frac 1 {r_0} \chi^{\prime}(\frac {|y_{\xi_i}|}{r_0})e^{-\varphi_{\xi_i}\circ y_{\xi_i}^{-1}}\frac {(y_{\xi_i})_2 } {|y_{\xi_i}|^3}dv_g -\frac {\varrho(\xi_i)} 4\Delta_g f(\xi_i)(2\log \rho_i+1)\right. \\
				&& \left. -2\Delta_g f(\xi_i) \int_{\Sigma} \frac 1 {r_0} \chi^{\prime}(\frac {|y_{\xi_i}|}{r_0})e^{-\varphi_{\xi_i}\circ y_{\xi_i}^{-1}}\frac {\log|y_{\xi_i}| } {|y_{\xi_i}|}dv_g  + 8 \int_{\Sigma} \chi_{i} e^{-\varphi_{i}\circ y_{\xi_i}^{-1}} \frac{f-P_2(f)}{|y_{\xi_i}|^4} dv_g \right)\rho_i^2
				+o(\rho^2);
			\end{array}
		\end{equation*}
		Lemma~\ref{lemb1} deduces that 
		\begin{equation*}
			\int_{\Sigma\setminus \cup_{i=1}^m U_{r_0}(\xi_i)} Ve^Wdv_g = \int_{\Sigma\setminus \cup_{i=1}^m U_{r_0}(\xi_i)} Ve^{\sum_{i=1}^m \varrho(\xi_i)G^g(x,\xi_i)}dv_g +\mathcal{O}(\rho^2|\log \rho|).
		\end{equation*}
		Since for any $x\in U_{r_0}(\xi_i)\setminus\{\xi_i\}$ $ \frac{\tau_{i}(x)}{|y_{\xi_i}(x)|^4}= V(x) e^{\sum_{l=1}^m \varrho(\xi_l)G^g(x,\xi_l)}$, 
		\[ \int_{U_{r_0}(\xi_i)}\chi_i \frac{\tau_{i}(x)}{|y_{\xi_i}(x)|^4} dv_g =  \int_{U_{r_0}(\xi_i)}V(x) e^{\sum_{l=1}^m \varrho(\xi_l)G^g(x,\xi_l)} dv_g. \]
		We observe that 
		$$\mathfrak{f}_2(\xi_i)=\begin{cases}
			0& \text{ if } i\leq k\\
			-\tau_i e^{c_{\rho,\xi}} (\partial_{ \nu_g}\log V+2 k_g)|_ {\xi_i}+ \mathcal{O}(\rho^2|\log \rho|) & \text{ if } i>k,
		\end{cases},$$  $$-\Delta_g f(\xi_i)=
		\begin{cases}
			e^{c_{\rho,\xi}}(-\Delta_g\tau_i+2K_g\tau_i)|_{\xi_i} +\mathcal{O}(\rho^2) & \text{ if } i\leq k\\
			e^{c_{\rho,\xi}}(-\Delta_g\tau_i+2K_g\tau_i-4k_g\tau_i(\partial_{\nu_g}\log V - k_g))|_{\xi_i} +\mathcal{O}(\rho^2)	&\text{ if } i>k
		\end{cases} $$
		and 
		\[\sum_{i=1}^m\varrho(\xi_i) e^{-2\sum_{l=1}^m\rho_l^2F_{\xi_l}(\xi_i)} =\lambda_{k,m} -2\rho^2\sum_{i,l=1}^m \varrho(\xi_i)\tau_iF_{\xi_l}(\xi_i)+\mathcal{O}(\rho^4)=\lambda_{k,m}-2F_{\rho,\xi}(\xi)+\mathcal{O}(\rho^4).\]
		Then 
		\begin{equation}\label{eq:rho2_int_veW}
			\begin{array}{lcl}
				8\rho^2 e^{-c_{\rho,\xi}}\int_{\Sigma}Ve^W dv_g &=& \lambda_{k,m} +\mathcal{A}_1(\xi)\rho- \mathcal{A}_2(\xi)\rho^2\log \rho  \\
				&&+B_{\chi}\rho^2-2F_{\rho,\xi} +o(\rho^2)\end{array}
		\end{equation}
		where  $$\mathcal{A}_1(\xi)=-\sum_{i=k+1}^m\varrho(\xi_i)\sqrt{\tau_i} (\partial_{ \nu_g}\log V+2k_g)|_{\xi_i},$$
		$$ \mathcal{A}_2(\xi)=\sum_{i=1}^m \frac{\varrho(\xi_i)}2(\Delta_g\tau_i-2K_g\tau_i)|_{\xi_i}+ \sum_{i=k+1}^m 2\varrho(\xi_i) k_g( \partial_{\nu_g}\log V+ k_g^2)\tau_i|_{\xi_i}$$ 
		and 
		\begin{equation}\label{eq:B_chi1}
			\begin{array}{lcl}
				B_{\chi}&=&  \sum_{i=1}^m \varrho(\xi_i)\tau_i \int_0^{\infty} \frac1 {r_0}\chi'(\frac s{r_0}) 
				\frac{ds}{s^2}-8 \sum_{i=k+1}^m ( \partial_{\nu_g}\log V(\xi_i)+2k_g(\xi_i))\tau_i\int_0^{\infty} \frac1 {r_0}\chi'(\frac s{r_0})
				\frac{ds}{s^2} \\
				&&-\sum_{i=1}^m \frac{\varrho(\xi_i)} 4\log(\tau_i) (\Delta_g\tau_i-2K_g\tau_i)|_{\xi_i} -\sum_{i=k+1}^m \varrho(\xi_i)k_g(\partial_{ \nu_g }\log V +k_g) \log(\tau_i)\tau_i|_{\xi_i}\\
				&&-\mathcal{A}_2(\xi)(\frac 1 2 + \int_0^{\infty}\frac 1 {r_0} \chi'(\frac s{r_0})\log s ds )\\
				&&+  8  \int_{\Sigma}\left( Ve^{\sum_{i=1}^m \varrho(\xi_i)G^g(x,\xi_i)}- \sum_{i=1}^m\chi_{i} e^{-\varphi_{i}\circ y_{\xi_i}^{-1}} \frac{P_2(\tau_i e^{\varphi_i\circ y_{\xi_i}^{-1}} )}{|y_{\xi_i}|^4}\right) dv_g. 
			\end{array}
		\end{equation}
		For any $r\leq r_0$, we divide the last integral in $B_{\chi}$ into two integrals on  $\Sigma\setminus \cup_{i=1}^mU_{r}(\xi_i)$ and $\cup_{i=1}^mU_{r}(\xi_i)$. 
		Using integrating by part, we have 
		\begin{equation*}
			\begin{array}{lcl}
				&&	\int_{\Sigma\setminus \cup_{i=1}^mU_{r}(\xi_i)}8\chi_{i} e^{-\varphi_{i}\circ y_{\xi_i}^{-1}} \frac{P_2(\tau_i e^{\varphi_i\circ y_{\xi_i}^{-1}} )}{|y_{\xi_i}|^4} dv_g =
				\sum_{i=1}^m8\tau_i \int_{B^{\xi}_{2r_0}\setminus B^{\xi}_{r}} \chi(|y|/r_0)\frac 1 {|y|^4} dy\\
				&&	+\sum_{i=k+1}^m 8(\frac{\partial}{\partial y_2} \tau_i\circ y_{\xi_i}^{-1}(0)+ \frac{\partial}{\partial y_2} \varphi_{\xi_i}(0) ) \int_{B^{\xi}_{2r_0}\setminus B^{\xi}_{r}} \chi(|y|/r_0)\frac {y_2} {|y|^4} dy\\
				&&+ 	\left(\sum_{i=1}^m4(\Delta_g\tau_i(\xi_i)-2K_g(\xi_i)\tau_i(\xi_i))+\sum_{i=k+1}^m  16k_g(\partial_{ \nu_g }\log V(\xi_i) +k_g(\xi_i)) \tau_i(\xi_i) \right)\\
				&&\cdot\int_{B^{\xi}_{2r_0}\setminus B^{\xi}_{r}} \chi(|y|/r_0)\frac 1 {|y|^2} dy\\
				&=&\left(\sum_{i=1}^m \varrho(\xi_i)\tau_i+8 \sum_{i=k+1}^m \tau_i (-\partial_{\nu_g}\log V(\xi_i)-2 k_g(\xi_i)) \right)\left(\int_0^{\infty}\frac1 {r_0}\chi'(\frac s{r_0})
				\frac{ds}{s^2} + \frac 1 {r^2}\right)\\
				&&-\mathcal{A}_2(\xi)\left( \int_0^{\infty}\frac 1 {r_0} \chi'(\frac s{r_0})\log s ds+\log r\right).
			\end{array}
		\end{equation*}
		The expression $B_{\chi}$ can be reformulated as follows: 
		\begin{equation}\label{eq:B_chi2}
			\begin{array}{lcl}
				B_{\chi}&=& -\sum_{i=1}^m \frac{\varrho(\xi_i)} 4\log(\tau_i) (\Delta_g\tau_i-2K_g\tau_i)|_{\xi_i}-\sum_{i=k+1}^m \varrho(\xi_i)k_g(\partial_{ \nu_g }\log V+k_g) \log(\tau_i)\tau_i|_{\xi_i}\\
				&&-\frac 1 2\mathcal{A}_2(\xi)+ 8\int_{\Sigma\setminus \cup_{i=1}^mU_{r}(\xi_i)}Ve^{\sum_{i=1}^m \varrho(\xi_i)G^g(x,\xi_i)} dv_g\\
				&&-\frac{1}{r^2} \left(\sum_{i=1}^m \varrho(\xi_i)\tau_i-8 \sum_{i=k+1}^m  (\partial_{\nu_g}\log V+2k_g)|_{\xi_i} \tau_i\right)\\
				&& 
				-\mathcal{A}_2(\xi)\log \frac 1 r+  8 \sum_{i=1}^m \int_{U_r(\xi_i)} e^{-\varphi_{i}\circ y_{\xi_i}^{-1}} \frac{\tau_i e^{\varphi_i\circ y_{\xi_i}^{-1}}-P_2(\tau_i e^{\varphi_i\circ y_{\xi_i}^{-1}} )}{|y_{\xi_i}|^4} dv_g. 
			\end{array}
		\end{equation}
		The equations~\eqref{eq:B_chi1} and  \eqref{eq:B_chi2} show that $B_{\chi}$ is independent with the choice of $\chi$ and $r$. 
		Since for some $\alpha\in(0,1)$ $e^{\varphi_i\circ y_{\xi_i}^{-1}}-P_2(\tau_i e^{\varphi_i\circ y_{\xi_i}^{-1}} )=o(|y_{\xi_i}(x)|^{2+\alpha})$ as $x\rightarrow \xi_i$, Lebesgue's dominated convergence theorem yields that  $$\lim_{r\rightarrow 0}\int_{U_r(\xi_i)} e^{-\varphi_{i}\circ y_{\xi_i}^{-1}} \frac{\tau_i e^{\varphi_i\circ y_{\xi_i}^{-1}}-P_2(\tau_i e^{\varphi_i\circ y_{\xi_i}^{-1}} )}{|y_{\xi_i}|^4} dv_g=0.$$
		Hence $B_{\chi}=\mathcal{B}(\xi)$, where $\cB(\xi)$ is defined by~\eqref{eq:B_xi}. 
		\begin{equation}\label{eq:int_veW}
			\begin{array}{lll}
				\lambda_{k,m} \log \int_{\Sigma}Ve^W dv_g &=& \lambda_{k,m}\log\left(\frac {\lambda_{k,m}} 8\right)-2\lambda_{k,m}\log\rho+\mathcal{A}_1(\xi)\rho-\mathcal{A}_2(\xi)\rho^2\log\rho+\mathcal{B}(\xi)\rho^2\\
				&&-2F_{\rho,\xi}+\lambda_{k,m}c_{\rho,\xi}-\frac{\mathcal{A}^2_1(\xi)\rho^2}{2\lambda_{k,m}}+o(\rho^2)
			\end{array}
		\end{equation}
		Applying~\eqref{eq:int_nablaW} and~\eqref{eq:int_veW}, we can derive that 
		\begin{equation*}
			\begin{array}{lcl}
				J_{\lambda_{k,m}}(W)&=&  \frac 1 2\int_{\Sigma} |\nabla W|^2_gdv_g -\lambda_{k,m}\log \int_{\Sigma} Ve^W dv_g\\
				& =& -\lambda_{k,m}(1+\log(\frac {\lambda_{k,m}} 8)) -\frac 1 2 \cF_{k,m}(\xi)-\mathcal{A}_1(\xi)\rho+\mathcal{A}_2(\xi)\rho^2\log\rho\\
				&&-\mathcal{B}(\xi)\rho^2+ \frac{\mathcal{A}^2_1(\xi)\rho^2}{2\lambda_{k,m}}+o(\rho^2).
			\end{array}
		\end{equation*}
		And 
		\begin{equation*}
			\begin{array}{lcl}
				J_{\lambda}(W)&=& J_{\lambda_{k,m}}(W)+(\lambda_{k,m}-\lambda)\log \int_{\Sigma}Ve^W dv_g\\
				&=& J_{\lambda_{k,m}}(W)+(\lambda_{k,m}-\lambda)(\log\left(\frac {\lambda_{k,m}} 8\right)-2\log\rho)+\mathcal{O}(\rho^3|\log\rho|)\\
				&=& -\lambda_{k,m}-\lambda\log\left(\frac{\lambda_{k,m}}{8}\right)-\frac 1 2 \cF_{k,m}(\xi)+2(\lambda-\lambda_{k,m})\log \rho\\
				&& -\mathcal{A}_1(\xi)\rho+\mathcal{A}_2(\xi)\rho^2\log\rho-\mathcal{B}(\xi)\rho^2+ \frac{\mathcal{A}^2_1(\xi)\rho^2}{2\lambda_{k,m}}+o(\rho^2).
			\end{array}
		\end{equation*} 
	\end{proof}
	
	\begin{proposition}
		\label{prop:C_1_expansion}
		Assume that~\eqref{xi}-\eqref{lamda1}, there exists $\rho_0>0$ such that  we have the expansion for any $i=1,\cdots,m$ and $j=1,\cdots,\i(\xi_i)$, where $\i(\xi)\text{ equals                                                                  }
		2 \text{ if }\xi\in\intsigma\text{ and }
		1 \text{ if }\xi \in \partial \Sigma$.
		\begin{equation}\label{expansion_JW_C_1}
			\partial_{(\xi_i)_j}J_{\lambda}(W)= -\frac 1 2 \partial_{(\xi_i)_j}\cF_{k,m}(\xi)+\mathcal{O}(\rho)
		\end{equation}
		which is  $C((0,\rho_0)\times M_{\xi^*,\delta})$ for $(\rho,\xi)$ as $\rho\rightarrow 0$.
	\end{proposition}
	\begin{proof}
		
		The results of Lemma~\ref{lem:expansion1st} and Lemma~\ref{expansion_rho}  imply that for any $i,l=1,\cdots,m$ and $j=1,\cdots, \i(\xi_i)$
		\begin{equation}~\label{1w}
			\begin{array}{ll}
				\partial_{\left(\xi_i\right)_j} W_l(x)= & \delta_{il}\left(  -2\chi_i \frac{  \partial_{\left(\xi_i\right)_j} |y_{\xi_i}(x)|^2  }{\rho_i^2+\left|y_{\xi_i}(x)\right|^2} + \varrho(\xi_i) \partial_{\left(\xi_i\right)_j} H^g\left(x, \xi_i\right)-4\log|y_{\xi_i}|\partial_{(\xi_i)_j}\chi_i\right)\\
				&-2\chi_l \frac{ \rho_l^2 \partial_{\left(\xi_i\right)_j}\log \tau_l\left(\xi_l\right)}{\rho_l^2+\left|y_{\xi_l}(x)\right|^2} +\mathcal{O}(\rho^2|\log\rho|),
			\end{array}
		\end{equation}
		in $C(\Sigma)$. For any $l\neq i$, 
		\[ \partial_{\left(\xi_{i}\right)_j} W_l(x)=-2\varrho(\xi_i) \frac{\chi_l \rho_l^2}{\rho_l^2+\left|y_{\xi_l}(x)\right|^2} \partial_{\left(\xi_i\right)_j} G\left(\xi_l, \xi_i\right)+\mathcal{O}(\rho^2|\log\rho|), 
		\]
		in $C(\Sigma)$ and
		$$
		\partial_{\left(\xi_i\right)_j} W_i(x)=\varrho(\xi_i) \partial_{\left(\xi_i\right)_j} G\left(x, \xi_i\right)+\mathcal{O}(\rho^2|\log\rho|), 
		$$
		in $C_{loc}(\Sigma \setminus\left\{\xi_i\right\})$. 
		Since $\int_{\Sigma} \partial_{(\xi_i)_j} W=0$, 
		\begin{equation*}
			\begin{array}{lcl}
				\partial_{(\xi_i)_j} J_{\lambda}(W)&=& \int_{\Sigma}\left(-\Delta_g W-\frac{\lambda Ve^W }{\int_{\Sigma} Ve^W dv_g } + \frac{\lambda}{|\Sigma|_g}\right) \partial_{(\xi_i)_j} W dv_g \\
				&=& \int_{\Sigma}\left(-\Delta_g W-\frac{\lambda Ve^W }{\int_{\Sigma} Ve^W dv_g } \right) \partial_{(\xi_i)_j} W dv_g,
			\end{array}
		\end{equation*} 
		for any $i=1,2,\cdots,m$ and $j=1,\cdots,\i(\xi_i).$  
		Applying $\int_{\Sigma} \partial_{(\xi_i)_j} W=0$ again, 
		\begin{equation*}
			\begin{array}{lcl}
				\int_{\Sigma}\left(-\Delta_g\right) W\partial_{(\xi_i)_j} W dv_g &=& \sum_{l=1}^m \sum_{q=1}^m  \int_{\Sigma} \chi_l e^{-\varphi_l\circ y_{\xi_l}} e^{U_l} \partial_{(\xi_i)_j}W_q dv_g. \\
			\end{array}
		\end{equation*} 
		By Corollary~\ref{cor0} and~\eqref{1w}, we have 
		\begin{equation*}
			\begin{array}{lcl}
				&&\int_{\Sigma} \sum_{l,q\neq i}\chi_l e^{-\varphi_l\circ y_{\xi_l}} e^{U_l} \partial_{(\xi_i)_j}W_qdv_g\\
				&=& \sum_{l\neq i} \int_{\Sigma}  \chi_l e^{-\varphi_l\circ y_{\xi_l}} e^{U_l} \left( -2\varrho(\xi_i) \frac{\chi_l \rho_l^2 }{\rho_l^2 +|y_{\xi_l}(x)|^2 } \partial_{(\xi_i)_j}G^g(\xi_l,\xi_i)\right)   dv_g   + \mathcal{O}(\rho^2|\log\rho|)\\
				&=& -\sum_{l\neq i}\varrho(\xi_l)\varrho(\xi_i) \partial_{(\xi_i)_j}G^g(\xi_l,\xi_i)+\mathcal{O}(\rho^2|\log\rho|), 
			\end{array}
		\end{equation*}
		\begin{equation*}
			\int_{\Sigma}	\sum_{l\neq i} \sum_{q=i} \chi_l e^{-\varphi_l\circ y_{\xi_l}} e^{U_l} \partial_{(\xi_i)_j}W_q dv_g=  \sum_{l\neq i}\varrho(\xi_l) \varrho(\xi_i) \partial_{(\xi_i)_j}G^g(\xi_l,\xi_i)+\mathcal{O}(\rho^2|\log\rho|),
		\end{equation*}
		\begin{equation*}
			\int_{\Sigma}	\sum_{l= i} \sum_{q\neq i} \chi_l e^{-\varphi_l\circ y_{\xi_l}} e^{U_l} \partial_{(\xi_i)_j}W_q dv_g=\mathcal{O}(\rho^2|\log\rho|),
		\end{equation*}
		and 
		\begin{equation*} \begin{array}{lcl}
				&&  \int_{\Sigma}	\sum_{l= i} \sum_{q=i} \chi_l e^{-\varphi_l\circ y_{\xi_l}} e^{U_l} \partial_{(\xi_i)_j}W_q dv_g\\
				&=& \int_{\Sigma} \chi_i e^{-\varphi_i\circ y_{\xi_i}^{-1}} e^{U_i}\left( \partial_{\left(\xi_i\right)_j}\left(\chi_i( U_i-\log (8\rho_i^2)\right) + \varrho(\xi_i) \partial_{\left(\xi_i\right)_j} H^g(x,\xi_i) +\mathcal{O}(\rho^2|\log\rho|)\right)  dv_g,\\
				&=& \int_{\Sigma}\chi_i e^{U_i}e^{-\varphi_i\circ y_{\xi_i}^{-1}}(U_i-\log(8\rho_i^2)) \partial_{\left(\xi_i\right)_j} \chi_idv_g+ \int_{\Sigma}\chi^2_i e^{U_i} e^{-\varphi_i\circ y_{\xi_i}^{-1}} \partial_{(\xi_i)_j} U_idv_g\\&& -\int_{\Sigma}\chi^2_i e^{U_i} e^{-\varphi_i\circ y_{\xi_i}^{-1}}   \partial_{\left(\xi_i\right)_j}\log\tau_i dv_g  +\left.\varrho(\xi_i)^2 \partial_{\left(\xi_i\right)_j}H^g(x,\xi_i)\right|_{x=\xi_i}+\mathcal{O}(\rho^2|\log\rho|).\\
				&=& -\varrho(\xi_i)^2\partial_{(\xi_i)_j}H^g(\xi_i,\xi_i)+\left.\varrho(\xi_i)^2 \partial_{\left(\xi_i\right)_j}H^g(x,\xi_i)\right|_{x=\xi_i}-\sum_{l\neq i} \varrho(\xi_i)\varrho(\xi_l)\partial_{(\xi_i)_j}G^g(\xi_i,\xi_l)\\
				&&-\varrho(\xi_i)\partial_{(\xi_i)_j}\log V(\xi_i)+\mathcal{O}(\rho^2|\log\rho|),
		\end{array}\end{equation*}
		in view of $\partial_{\nu_g} \partial_{(\xi_i)_j} H^g_{\xi_i}(x)|_{x=\xi_i}=0$ for $i=k+1,\cdots,m$
		and $j=1.$		
		By straightforward calculation, we have 
		for $j=1,\cdots, \i(\xi_i)$
		\[ \partial_{(\xi_i)_j} R^g(\xi_i,\xi_i)=\partial_{(\xi_i)_j} H^g(\xi_i,\xi_i)=2 \partial_{(\xi_i)_j} H^g(x,\xi_i)|_{x=\xi_i}. \]
		Thus 
		\[	\int_{\Sigma}\left(-\Delta_g W\right) \partial_{(\xi_i)_j} W dv_g = - \frac 1 2\partial_{(\xi_i)_j}\cF_{k,m}(\xi)+\mathcal{O}(\rho^2|\log\rho|). \]
		We observe that  
		\begin{equation}~\label{extau}
			\tau_i(x)=\tau_i(\xi_i)+\la \nabla(\tau_i\circ y_{\xi_i}^{-1})(0), y_{\xi_i}(x)\ra+ \mathcal{O}(|y_{\xi_i}(x)|^2), \text{ for any }x\in U_{2r_0}(\xi_i). 
		\end{equation}
		Lemma~\ref{lemb1} implies  $Ve^W=\begin{cases}
			\frac{\tau_i(x)}{8\tau_i\rho^2}e^{U_i}(1+\mathcal{O}(\rho^2|\log \rho|)), & x\in U_{r_0}(\xi_i)\\
			\mathcal{O}(1), & x\in \Sigma\setminus \bigcup_{i=1}^m U_{r_0}(\xi_i)
		\end{cases}$. 
		\begin{equation*} \begin{array}{lcl}
				\int_{\Sigma}{Ve^W }\partial_{(\xi_i)_j} W dv_g&=&\frac 1 {8\rho_i^2}\left(   \sum_{l=1}^m \int_{U_{r_0}(\xi_l)}\chi_l \tau_l(x)e^{U_l}(1+\mathcal{O}(\rho^2|\log\rho|) )\partial_{(\xi_i)_j} W dv_g+  \mathcal{O}(\rho^2)\right). 
		\end{array}\end{equation*}
		By Corollary~\ref{cor0} and  \eqref{1w}, we have 
		\begin{equation*}
			\begin{array}{lcl}
			\sum_{l\neq i}	\int_{U_{r_0}(\xi_l)} \chi_l \tau_l(x)e^{U_l} \partial_{(\xi_i)_j} W_i &=& \sum_{l\neq i} \int_{U_{r_0}(\xi_l)} \chi_l \tau_l(x) e^{U_l} \partial_{(\xi_i)_j} W_i dv_g +\mathcal{O}(\rho^2|\log\rho|)\\
				&=&  \sum_{l\neq i} \tau_l(\xi_i) \varrho(\xi_i)\varrho(\xi_l)\partial_{(\xi_i)_j} G^g(\xi_l,\xi_i)+\mathcal{O}(\rho^2|\log\rho|),
			\end{array}
		\end{equation*}
		\begin{equation*}
			\begin{array}{lcl}
			\sum_{l\neq i}	\int_{U_{r_0}(\xi_i)} \chi_i \tau_i(x)e^{U_i} \partial_{(\xi_i)_j} W_l dv_g  &=&  -\sum_{l\neq i} \tau_l(\xi_l) \varrho(\xi_l)\varrho(\xi_i)\partial_{(\xi_i)_j} G^g(\xi_l,\xi_i)+\mathcal{O}(\rho^2|\log\rho|),
			\end{array}
		\end{equation*}
		and 
		\begin{equation*}
			\begin{array}{lcl}
				&&\int_{U_{r_0}(\xi_i)} \chi_i \tau_i(x)e^{U_i} \partial_{(\xi_i)_j} W_i dv_g \\
				&=&\int_{\Sigma} \chi_i \tau_i(x)e^{U_i} \left( \partial_{\left(\xi_i\right)_j}\left(\chi_i( U_i-\log (8\rho_i^2)\right) + \varrho(\xi_i) \partial_{\left(\xi_i\right)_j} H^g_{\xi_i} +\mathcal{O}(\rho^2|\log\rho|)\right) dv_g +\mathcal{O}(\rho^2)\\
				&=& \partial_{(\xi_i)_j} \int_{\Sigma} \chi_i \tau_i(x) e^{U_i} dv_g - \int_{\Sigma} \chi_i e^{U_i}  \partial_{(\xi_i)_j}  \tau_i(x) dv_g\\
				&& + \int_{\Sigma} \chi_i\tau_i(x) e^{U_i} ( -                                                                                                                                                                                                                                                                                                                                                                                                   \partial_{(\xi_i)_j}  \log \tau_i+ \varrho(\xi_i) \partial_{(\xi_i)_j} H^g_{\xi_i}) dv_g+\mathcal{O}(\rho^2|\log\rho|)\\
				&=&\partial_{(\xi_i)_j} \int_{\Sigma} \chi_i \tau_i(x) e^{U_i} dv_g -\int_{\Sigma} \chi_i\tau_i(x) e^{U_i}                                                                                                                                                                                                                                                                                                                                                                                                    \partial_{(\xi_i)_j}  \log \tau_i dv_g +\mathcal{O}(\rho^2|\log\rho|),
			\end{array}
		\end{equation*}
		where we applied $\varrho(\xi_i)\partial_{(\xi_i)_j}H^g(x,\xi_i)= \partial_{(\xi_i)_j}\log \tau_i(x)$. 
		Then via Corollary~\ref{cor0}, we deduce  
		\begin{equation*}
			\begin{array}{lcl}
				\int_{\Sigma} \chi_i\tau_i(x) e^{U_i}                                       \partial_{(\xi_i)_j}  \log \tau_i dv_g &=&\varrho(\xi_i)\tau_i                                     \partial_{(\xi_i)_j}  \log \tau_i  +\mathcal{O}(\rho), \\
				\partial_{(\xi_i)_j} \int_{\Sigma} \chi_i \tau_i(x) e^{U_i} dv_g
				&=& 
				\varrho(\xi_i)\partial_{(\xi_i)_j} \tau_i+\mathcal{O}(\rho).
			\end{array}
		\end{equation*}
		Thus $\int_{U_{r_0}(\xi_i)} \chi_i \tau_i(x)e^{U_i} \partial_{(\xi_i)_j} W_i dv_g  =\mathcal{O}(\rho).$
		Combining the estimates above, we have  
		\begin{equation}\label{eq:intvew_upper}
			\int_{\Sigma} Ve^W \partial_{(\xi_i)_j} W dv_g = \mathcal{O}(\rho^{-1}).
		\end{equation}
		By \eqref{eq:rho2_int_veW} and \eqref{eq:intvew_upper},
		\begin{equation}
			\label{eq:int_ve_pa_xi_normal}\int_{\Sigma} \frac{Ve^W}{\int_{\Sigma} Ve^W dv_g} \partial_{(\xi_i)_j} W dv_g =\frac{  \mathcal{O}(\rho^{-1})}{ \frac{\lambda_{k,m}+ \mathcal{O}(\rho)}{8\rho^2}}= \mathcal{O}(\rho).
		\end{equation} 
		Hence, we deduce that 
		$ \partial_{(\xi_i)_j}J_{\lambda}(W)= -\frac 1 2 \partial_{(\xi_i)_j}\cF_{k,m}(\xi)+\mathcal{O}(\rho). $
	\end{proof}
	\begin{proposition}
		Assume that~\eqref{xi}-\eqref{lamda1}, there exists $\rho_0>0$ such that we have the following expansions
		\begin{equation}\label{expansion_JW_rho1}
			\begin{array}{lcl}
				\partial_{\rho}	J_{\lambda}(W)	&=& \frac{2(\lambda-\lambda_{k,m})}{\rho}-\mathcal{A}_1(\xi)+\frac{\mathcal{A}^2_1(\xi)}{\lambda_{k,m}}\rho +\mathcal{A}_2(\xi)\rho+2\mathcal{A}_2(\xi)\rho\log\rho-2 \mathcal{B}(\xi)\rho+o(\rho).
			\end{array}
		\end{equation}
		which is  $C((0,\rho_0)\times M_{\xi^*,\delta})$ for $(\rho,\xi)$ as $\rho\rightarrow 0$. And 
		\begin{equation}\label{expansion_JW_rho2}
			\begin{array}{lcl}
				\partial^2_{\rho}	J_{\lambda}(W)	&=&- \frac{2(\lambda-\lambda_{k,m})}{\rho^2}+\frac{\mathcal{A}_1^2(\xi)}{\lambda_{k,m}}+3\mathcal{A}_2(\xi)+2\mathcal{A}_2\log\rho-2\mathcal{B}(\xi)+o(1),
			\end{array}
		\end{equation}
		which is  $C((0,\rho_0)\times M_{\xi^*,\delta})$ for $(\rho,\xi)$ as $\rho\rightarrow 0$. 
		\begin{equation}\label{expansion_JW_mix}
			\begin{array}{lcl}
				\partial_{(\xi_i)_j}\partial_{\rho}J_{\lambda}(W)=\mathcal{O}(\rho|\log\rho|),
			\end{array}
		\end{equation}
		which is  $C((0,\rho_0)\times M_{\xi^*,\delta})$ for $(\rho,\xi)$ as $\rho\rightarrow 0$.
	\end{proposition}
	\begin{proof}
		Applying  Lemma~\ref{expansion_rho}, for any $l=1,\cdots,m$
		\begin{equation}~\label{w1strho}
			\rho_l\partial_{\rho_l}W_l=\rho\sqrt{\tau_l} \partial_{\tau}PU_{\tau,\xi_l}|_{\tau=\rho_l}  = - \chi_l\frac{4 \rho_l^2 }{\rho_l^2+|y_{\xi_l}(x)|^2}-4\rho_l^2F_{\xi_l}+h^1_{\rho_l,\xi_l}+o(\rho^2)
		\end{equation}
		and 
		\begin{equation}~\label{w2strho}
			\rho_l^2	\partial^2_{\rho_l}W_l=\rho^2\tau_l	\partial^2_{\tau}PU_{\tau,\xi_l}|_{\tau=\rho_l}= \chi_l\frac{4\rho_l^2(\rho_l^2-|y_{\xi_l}|^2)}{(\rho_l^2+|y_{\xi_l}|^2)^2}-4\rho_l^2F_{\xi_l}+h^2_{\rho_l,\xi_l}+o(\rho^2).
		\end{equation}
		
		We calculate the first derivative of $J_{\lambda}(W)$ with respect to  $\rho$. 
		$
		\partial_{\rho}J_{\lambda_{k,m}}(W)= \int_{\Sigma}(-\Delta_g W) \partial_{\rho} W dv_g-\lambda_{k,m}\frac{\int_{\Sigma} Ve^W \partial_{\rho}W dv_g }{\int_{\Sigma}Ve^W dv_g}.
		$
		Since $\int_{\Sigma}\partial_{\rho}W=0$,   we have
		\begin{equation*}
			\begin{array}{lcl}
				&&	\int_{\Sigma}(-\Delta_g W) \partial_{\rho} W dv_g
				=\sum_{i=1}^m \int_{\Sigma} \chi_ie^{-\varphi_i\circ y_{\xi_i}^{-1}} e^{U_i} \partial_{\rho}W dv_g \\
				&=& \sum_{i,l=1}^m  \sqrt{\tau_l}\int_{\Sigma}\chi_ie^{-\varphi_i\circ y_{\xi_i}^{-1}} e^{U_i} \partial_{\rho_l}W_ldv_g\\
				&\stackrel{\eqref{w1strho}}{=}&\frac 1 {\rho} \sum_{i,l=1}^m  \int_{\Sigma}\chi_ie^{-\varphi_i\circ y_{\xi_i}^{-1}} e^{U_i} \left(- \chi_l\frac{4 \rho_l^2 }{\rho_l^2+|y_{\xi_l}(x)|^2}-4\rho_l^2F_{\xi_l}+h^1_{\rho_l,\xi_l}+o(\rho^2)\right)dv_g\\
				&\stackrel{\eqref{eqspan}\text{ and }\eqref{eqspan_1}}{=}& \frac 1 \rho\sum_{i,l=1}^m\left(- 2\varrho(\xi_i)+\varrho(\xi_i)h^1_{\rho_l,\xi_l}+o(\rho^2)\right)=-\frac{2\lambda_{k,m}}{\rho}	-\frac{4F_{\rho,\xi}} {\rho}+\frac {\lambda_{k,m}c^1_{\rho,\xi}}\rho+o(\rho),
			\end{array}
		\end{equation*}
		where $\delta_{il}$ is the Kronecker symbol and $c^1_{\rho,\xi}=\sum_{i=1}^m h^1_{\rho_i,\xi_i}.$
		By Lemma~\ref{lemb1}, we can deduce that 
		\begin{equation}
			\label{eq:veW_detail}
			V(x)e^{W(x)}= \sum_{i=1}^m \chi_i e^{c_{\rho,\xi}-2\sum_{l=1}^m\rho_l^2F_{\xi_l}(x)} \frac{\tau_i(x)}{8\tau_i\rho^2}e^{U_i}(1+\mathcal{O}(\rho^3|\log \rho|))+ \mathcal{O}(1)\mathbbm{1}_{\Sigma\setminus\cup_{i=1}^m U_{r_0}(\xi_i)}. 
		\end{equation}
		By Corollary~\ref{cor0}, \eqref{w1strho} and  \eqref{eq:veW_detail}, it follows the following estimate:
		\begin{equation*}
			\begin{array}{lcl}
				&&\int_{\Sigma}Ve^W \partial_{\rho}W dv_g= \sum_{i,l=1}^m \int_{U_{2r_0}(\xi_i)} Ve^W\partial_{\rho}W_l dv_g +\mathcal{O}(\rho|\log\rho|)  \\
				&=&\sum_{i,l=1}^m \frac{e^{c_{\rho,\xi}}}{8\tau_i\rho^2}\int_{U_{2r_0}(\xi_i)}\chi_i\tau_i(x)e^{-2\sum_{l=1}^m \rho_l^2F_{\xi_l}}e^{U_i}\frac 1{\rho}\left(- \chi_l\frac{4 \rho_l^2 }{\rho_l^2+|y_{\xi_l}|^2}-4\rho_l^2F_{\xi_l}+h^1_{\rho_l,\xi_l}+o(\rho^2)\right)dv_g\\
				&& +\mathcal{O}(\rho|\log\rho|)\\
				&=& -\sum_{i=1}^m\frac{e^{c_{\rho,\xi}}}{2\rho}\left(\frac{\varrho(\xi_i)\tau_i e^{-2\sum_{l=1}^m\rho_l^2F_{\xi_l}(\xi_i)}}{2\rho_i^2}+\frac{\varrho(\xi_i)}{8}\Delta_g( e^{\varphi_{i}\circ y_{\xi_i}^{-1}\tau_i})(\xi_i)\right)\\
				&&+\sum_{i=k+1}^m\frac{e^{c_{\rho,\xi}}}{2\rho}\frac{\varrho(\xi_i) \partial_{\nu_g}( e^{\varphi_{i}}\circ y_{\xi_i}^{-1}\tau_i)(\xi_i) }{4\rho_i}
				-4\sum_{i,l=1}^m\frac{e^{c_{\rho,\xi}}}{8\rho_i^2\rho}\varrho(\xi_i)\tau_i e^{-2\sum_{l=1}^m\rho_l^2F_{\xi_l}(\xi_i)}\rho_l^2F_{\xi_l}(\xi_i)\\
				&&+\sum_{i,l=1}^m\frac{e^{c_{\rho,\xi}}}{8\rho_i^2\rho}\varrho(\xi_i)\tau_i e^{-2\sum_{l=1}^m\rho_l^2F_{\xi_l}(\xi_i)}h^1_{\rho_l,\xi_l} +o(\frac 1 \rho)\\
				&=&e^{c_{\rho,\xi}} \left(\frac{-\lambda_{k,m}}{4\rho^3}-\frac{1}{16\rho}\sum_{i=1}^m\varrho(\xi_i)(\Delta_g\tau_i-2K_g\tau_i)|_{\xi_i}-\frac 1 {16\rho} \sum_{i=k+1}^m 4 \varrho(\xi_i)k_g(\partial_{\nu_g}\log V+ k_g)\tau_i|_{\xi_i}
				\right. \\
				&&	\left.+\sum_{i=k+1}^m \frac{1}{8\rho^2}\varrho(\xi_i)\sqrt{\tau_i
				}(\partial_{ \nu_g}\log V+2 k_g)|_{\xi_i}+\frac{\lambda_{k,m}\sum_{l=1}^m h^1_{\rho_l,\xi_l}}{8\rho^3}\right) +o(\frac 1 \rho)\\
				&=&\frac {e^{c_{\rho,\xi}}}{8\rho^2} \left(\frac{-2\lambda_{k,m}}{\rho}-\mathcal{A}_1(\xi)-\rho\mathcal{A}_2(\xi)+\frac{\lambda_{k,m}c^1_{\rho,\xi}}{\rho}\right) +o(\frac 1\rho).
			\end{array}
		\end{equation*}
		Combining with~\eqref{eq:rho2_int_veW},
		\begin{equation}\label{eq:int_vew_pa_rho_normal} 
			\begin{array}{lcl}
				\frac{\int_{\Sigma}Ve^W \partial_{\rho}W dv_g}{\int_{\Sigma}Ve^W dv_g}&=& \frac{\frac {e^{c_{\rho,\xi}}}{8\rho^2} \left(\frac{-2\lambda_{k,m}}{\rho}-\mathcal{A}_1(\xi)-\rho\mathcal{A}_2(\xi)+\frac{\lambda_{k,m}c^1_{\rho,\xi}}{\rho}\right) +o(\frac 1\rho)}{\frac {e^{c_{\rho,\xi}}}{8\rho^2}(\lambda_{k,m} +\mathcal{A}_1(\xi)\rho- \mathcal{A}_2(\xi)\rho^2\log \rho  +\mathcal{B}(\xi)\rho^2-2F_{\rho,\xi} +o(\rho^2))}\\
				&=&-\frac 2 \rho+\frac{\mathcal{A}_1(\xi)}{\lambda_{k,m}} +\left(-\frac{\mathcal{A}_2(\xi)}{\lambda_{k,m}}+\frac{2 \mathcal{B}(\xi)}{\lambda_{k,m}}-\left(\frac{\mathcal{A}_1(\xi)}{\lambda_{k,m}}\right)^2\right)\rho-\frac{2\mathcal{A}_2(\xi)}{\lambda_{k,m}}\rho\log\rho\\
				&&+ \frac{c^1_{\rho,\xi}}{\rho} -\frac{4F_{\rho,\xi}}{\lambda_{k,m}\rho}+o(\rho).
			\end{array}
		\end{equation}
		Thus we have 
		\begin{equation*}
			\begin{array}{lll}
				\partial_{\rho}J_{\lambda}(W)&=& \int_{\Sigma}(-\Delta_g W) \partial_{\rho} W dv_g-\lambda_{k,m}\frac{\int_{\Sigma} Ve^W \partial_{\rho}W dv_g }{\int_{\Sigma}Ve^W dv_g}+(\lambda_{k,m}-\lambda)\frac{\int_{\Sigma} Ve^W \partial_{\rho}W dv_g }{\int_{\Sigma}Ve^W dv_g}\\
				&\stackrel{\eqref{lamda1}}{=}& -\mathcal{A}_1(\xi) +\left(\mathcal{A}_2(\xi)-2 \mathcal{B}(\xi)+\frac{\mathcal{A}^2_1(\xi)}{\lambda_{k,m}}\right)\rho+2\mathcal{A}_2(\xi)\rho\log\rho\\
				&&
				+(\lambda_{k,m}-\lambda)\left(-\frac 2 {\rho}+\mathcal{O}(1)\right)+o(\rho)\\
				&=&\frac{2(\lambda-\lambda_{k,m})}{\rho}-\mathcal{A}_1(\xi) +\left(\mathcal{A}_2(\xi)-2 \mathcal{B}(\xi)+\frac{\mathcal{A}^2_1(\xi)}{\lambda_{k,m}}\right)\rho+2\mathcal{A}_2(\xi)\rho\log\rho+o(\rho). 
			\end{array}
		\end{equation*}
		\[\begin{array}{
				lcl}
			\partial^2_{\rho}J_{\lambda_{k,m}}(W)&=&\int_{\Sigma}(-\Delta_g W) \partial^2_{\rho}Wdv_g -\lambda_{k,m}\frac{\int_{\Sigma}Ve^W(\partial^2_{\rho}W+(\partial_{\rho}W)^2) dv_g}{\int_{\Sigma}Ve^W dv_g}\\
			&&+\int_{\Sigma}(-\Delta_g \partial_{\rho}W) \partial_{\rho}Wdv_g+\lambda_{k,m}\left(\frac{\int_{\Sigma}Ve^W\partial_{\rho}Wdv_g}{\int_{\Sigma}Ve^W dv_g}\right)^2 
		\end{array} .\]
		Since $\int_{\Sigma}\partial_{\rho}W dv_g=\int_{\Sigma}\partial^2_{\rho}W dv_g=0,$ we have 
		\begin{equation*}
			\begin{array}{lcl}
				&&\int_{\Sigma}(-\Delta_g W) \partial^2_{\rho}Wdv_g=	\sum_{i=1}^m \int_{\Sigma} \chi_ie^{-\varphi_i\circ y_{\xi_i}^{-1}} e^{U_i} \partial^2_{\rho}W dv_g \\
				&=&\frac 1 {\rho^2} \sum_{i,l=1}^m  \int_{\Sigma}\chi_ie^{-\varphi_i\circ y_{\xi_i}^{-1}} e^{U_i}\rho_l^2 \partial^2_{\rho_l}W_ldv_g\\
				&\stackrel{\eqref{w2strho}}{=}&\frac 1 {\rho^2} \sum_{i,l=1}^m  \int_{\Sigma}\chi_ie^{-\varphi_i\circ y_{\xi_i}^{-1}} e^{U_i} \left(\chi_l\frac{4\rho_l^2(\rho_l^2-|y_{\xi_l}|^2)}{(\rho_l^2+|y_{\xi_l}|^2)^2}-4\rho_l^2F_{\xi_l}+h^2_{\rho_l,\xi_l}+o(\rho^2)\right)dv_g\\
				&\stackrel{\eqref{eqspan}\text{ and }\eqref{eqspan_2}}{=}& \frac 1 {\rho^2}\sum_{i,l=1}^m\left( \frac{2\varrho(\xi_i)\delta_{il}}{3}-4\rho_l^2\varrho(\xi_i)F_{\xi_l}(\xi_i)+\varrho(\xi_i)h^2_{\rho_l,\xi_l}+o(\rho^2)\right)\\
				&=&\frac{2\lambda_{k,m}}{3\rho^2}	-\frac{4F_{\rho,\xi}}{\rho^2}+\frac {\lambda_{k,m}c^2_{\rho,\xi}}{\rho^2}+o(1)
			\end{array}
		\end{equation*}
		and 
		\begin{equation*}
			\begin{array}{lcl}
				&&\int_{\Sigma}(-\Delta_g \partial_{\rho}W) \partial_{\rho}Wdv_g=	\sum_{i=1}^m \int_{\Sigma} \chi_ie^{-\varphi_i\circ y_{\xi_i}^{-1}} e^{U_i}\partial_{\rho}U_i \partial_{\rho}W dv_g \\
				&=& \sum_{i,l=1}^m  \tau_l\int_{\Sigma}\chi_ie^{-\varphi_i\circ y_{\xi_i}^{-1}} e^{U_i} \partial_{\rho}U_i \partial_{\rho}Wdv_g\\
				&\stackrel{\eqref{w2strho}}{=}&\frac 2 {\rho^2} \sum_{i,l=1}^m  \int_{\Sigma}\chi_ie^{-\varphi_i\circ y_{\xi_i}^{-1}} e^{U_i}\frac{|y_{\xi_i}|^2-\rho_i^2}{\rho_i^2+|y_{\xi_i}|^2}\left(- \chi_l\frac{4 \rho_l^2 }{\rho_l^2+|y_{\xi_l}(x)|^2}-4\rho_l^2F_{\xi_l}+h^1_{\rho_l,\xi_l}+o(\rho^2)\right)dv_g\\
				&\stackrel{\eqref{eqspan}\text{-}\eqref{eqspan_2}}{=}& \frac 2 {\rho^2}\sum_{i,l=1}^m\left( \frac{2\varrho(\xi_i)\delta_{il}}{3}-4\rho_l^2\varrho(\xi_i)F_{\xi_l}(\xi_i)+\varrho(\xi_i)h^1_{\rho_l,\xi_l}+o(\rho^2)\right)\\
				&=&\frac{4\lambda_{k,m}}{3\rho^2}+o(1),
			\end{array}
		\end{equation*}
		where $c^2_{\rho,\xi}=\sum_{i=1}^m h^2_{\rho_i,\xi_i}$. 
		By \eqref{w1strho} and \eqref{w2strho}, we have 
		$\partial^2_{\rho}W+(\partial_{\rho}W)^2=\mathcal{O}(|\log\rho|)$ for any $x\in\Sigma\setminus\cup_{i=1}^mU_{r_0}(\xi_i)$.
		By Corollary~\ref{cor0}, \eqref{w1strho}, \eqref{w2strho} and  \eqref{eq:veW_detail}, it follows the following estimate:
		\begin{equation*}
			\begin{array}{lcl}
				&&\int_{\Sigma}Ve^W(\partial^2_{\rho}W+(\partial_{\rho}W)^2) dv_g\\
				&=&\sum_{i,l=1}^m \frac{e^{c_{\rho,\xi}}}{8\tau_i\rho^2}\int_{U_{r_0}(\xi_i)}\chi_i\tau_i(x)e^{-2\sum_{l=1}^m \rho_l^2F_{\xi_l}}e^{U_i}\frac 1{\rho^2}\left(\chi_l\frac{4\rho_l^2(\rho_l^2-|y_{\xi_l}|^2)}{(\rho_l^2+|y_{\xi_l}|^2)^2}-4\rho_l^2F_{\xi_l}+h^2_{\rho_l,\xi_l}+o(\rho^2)\right)dv_g\\
				&&+ \sum_{i=1}^m \frac{e^{c_{\rho,\xi}}}{8\tau_i\rho^2}\int_{U_{r_0}(\xi_i)}\chi_i\tau_i(x)e^{-2\sum_{l=1}^m \rho_l^2F_{\xi_l}}e^{U_i}\frac 1{\rho^2}\\
				&&\left(\sum_{l=1}^m\left(- \chi_l\frac{4 \rho_l^2 }{\rho_l^2+|y_{\xi_l}|^2}-4\rho_l^2F_{\xi_l}+h^1_{\rho_l,\xi_l}+o(\rho^2)\right)\right)^2dv_g +\mathcal{O}(|\log\rho|)\\
				&=& \sum_{i=1}^m \frac{e^{c_{\rho,\xi}}}{8\tau_i\rho^2}\int_{U_{r_0}(\xi_i)}\tau_i(x)e^{-2\sum_{l=1}^m \rho_l^2F_{\xi_l}}e^{U_i}\frac 1{\rho^2}\frac{4\rho_i^2(5\rho_i^2-|y_{\xi_i}|^2)}{(\rho_i^2+|y_{\xi_i}|^2)^2} dv_g \\
				&&+ \sum_{i=1}^m \frac{e^{c_{\rho,\xi}}}{8\tau_i\rho^2}\int_{U_{r_0}(\xi_i)}\chi_i\tau_i(x)e^{-2\sum_{l=1}^m \rho_l^2F_{\xi_l}}e^{U_i}\frac 2{\rho^2} \frac{4\rho_i^2}{\rho_i^2+|y_{\xi_i}|^2}\left(\sum_{l=1}^m 4\rho_l^2F_{\xi_l}-\sum_{l=1}^m h^1_{\rho_l,\xi_l}+o(\rho^2) \right) dv_g \\
				&&+ \sum_{i,l=1}^m \frac{e^{c_{\rho,\xi}}}{8\tau_i\rho^2}\int_{U_{r_0}(\xi_i)}\chi_i\tau_i(x)e^{-2\sum_{l=1}^m \rho_l^2F_{\xi_l}}e^{U_i}\frac 1{\rho^2}\left(-4\rho_l^2F_{\xi_l}+h^2_{\rho_l,\xi_l}+o(\rho^2) \right)+\mathcal{O}\left(|\log\rho|^2\right)\\
				&=&\frac{e^{c_{\rho,\xi}} }{8\rho^2} \left(\sum_{ i=1}^m\frac{6\varrho(\xi_i)e^{-2\sum_{l=1}^m\rho_l^2F_{\xi_l}(\xi_i)}}{\rho^2}-\sum_{ i=k+1}^m\frac {2 \varrho(\xi_i)\sqrt{\tau_i}(\partial_{ \nu_g}\log V +2k_g)|_{\xi_i}}{\rho}
				\right. \\
				&&\left.+ \sum_{ i=1}^m\frac{ \varrho(\xi_i)(\Delta_g\tau_i-2K_g\tau_i)|{\xi_i}}{2}
				+\left. \sum_{i=1}^m  2\varrho(\xi_i)k_g(\partial_{\nu_g}\log V+k_g)\tau_i\right|_{\xi_i} \right)+\frac{e^{c_{\rho,\xi}} }{8\rho^2}\left( \frac{16F_{\rho,\xi}}{\rho^2}-\frac{4\lambda_{k,m}c^1_{\rho,\xi}}{\rho^2}\right)\\
				&&+ \frac{e^{c_{\rho,\xi}} }{8\rho^2}\left(\frac{-4F_{\rho,\xi}}{\rho^2}+\frac{\lambda_{k,m}c^2_{\rho,\xi}}{\rho^2}\right)+\mathcal{O}\left(\rho^{\alpha-2}\right)\\
				&=& \frac{e^{c_{\rho,\xi}} }{8\rho^2}\left( \frac{6\lambda_{k,m}}{\rho^2}-\frac{4\lambda_{k,m}c^1_{\rho,\xi}}{\rho^2}+\frac{\lambda_{k,m}c^2_{\rho,\xi}}{\rho^2}+\frac 2 \rho\mathcal{A}_1(\xi)  +\mathcal{A}_2(\xi)+ \mathcal{O}(\rho^{\alpha})\right).
			\end{array}
		\end{equation*}
		Combining with~\eqref{eq:rho2_int_veW},
		\begin{equation*}
			\begin{array}{lcl}
				\frac{\int_{\Sigma}Ve^W(\partial^2_{\rho}W+(\partial_{\rho}W)^2) dv_g }{\int_{\Sigma}Ve^W dv_g}
				&=& \frac 6{\rho^2} -\frac{4 \mathcal{A}_1(\xi)}{\lambda_{k,m}\rho}+\frac{\mathcal{A}_2(\xi)}{\lambda_{k,m}}+\frac{6\mathcal{A}_2(\xi)\log\rho}{\lambda_{k,m}}-\frac{6\mathcal{B}(\xi)}{\lambda_{k,m}} +\frac{12F_{\rho,\xi}}{\lambda_{k,m}\rho^2}+4\left(\frac{\mathcal{A}_1(\xi)}{\lambda_{k,m}}\right)^2\\
				&&-\frac{4c^1_{\rho,\xi}}{\rho^2} +\frac{c^2_{\rho,\xi}}{\rho^2} +o(1). 
			\end{array}
		\end{equation*}
		And~\eqref{eq:int_vew_pa_rho_normal} yields that 
		\begin{equation*}
			\begin{array}{lcl}
				\left(\frac{\int_{\Sigma}Ve^W\partial_{\rho}W dv_g }{\int_{\Sigma}Ve^W dv_g}\right)^2
				&=& \frac 4{\rho^2} -\frac{4\mathcal{A}_1(\xi)}{\lambda_{k,m}\rho}+\frac{4\mathcal{A}_2(\xi)}{\lambda_{k,m}}-\frac{8\mathcal{A}_2(\xi)\log\rho}{\lambda_{k,m}}-\frac{8\mathcal{B}(\xi)}{\lambda_{k,m}}+5\left(\frac{\mathcal{A}_1(\xi)}{\lambda_{k,m}}\right)^2+\frac{16F_{\rho,\xi}}{\lambda_{k,m}\rho^2}\\
				&&-\frac{4c^1_{\rho,\xi}}{\rho^2}+o(1)
			\end{array}
		\end{equation*}
		It follows that 
		\begin{equation*}
			\begin{array}{lcl}
				\partial_{\rho}^2J_{\lambda_{k,m}}(W)&=& 3\mathcal{A}_2(\xi)+2\mathcal{A}_2\log\rho-2\mathcal{B}(\xi)+\frac{\mathcal{A}_1^2(\xi)}{\lambda_{k,m}}+o(1).
			\end{array}
		\end{equation*}
		Hence, 
		\[ 	\begin{array}{lcl}
			\partial_{\rho}^2J_{\lambda_{k,m}}(W)&=&-(\lambda-\lambda_{k,m})\left( 	\frac{\int_{\Sigma}Ve^W(\partial^2_{\rho}W+(\partial_{\rho}W)^2) dv_g }{\int_{\Sigma}Ve^W dv_g}-	\left(\frac{\int_{\Sigma}Ve^W\partial_{\rho}W dv_g }{\int_{\Sigma}Ve^W dv_g}\right)^2\right)\\
			&&+ 3\mathcal{A}_2(\xi)+2\mathcal{A}_2\log\rho-2\mathcal{B}(\xi)+\frac{\mathcal{A}_1^2(\xi)}{\lambda_{k,m}}+o(1)\\
			&=&-\frac{2(\lambda-\lambda_{k,m})}{\rho^2}+3\mathcal{A}_2(\xi)+2\mathcal{A}_2\log\rho-2\mathcal{B}(\xi)+\frac{\mathcal{A}_1^2(\xi)}{\lambda_{k,m}}+o(1).
		\end{array}\]
		For any $i=1,\cdots,m$, $j=1,\cdots, \i(\xi_i)$, 
		\[\begin{array}{
				lcl}
			\partial_{(\xi_i)_j}\partial_{\rho}J_{\lambda_{k,m}}(W)&=&\int_{\Sigma}(-\Delta_g W) 	\partial_{(\xi_i)_j}\partial_{\rho}Wdv_g -\lambda_{k,m}\frac{\int_{\Sigma}Ve^W(	\partial_{(\xi_i)_j}\partial_{\rho}W+\partial_{\rho}W\partial_{(\xi_i)_j}W) dv_g}{\int_{\Sigma}Ve^W dv_g}\\
			&&+\int_{\Sigma}(-\Delta_g \partial_{(\xi_i)_j}W) \partial_{\rho}Wdv_g+\lambda_{k,m}\left(\frac{\int_{\Sigma}Ve^W\partial_{\rho}Wdv_g}{\int_{\Sigma}Ve^W dv_g}\right)\left(\frac{\int_{\Sigma}Ve^W\partial_{(\xi_i)_j}Wdv_g}{\int_{\Sigma}Ve^W dv_g}\right)
		\end{array} .\]
		Applying Lemma~\ref{lem:expansion_mix}, we can deduce that 
		\[ 	\partial_{(\xi_i)_j}\partial_{\rho}J_{\lambda}(W)=\mathcal{O}(\rho|\log\rho|),\]	via the same approach for calculating the second-order derivatives with respect to $\rho$.
	\end{proof}
	\newpage
	
	%%%%%%%%%%%%%%%%%%%%%%%%%%%%%%%%%%%%%%
	\section{The Lyapunov-Schmidt reduction}\label{reduction}
	\subsection{The solutions to a linear problem}\label{sec_linear}
	To study the non-linear problem~\eqref{eq:main_eq}, we try to find a solution with the form $W+\phi$ for $(\rho,\xi)\in (0,+\infty)\times M_{\xi^*,\delta}$.
	Firstly, we introduce the linear operator,
	\[ L(\phi)= -\Delta_g\phi  -\lambda\frac{Ve^{W}}{\int_{\Sigma} Ve^{W} dv_g }\left(\phi- \frac{\int_{\Sigma} Ve^{W} \phi dv_g}{\int_{\Sigma} Ve^{W} dv_g} \right).  \]
	To solve~\eqref{eq:main_eq}, it is sufficient to find $\phi\in \overline{\mathrm{H}}^1$ solves 
	\begin{equation}~\label{linear}
		\begin{cases}
			L(\phi) = R+N(\phi) & x\in \intsigma\\
			\partial_{ \nu_g } \phi=0 & x\in \partial\Sigma
		\end{cases},
	\end{equation}
	where 
	\[R= \Delta_g W +\lambda\left(\frac{Ve^{W}}{\int_{\Sigma} Ve^{W} dv_g }-\frac 1 {|\Sigma|_g} \right),\]
	and 
	\[ N(\phi)= \lambda\left( \frac{Ve^{W+\phi}}{\int_{\Sigma} Ve^{W+\phi} dv_g } -\frac{Ve^{W}\phi}{\int_{\Sigma} Ve^{W} dv_g } + \frac{Ve^{W} \int_{\Sigma} V e^W\phi dv_g}{(\int_{\Sigma} Ve^{W} dv_g )^2}-\frac{Ve^{W}}{\int_{\Sigma} Ve^{W} dv_g }\right).  \]
	Providing that $W, \phi \in \overline{\mathrm{H}}^1$,  
	\[ \int_{\Sigma}L(\phi) dv_g=\int_{\Sigma} R dv_g =\int_{\Sigma}N(\phi) dv_g. \]
	Formally, we do transformation on $\phi$ to minus a normalized integral over $\Sigma$ and  do scaling centered by $0$ with $y=y_{\xi_i}(x)/\rho_i$ in local isothermal coordinates, more precisely $\Phi_{i}(y)= \phi\circ y_{\xi_i}^{-1}(\rho_i y)-\frac{\int_{\Sigma} Ve^{W}\phi dv_g }{\int_{\Sigma} Ve^{W} dv_g}$ for any $y\in \frac 1 {\rho_i}B^{\xi_i}$.
	Define that 
	\[ L_i(\Phi_i)(y)= \rho^2_i e^{\varphi_i(\rho_i y)} L(\phi)\circ y_{\xi_i}^{-1}(\rho_i y), \text{ for any } y\text{ in } \frac{1}{\rho_i}B^{\xi_i}. \] 
	Let  $\mathbb{R}_i$ be 
	$	\mathbb{R}^2$ when $\xi_i \in \intsigma$ and 
	$	\mathbb{R}^2_+$ when $\xi_i\in \partial\Sigma.$
	As $\rho\rightarrow 0$, for any $i=1,\cdots,m$, the linear operator $L_i$ approaches to $\hat{L}_i$ (refer to the detail to  Lemma~\ref{lemlin1}), 
	\begin{equation*}
		\hat{L}_i(\Phi)(y)= -\Delta \Phi -\frac{8}{(1+|y|^2)^2} \Phi, 
	\end{equation*}
	where $\Phi\in L_b^{\infty}(\RR_i):=L^{\infty}(\RR_i)\cap\{ 	\frac{\partial}{\partial y_2}\Phi=0 \text{ on  }\partial \mathbb{R}_i
	\}$. 
	The kernel of $\hat{L}_i$ in $L_b^{\infty}(\mathbb{R}_i)$ is spanned by  
	$$z_0(y)=2\frac{ 1-|y|^2}{1+|y|^2} \text{ and } z_{j}(y)=\frac{4y_j}{1+|y|^2} \text{ for }j=1,\cdots, \i(\xi_i).$$ 
	The detail refers to Lemma 2.3 in \cite{chen2002sharp} or Lemma D.1. in \cite{Esposito2005}. 
	Let $\phi_i=\phi\circ y_{\xi_i}^{-1}(\rho_i y).$ 
	As $\rho\rightarrow 0$, we assume that $\Phi_i\rightarrow \Phi_{i,\infty}, \phi_i\rightarrow \phi_{i,\infty} $. It follows $$- \frac{\int_{\Sigma} Ve^{W}\phi dv_g }{\int_{\Sigma} Ve^{W} dv_g}\longrightarrow -\frac 1 {\lambda_{k,m}} \sum_{j=1}^m\int_{\RR_i}\frac{8}{(1+|z|^2)^2} \phi_{i,\infty} dz.$$ 
	The limit linearized operator about $\phi_{i,\infty}(i=1,\cdots,m)$ is as follows: 
	\[ L_{\infty}(\psi)(y^1,\cdots,y^m)=\left(\begin{array}{c}
		-\Delta \psi_1(y^1)-\frac{8}{(1+|y^1|^2)^2}\left(\psi_1(y^1)- \frac1 {\lambda_{k,m}}\sum_{i=1}^m \int_{\RR_i} \frac{8}{(1+|\cdot|^2)^2}\psi_i \right )\\
		\\
		\vdots
		\\
		-\Delta \psi_1(y^m)-\frac{8}{(1+|y^m|^2)^2}\left(\psi_1(y^m)- \frac1 {\lambda_{k,m}}\sum_{i=1}^m \int_{\RR_i} \frac{8}{(1+|\cdot|^2)^2}\psi_i\right )
	\end{array}\right),  \]
	where $\psi=(\psi_1,\cdots,\psi_m)\in L_b^{\infty}(\RR_1)\times\cdots\times L_b^{\infty}(\RR_m) $.
	The kernel of the linear operator $L_{\infty}$ on $ L_b^{\infty}(\RR_1)\times\cdots\times L_b^{\infty}(\RR_m)$ is generated by the constant function 
	$(\underbrace{1,\cdots,1}_{\text{m}})$
	and $\la z_{j}:j=0,\cdots,\i(\xi_1) \ra\times\cdots\times\la z_{j}:j=0,\cdots,\i(\xi_m)\ra$. 
	Based on the works in~\cite{del_pino_singular_2005,Esposito2005}, it is standard to show that  $L$ is invertible in some ``orthogonal subspace'' of $\overline{\mathrm{H}}^1$.

	For any $i=1,\cdots, m$ and $j=0,1,\cdots,\i(\xi_i)$, we define $Z_{ij}$ on $U_{2r_0}(\xi_i)\subset \Sigma$ by $z_j$ which defined on $\mathbb{R}_i$ with a proper scaling centered at the origin. Specifically, for any $\xi_i\in \Sigma$, 
	\[ Z_{ij}(x)= z_j\left(\frac{y_{\xi_i}(x)}{\rho_i}\right) \text{ in } U_{2r_0}(\xi_i), \]
	for any  $i=1,2,\cdots,m$ and $j=0,1,\cdots, \i(\xi_i).$
	We project $Z_{i0}, Z_{ij}$ into the space $\overline{\mathrm{H}}^1$ by following equations for $i=1,2,\cdots,m$ and $j=0,1,\cdots,\i(\xi_i)$,
	\begin{equation}
		\begin{cases}
			-\Delta_gP Z_{ij} =-\chi_i\Delta_g Z_{ij}+\overline{\chi_i\Delta_g Z_{ij}},& \text{ in } \intsigma\\
			\partial_{ \nu_g } PZ_{ij} =0, & \text{ on } \partial\Sigma
		\end{cases}
	\end{equation}
	where $\chi_i(x)= \chi(|y_{\xi_i}(x)|/r_0)$.
	Let $PZ=\sum_{i=1}^m Z_{i0}.$
	In this part, we will mainly study a linear problem: given $h\in L^{\infty}(\Sigma):=\{h: \|h\|_{\infty}:=\sup _{\Sigma}|h|<\infty\}$, find $\phi\in \overline{\mathrm{H}}^1\cap W^{2,2}(\Sigma)$ and $c_0, c_{ij}\in \mathbb{R}$ such that 
	\begin{equation}~\label{linprob}
		\begin{cases}
			L(\phi)= h -c_0\Delta_g PZ -\sum_{i=1}^m \sum_{j=1}^{\i(\xi_i)} c_{ij}\Delta_gPZ_{ij}& \text{ in }\intsigma;\\
			\partial_{ \nu_g }\phi= 0& \text{ on }\partial \Sigma;\\
			\int_{\Sigma} \phi \Delta_g PZ_{ij}dv_g=0, \int_{\Sigma} \phi \Delta_g PZ dv_g=0& \forall i=1,2,\cdots,m, j=1,\cdots, \i(\xi_i).  
		\end{cases}
	\end{equation}
	To estimate the size of the solutions, for any $h\in L^{\infty}(\Sigma)$, we introduce the following weighted norm (see~\cite{Esposito2014singular})
	\begin{equation}~\label{weightednorm}
		\|h\|_*= \sup_{x\in \Sigma}\left( \sum_{i=1}^m \frac{ \rho_i^{\kappa}}{( \rho_i^2 +\mathbbm{1}_{U_{r_0}(\xi_i)}(x) |y_{\xi_i}(x)|^2 + r_0^2 \mathbbm{1}_{\Sigma\setminus U_{r_0}(\xi_i)}(x))^{1+\frac{\kappa}{2}}}\right)^{-1}|h(x)|.
	\end{equation}
	Here, the characteristic functions  $\mathbbm{1}_A :=\begin{cases}
		1, & x\in A\\
		0, & x\notin A
	\end{cases},$ for any $A\subset \Sigma$ and $\kappa\in (0,\alpha_0)$, where $\alpha_0\in (0,1)$ is a constant. 
	\begin{lemma}~\label{est_R}
		For any $\delta>0$ sufficiently small, assume that $\xi\in M_{\xi^*,\delta}$,  \eqref{taui}-\eqref{lamda1} hold truly. Then there is a constant $C>0$ such that 
		\[\|R\|_*\leq 	C(|\nabla \cF_{k,m}(\xi)|_g\rho+\rho^{2-\kappa}|\log\rho|). \]
	\end{lemma}
	\begin{proof}
		Denote $R_{\lambda_{k,m}}:= \Delta_g W +\lambda_{k,m}\left( \frac{Ve^W}{\int_{\Sigma} Ve^W dv_g} -\frac{ 1 }{|\Sigma|_g}\right)$. 
		Applying Corollary~\ref{cor0} with $f\equiv1$, we derive that 
		$ \int_{\Sigma}\chi_i e^{-\varphi_i\circ y_{\xi_i}^{-1}} e^{U_i} dv_g= \varrho(\xi_i)+ \mathcal{O}(\rho^2). $
		Then, 
		\begin{equation*}
			\begin{array}{lll}
				R_{\lambda_{k,m}}&=& -\sum_{i=1}^m \chi_i e^{-\varphi_i\circ y_{\xi_i}^{-1}} e^{U_i}+ \frac 1 {|\Sigma|_g}\left( \sum_{i=1}^m\int_{\Sigma} \chi_i e^{-\varphi_i\circ y_{\xi_i}^{-1}} e^{U_i} dv_g - \lambda_{k,m}\right)+  \frac{\lambda_{k,m}Ve^W}{\int_{\Sigma} Ve^W dv_g}\\
				&=&-\sum_{i=1}^m \chi_i e^{-\varphi_i\circ y_{\xi_i}^{-1}} e^{U_i} + \frac {\mathcal{O}(\rho^2) }{|\Sigma|_g}+  \frac{\lambda_{k,m}Ve^W}{\int_{\Sigma} Ve^W dv_g}
			\end{array}
		\end{equation*}
		Using Lemma~\ref{lemb1} combined with \eqref{eq:rho2_int_veW} and \eqref{extau}, 
		\[\lambda_{k,m}\frac{Ve^W}{\int_{\Sigma}Ve^Wdv_g}  =\frac{\lambda_{k,m}\mathcal{O}(1)}{ \frac{\lambda_{k,m}} {8\rho^2} +\mathcal{O}(\frac 1 {\rho})} =\mathcal{O}(\rho^2), \text{ for any } x\in \Sigma\setminus\bigcup_{i=1}^m U_{r_{0}}(\xi_i); \text{ and } \]
		\begin{equation*}
			\begin{array}{lcl}
				\lambda_{k,m}\frac{Ve^W}{\int_{\Sigma}Ve^Wdv_g}  &=& \frac{\lambda_{k,m}\tau_i(x)\frac {e^{U_i}}{8\rho^2 \tau_i} (1+c_{\rho,\xi}-2\sum_{l=1}^m\rho_l^2F_{\xi_l}(x)+\mathcal{O}(\rho^4|\log\rho|^2))}
				{\frac{\lambda_{k,m}} {8\rho^2} +\frac {\mathcal{A}_1(\xi)} {8\rho}+\mathcal{O}(|\log \rho|) }\\
				&=& \left(1+\left\la\frac 1 {\tau_i} \nabla(\tau_i\circ y_{\xi_i}^{-1})(0), y_{\xi_i}(x)\right\ra+\mathcal{O}(|y_{\xi_i}(x)|^2+\rho^2|\log \rho|) \right)\\
				&&\cdot (1+\mathcal{A}_1(\xi)\lambda_{k,m}^{-1}\rho)^{-1}e^{U_i},  
			\end{array}
		\end{equation*}
		$\text{ for any } x\in U_{r_0}(\xi_i).$ Combining the estimates above, 
		\begin{equation}~\label{sumex1}
			\begin{array}{ll}
				\lambda_{k,m}\frac{Ve^W}{\int_{\Sigma}Ve^Wdv_g} 
				&= \sum_{i=1}^m \chi_i  \left(1+\left\la\frac 1 {\tau_i} \nabla(\tau_i\circ y_{\xi_i}^{-1})(0), y_{\xi_i}(x)\right\ra+\mathcal{O}(|y_{\xi_i}(x)|^2+\rho^2|\log \rho|) \right)\\
				& \cdot(1+\mathcal{A}_1(\xi)\lambda_{k,m}^{-1}\rho)^{-1}e^{U_i}+ \mathcal{O}(\rho^2)\mathbbm{1}_{\Sigma\setminus \cup_{i=1}^m U_{r_0}(\xi_i)}.
			\end{array} 
		\end{equation}
		Applying~\eqref{sumex1}, 
		if $x\in \Sigma\setminus\cup_{i=1}^m U_{r_0}(\xi_i)$, 
		$R_{\lambda_{k,m}}=\mathcal{O}(\rho^2);$ if $x\in U_{r_0}(\xi_i)$, for any $i=1,2,\cdots,m$, 
		\begin{equation*}
			\begin{array}{lll}
				R_{\lambda_{k,m}}&=&\mathcal{O}(  |\nabla \log \tau_i\circ y_{\xi_i}^{-1}(0)||y_{\xi_i}(x)|+\frac{|\mathcal{A}_1(\xi)|}{\lambda_{k,m}}\rho+ |y_{\xi_i}(x)|^2+\rho^2|\log \rho| ) e^{U_i} +\mathcal{O}(\rho^2).
			\end{array}
		\end{equation*}
		By the definition of $\|\cdot\|_*$, 
		\begin{equation}
			\|R_{\lambda_{k,m}}\|_*\leq C(|\nabla \cF_{k,m}(\xi)|_g\rho+\rho^{2-\kappa}|\log \rho|),
		\end{equation}
		where we applied that $\lambda_{k,m}|\mathcal{A}_1(\xi)|\leq|\nabla\log \tau_i\circ y_{\xi_i}^{-1}(0)|\leq C |\nabla_{\xi_i}\cF_{k,m}(\xi)|_g.$
		\begin{equation*}
			\begin{array}{lll}
				R-R_{\lambda_{k,m}}&=& (\lambda-\lambda_{k,m}) \left(\frac{Ve^W}{\int_{\Sigma} Ve^W dv_g}-\frac{1}{|\Sigma|_g}\right)\\
				&=&\mathcal{O}\left( |\lambda-\lambda_{k,m}|\sum_{i=1}^m \chi_i e^{U_i} + |\lambda-\lambda_{k,m}|\right)
			\end{array}
		\end{equation*}
		Then $\|R-R_{\lambda_{k,m}}\|_*\leq C\rho^{-\kappa} |\lambda-\lambda_{k,m}|.$
		By~\eqref{lamda1}, $\|R-R_{\lambda_{k,m}}\|_*\leq \mathcal{O}\left(\rho^{2-\kappa}|\log \rho|\right).$
		Therefore, 
		\[ \|R\|_*\leq C(|\nabla \cF_{k,m}(\xi)|_g\rho+\rho^{2-\kappa}|\log\rho|).\]
	\end{proof}
	
	We denote that the coefficient vector as follows: $$\tilde{c}:=(c_0, c_{11},\cdots,c_{1\i(\xi_1)}, c_{21}, \cdots\cdots, c_{2\i(\xi_2)},\cdots, c_{m1},\cdots, c_{m\i(\xi_m)})\in \RR^{m+k+1}.$$
	\begin{proposition}~\label{propolin}
		Let $k\leq m\in \mathbb{N}_+$ and $\delta>0$.
		Then there exists $\rho_0>0$ such that for any $\rho\in (0,\rho_0)$, for any $h\in C(\Sigma)$ with $\int_{\Sigma}hdv_g=0$ and for any $\xi\in M_{\xi^*,\delta}$, there is a unique solution $(\phi, \tilde{c})\in \overline{\mathrm{H}}^1\cap W^{2,2}(\Sigma)\times \mathbb{R}^{k+m+1}$ to~\eqref{linprob}. Moreover, the map 
		\[ (\rho,\xi)\in (0,\rho_0)\times M_{\xi^*,\delta} \mapsto (\phi,\tilde{c})\]
		is $C^1$-differentiable in $\xi$ with 
		\begin{eqnarray}
			& \|\phi\|_{\infty} \leq C|\log \rho|\|h\|_*, \quad |c_0|+\sum_{i=1}^m \sum_{j=1}^{\i(\xi_i)}\left|c_{i j}\right| \leq C\|h\|_*,~\label{0rest}\\
			&\|\partial_{\rho}\phi\|_{\infty}+\frac{\rho}{|\log \rho|}\|\partial^2_{\rho}\phi\|_{\infty}+\frac{\rho}{|\log \rho|}\|\nabla_{\xi}\partial_{\rho}\phi\|_{\infty}+	\|\nabla_{\xi} \phi\|_{\infty}
			\leq C\frac{|\log \rho|^2}{\rho} \|h\|_*~\label{est_xi1}, 
		\end{eqnarray}
		for some $C>0$.
	\end{proposition}
	Before giving the proof of Proposition \ref{propolin}, we sate a prior estimate for the solutions $\phi$ to~\eqref{linprob}, assuming the coefficient vector  $\tilde{c}$ vanishes. Denote
	\[ 	L_{\lambda_{k,m}(\phi)}:= -\Delta_g\phi - \frac{\lambda_{k,m} V e^W}{\int_{\Sigma} Ve^W dv_g }\left(\phi- \frac{ \int_{\Sigma} Ve^W \phi dv_g }{\int_{\Sigma} Ve^W dv_g}\right). \]
	\begin{lemma}~\label{lemlin1}
		There exists $\rho_0>0$ and $C>0$ such that for any $\rho\in (0, \rho_0)$  $\xi=(\xi_1,\xi_2,\cdots,\xi_m)\in M_{\xi^*,\delta}$,  $h\in C(\Sigma)$ with $\int_{\Sigma}h dv_g =0$ and solution $\phi \in \overline{\mathrm{H}}^1\cap W^{2,2}(\Sigma)$ to~\eqref{linprob} with $L=L_{\lambda_{k,m}}$ and $\tilde{c}=0(\in \RR^{m+k+1})$,  
		\begin{equation}
			\|\phi\|_{\infty}\leq C |\log \rho|\|h\|_*.
		\end{equation}
	\end{lemma}
	\begin{proof}
		We will prove the lemma by constructing a contradiction. Assume Lemma~\ref{lemlin1} fails. There exists a sequence of $\rho^n\rightarrow 0$, $\xi^n(\in M_{\xi^*,\delta})\rightarrow \xi^*$ in $M_{\xi^*,\delta}$, $h_n\in C(\Sigma)$ with $\int_{\Sigma} h_n dv_g =0$ and $\phi_n\in \overline{\mathrm{H}}^1\cap W^{2,2}(\Sigma)$  solutions to \eqref{linprob} with $L=L_{\lambda_{k,m}}$ and  $\|\phi_n\|_{\infty}=1$ satisfying that 
		\[ \|h_n\|_*|\log (\rho^n)|=o(1),\]
		as $n\rightarrow +\infty$. 
		
		For simplicity, we still use the notations $\rho$, $\xi$, $h$ and $\phi$ instead of $\rho^n$, $\xi^n$, $h_n$ and $\phi_n$, respectively.
		Define that $
		\Phi= \phi- \frac{ \int_{\Sigma} Ve^W\phi dv_g }{\int_{\Sigma} Ve^W dv_g},$ and $\mathcal{K}= \frac{\lambda_{k,m} Ve^W}{\int_{\Sigma} Ve^W dv_g}$. 
		Consequently, problem~\eqref{linprob} transforms into the following equation:
		\begin{equation}~\label{Phi}
			\begin{cases}
				-\Delta_g\Phi -\mathcal{K}\Phi= h & \text{ in }\intsigma\\
				\partial_{ \nu_g }\Phi =0 & \text{ on }\partial \Sigma\\
				\int_{\Sigma} \Phi \Delta_g PZ_{ij}dv_g=0, 	\int_{\Sigma} \Phi \Delta_g PZdv_g=0& \forall i=1,2,\cdots,m, j=1,\cdots, \i(\xi_i)
			\end{cases},
		\end{equation} 
		where $\mathcal{C}= \frac{ \int_{\Sigma} Ve^W\phi dv_g }{\int_{\Sigma} Ve^W dv_g}$. Additionally, $\Phi$ satisfies the same orthogonality properties as $\phi$, and $|\Phi|_{\infty} \leq 2|\phi|_{\infty} \leq 2.$
		
		Recall that $\rho_i^2=\rho^2 \tau_i$. 
		Let $\Phi_i(y)=\Phi\circ y_{\xi_i}^{-1}(\rho_i y)$, for any $y\in \frac 1 {\rho_i}B_{2r_0}^{\xi_i}$.
		It follows that  
		\[- \Delta\Phi_i - \tilde{\mathcal{K}}_i\Phi_i + \tilde{\mathcal{C}}=\tilde{h}_i,\]
		where $\tilde{\mathcal{K}}_i(y)=\rho_i^2 e^{{\varphi}_{\xi_i}(\rho_i y)} \mathcal{K}\circ y^{-1}_{\xi_i}(\rho_i y) $, $\tilde{\mathcal{C}}=\rho_i^2 e^{{\varphi}_{\xi_i}(\rho_i y)} \mathcal{C}$ and $\tilde{h}_i(y)= \rho_i^2 e^{{\varphi}_{\xi_i}(\rho_i y)} h\circ y^{-1}_{\xi_i}(\rho_i y)$. 
		Since $|\mathcal{C}|\leq 1$ and $\|\Phi_i\|_{\infty}\leq 2$, 
		$\rho_i^2 e^{{\varphi}(y)}\mathcal{C}={\mathcal{O}}(\rho^2).$ $|\tilde{h}_i(y)|\leq C \|h\|_{*}$, where $C$ is a constant. And $$\tilde{\mathcal{K}}_i(y)=\frac{8}{(1+|y|^2)^2} \left( 1+\mathcal{O}(\rho) \right)+\mathcal{O}(\rho^2)$$ uniformly in any compact subset of $\RR_i$, by \eqref{sumex1}. 
		For any $p>1$, we have the estimate:
		\begin{equation}~\label{phii}
			\|\Phi_i\|_{W^{2,p}( \frac 1 {\rho_i} B_{2r_0}^{\xi_i})}\leq C,
		\end{equation} 
		where $C$ is a constant that depends only on $\Sigma$.
		Denote the scaled domain $\Omega^{\rho}_{i}:=\frac 1 {\rho_i} B_{2r_0}^{\xi_i}. $ 
		
		\begin{claim}~\label{c3.1}
			For any $i=1,2,\cdots,m$, up to a subsequence,  ${\Phi}_i\rightarrow \Phi_{i,\infty} $ in $C^1_{loc}(\mathbb{R}_i)$ as $\rho\rightarrow 0$, where  ${\Phi}_{i,\infty}$ is  a solution of 
			\begin{equation}
				\begin{cases}
					-\Delta {\Phi}_{i,\infty}=\frac 8 {(1+|z|^2)^2}\Phi_{i,\infty}, & \text{ in } \mathbb{R}_i \\
					|{\Phi}_{i,\infty}|\leq C
				\end{cases},
			\end{equation}
			for some constant $C$. 
			Moreover, for any $i=1,\cdots,m$, $j=1,\cdots,\i(\xi_i)$, \[ \int_{\mathbb{R}_i} \frac{z_j}{(1+|z|^2)^3} {\Phi}_{i,\infty}=0,\]
			and 
			\[ \sum_{i=1}^m \int_{\RR_i} \frac{1-|z|^2}{(1+|z|^2)^3}\Phi_{i, \infty}=0.\]
		\end{claim}
		By the regularity theory, ${\Phi}_{i,\infty}$ is a smooth solution. 
		It follows that $ {\Phi}_{i,\infty}$ is a linear combination of $z_j(j=0,\cdots, \i(\xi_i))$ (see~\cite{baraket_construction_1997} or \cite{Esposito2005}), i.e. ${\Phi}_{i,\infty}(z)= \sum_{j=0}^{\i(\xi_i)} a_{ij}z_j(z) $ in $\mathbb{R}_i$.
		From Claim~\ref{c3.1}, we deduce that $a_{ij} = 0$ for $i=1,\cdots,m$ and $j=1,\cdots, \i(\xi_i)$ and $\sum_{i=1}^m a_{i0}=0.$
		\begin{claim}~\label{c3.2}
			For any $i=1,2,\cdots,m$, as $\rho\rightarrow 0$,
			\begin{equation}
				\Phi_i(0) =\Phi(\xi_i) = \mathfrak{c}_0 +4a_{i0} +o(1), 
			\end{equation}
			where $ \mathfrak{c}_0=-\lim \frac{\int_{\Sigma} Ve^W \phi dv_g }{\int_{\Sigma} Ve^W dv_g}. $
		\end{claim}
		Claim~\ref{c3.2} implies $2 a_{i0} = \mathfrak{c}_0 + 4 a_{i0},$ thereby yielding $a_{i0}=-\frac{1}{2} \mathfrak{c}_0$. Hence, 
		for any $i=1,2,\cdots,m$, $a_{i0}=0$ and $\mathfrak{c}_0=0$. 
		
		Consequently, $\Phi$ converges uniformly to 0 on $\cup_{i=1}^m B_{\rho R}(\xi_i)$ for any $R > 0$.
		Following the idea in~\cite{del_pino_singular_2005}, we now will prove 	\[ \Phi\rightarrow 0 \text{ in } L^{\infty}(\Sigma). \]
		utilizing the maximum principle. 
		
		The linear operator $\mathcal{L}=-\Delta_g -\mathcal{K}$ satisfies the comparison principle in $\Sigma_{\rho}= \bigcup_{i=1}^m (U_{r_0}(\xi_i) \setminus U_{\rho_i R_0}(\xi_i))$, where $R_1 > 0$ is chosen later. Specifically,  if  $\varphi\in C^2(\Sigma_{\rho})$ satisfies 
		\begin{equation}
			~\label{comparison}
			\begin{cases}
				\mathcal{L} \varphi \geq 0& \text{ in } \Sigma_{\rho}\\
				\partial_{ \nu_g } \varphi \geq 0 & \text{ on }  \partial \Sigma\cap \Sigma_{\rho}\\
				\varphi \geq 0 & \text{ on } \intsigma\bigcap \left( \bigcup_{i=1}^m (\partial U_{\rho_i R_0}(\xi_i)\cup \partial U_{r_0}(\xi_i))\right)
			\end{cases},
		\end{equation}
		then 
		\[ \psi\geq 0 \text{ in } \Sigma_{\rho}. \]
		To establish the comparison principle, we need to construct a barrier function $\psi$ that satisfies the following conditions:
		\begin{equation*}
			\begin{array}{cl}
				-\Delta_g\psi-\mathcal{K}\psi\geq  \sum_{i=1}^m\mathbbm{1}_{U_{r_0}(\xi_i)}\frac{\kappa^2}{2}  \frac{\rho_i^{\kappa}}{|y_{\xi_i}(x)|^{\kappa+2}}&\text{ in } \bigcup_{i=1}^m( U_{r_0}(\xi_i)\setminus  U_{\rho_i R_1}(\xi_i)), \\
				\frac{\partial \psi}{\partial_{\nu_g}}\geq 0  &\text{ on }\partial\Sigma\cap \cup_{i=1}^m U_{r_0}(\xi_i)   \setminus  \cup_{i=1}^m U_{\rho_i R_1}(\xi_i), \\
				\psi>0  &\text{ on } \intsigma \bigcap \left(  \bigcup_{i=1}^m(\partial U_{r_0}(\xi_i)\cup \partial  U_{\rho_i R_1}(\xi_i) )\right). 
			\end{array}
		\end{equation*}
		Furthermore, we can expect that $\psi$ is uniformly bounded, i.e.,  for some constant $C>0$, 
		\[ 0<\psi\leq C \text{ in }  \bigcup_{i=1}^m( U_{r_0}(\xi_i)\setminus  U_{\rho_i R_1}(\xi_i)). \]
		% Moreover, the constants $R_0$ and $C$ can be chosen independently with $\rho$ and $\psi$. 
		
		Define $\psi_{1i}(x)= 1-\chi_{i}(x) \frac{\rho_i^{\kappa}}{r^{\kappa}}$, where $r= |y_{\xi_i}(x)|$. 
		Then 
		\begin{equation*}
			\begin{array}{llll}
				-\Delta_g\psi_{1i} &=& \mathbbm{1}_{U_{r_0}(\xi_i)} (x) e^{-{\varphi}_{\xi_i}\circ y_{\xi_i}^{-1}(x)}
				\frac{\rho_i^{\kappa}\kappa^2}{ |y_{\xi_i}(x)|^{\kappa+2}}&\text{ in } U_{r_0}(\xi_i)\setminus U_{\rho_i R}(\xi_i)\\
				\partial_{\nu_g}\psi_{1i}(x)&=&  \mathbbm{1}_{U_{r_0}(\xi_i)} (x) \frac{\kappa \rho_i^{\kappa}}{|y_{\xi_i}(x)|^{\kappa}} \frac{\partial_{ \nu_g }|y_{\xi_i}(x)|}{|y_{\xi_i}(x)|}&\text{ on } \partial(U_{r_0}(\xi_i) \setminus U_{\rho_i R}(\xi_i)). 
			\end{array}
		\end{equation*}
		We observe that  $$	\partial_{\nu_g}\psi_{1i}(x)\equiv 0 $$ on the boundary $\partial\Sigma$. 
		Let $\psi= \sum_{i=1}^m \psi_{1i}$.
		Now, we  choose $R_1>{2}/{\kappa^2} $ to be sufficiently large and $\rho_* > 0$ to be sufficiently small such that for any $\rho \in (0, \rho_*]$ the following conditions hold:
		\begin{eqnarray*}
			\psi_{1i}\geq \frac 1 2,  \text{ in }U_{r_0}(\xi_i)\setminus U_{\rho_i R_1}(\xi_i);\\
		0<\psi \leq m:=C,  \text{ in } 
			\bigcup_{i=1}^m  (U_{r_0}(\xi_i)\setminus B_{\rho_iR_1}(\xi_i)); \\
			\psi \geq\frac  1 2,  \text{ on } \intsigma \bigcap \left(  \bigcup_{i=1}^m(\partial U_{r_0}(\xi_i)\cup \partial  U_{\rho_i R_1}(\xi_i) )\right).
		\end{eqnarray*}
		Moreover, since  for any $x\in U_{r_0}(\xi_i)$
		\[\rho^2 \mathcal{K}\leq   \sum_{i=1}^m \mathbbm{1}_{U_{r_0}(\xi_i)}\frac 1 { \tau_i\left(1+ \left|\frac{y_{\xi_i}(x)}{\rho_i}\right|^2\right)^2} +\mathcal{O}(\rho^2). \]
		\begin{equation*}
			\begin{array}{lll}
				\rho^2 \left(-\Delta_g\psi -\mathcal{K}\psi\right)&\geq &  \sum_{i=1}^m\frac{\kappa^2}{\tau_i} \mathbbm{1}_{U_{r_0}(\xi_i)} \frac{\rho_i^{\kappa+2}}{|y_{\xi_i}(x)|^{\kappa+2}}-C\rho^2\mathcal{K}  \\
				&\geq&  \sum_{i=1}^m\frac{\kappa^2}{2\tau_i}  \mathbbm{1}_{U_{r_0}(\xi_i)} \frac{\rho_i^{\kappa+2}}{|y_{\xi_i}(x)|^{\kappa+2}},
			\end{array}
		\end{equation*}
		$\text{ in } \Sigma\setminus \cup_{i=1}^m U_{\rho_i R_1}(\xi_i).$ And 
		\[ 
		\partial_{ \nu_g }\psi \geq 0,  \text{ on }  \partial\Sigma\cap U_{r_0}(\xi_i)\setminus U_{\rho_i R_1}(\xi_i).
		\]
		Consider $\varphi \in \Sigma_{\rho}$ satisfying ~\eqref{comparison} and take $R_0 = 2R_1$. Let $\varphi_{\epsilon} = \varphi + \epsilon\psi$ for any $\epsilon > 0$. Then $\varphi_{\epsilon}$ solves the following problem: 
		\begin{equation}~\label{com_ep}
			\begin{cases}
				\mathcal{L} \varphi_{\epsilon} \geq  \epsilon \sum_{i=1}^m\frac{\kappa^2}{2}  \mathbbm{1}_{U_{r_0}(\xi_i)} \frac{\rho_i^{\kappa}}{|y_{\xi_i}(x)|^{\kappa+2}} & \text{ in } \Sigma_{\rho}\\
				\partial_{ \nu_g } \varphi_{\epsilon} \geq 0 & \text{ on }  \partial \Sigma\cap \Sigma_{\rho}\\
				\varphi_{\epsilon} >0 & \text{ on } \intsigma\bigcap \left( \bigcup_{i=1}^m (\partial U_{\rho_i R_0}(\xi_i)\cup \partial U_{r_0}(\xi_i))\right)
			\end{cases}.
		\end{equation}
		Consider $\varphi_{\epsilon}(x_0) = \min_{\overline{\Sigma}{\rho}} \varphi_{\epsilon}$. Assuming $\varphi_{\epsilon}(x_0) < 0$, we can deduce $x_0 \notin \partial \Sigma_{\rho};$ otherwise, by the boundary conditions of~\eqref{com_ep}, $x_0\in \partial\Sigma\cap \Sigma_{\rho}$. Then it follows that $\partial_{\nu_g}\varphi_{\epsilon}\equiv 0$ on $\partial\Sigma\cap \Sigma_{\rho}$. Without loss of generality, we suppose that $x_0\in U_{r_0}(\xi_i)$. We conduct an even extension of $\varphi_{\epsilon}\circ y_{\xi_i}(y)$ from $\B^+_{r_0}$ to $\B_{r_0}$, denoted the function by $\tilde{\varphi}^i_{\epsilon}$.
		We define that $y^*= \begin{cases}
			(y_1,y_2) &\text{ if } y_2\geq 0\\
			(y_1,-y_2)& \text{ if } y_2<0
		\end{cases}$ for any $y\in \B_{r_0}.$
		Then for any  $y\in \B_{r_0}$. we have 
		\[ (-\Delta -e^{\varphi_i(y^*)}\mathcal{K}(y_{\xi_i}^{-1}(y^*))\ge  e^{\varphi_i(y^*)}\geq  \epsilon \frac{\kappa^2}{2}   \frac{\rho_i^{\kappa}}{|y|^{\kappa+2}}>0.\]
		The strong maximum principle yields that the minimum point of $\tilde{\varphi}^i_{\epsilon}$ must be on $\partial \B_{r_0}$ if $\tilde{\varphi}^i_{\epsilon}$ is nontrivial, which implies that $x_0\in\intsigma\cap  \partial U_{r_0}(\xi_i)$. It contradicts to $\varphi_{\epsilon}>0$ in~\eqref{com_ep}.

		If $x_0 \in \intsigma_{\rho}$, applying the strong maximum principle yields that $\varphi_{\epsilon}$ is constant on each connected component of $\Sigma_{\rho}$,
		However, since $\varphi_{\epsilon} > 0$ on $\intsigma\cap \partial\Sigma_{\rho}$, this implies that $\varphi_{\epsilon}(x_0) > 0$, leading to a contradiction.
		
		Hence, we can conclude that $\varphi_{\epsilon} = \varphi + \epsilon\psi \geq 0$ on $\Sigma_{\rho}$. By the arbitrary choice of $\epsilon > 0$, it follows that $\varphi \geq 0$ on $\Sigma_{\rho}$.
		
		Let us define
		\[\|\Phi\|_{\rho}= \sup_{\bigcup_{i=1}^m (B_{\rho_i R_0}(\xi_i)\cup (\Sigma\setminus U_{r_0}(\xi_i)))} |\Phi|.\]
		Consider $\hat{\Phi} = C_1 \psi(\mathcal{C} + |\Phi|{\rho} + |h|_{*})$, where $C_1$ is a sufficiently large constant independent of $\Phi$ and $\rho$ such that 
		\[
		\left|h\right|\leq C_1\|h\|_{*} \left( \sum_{i=1}^m\mathbbm{1}_{U_{r_0}(\xi_i)}\frac{\kappa^2}{2}  \frac{\rho_i^{\kappa}}{|y_{\xi_i}(x)|^{\kappa+2}} \right),
		\]
		\[ |\Phi|\leq \hat{\Phi} \text{ on } \intsigma\bigcap \left( \bigcup_{i=1}^m (\partial U_{\rho_i R_1}(\xi_i)\cup \partial U_{r_0}(\xi_i))\right).\]
		Applying the comparison principle~\eqref{comparison} stated earlier, we have
		\[ -\hat{\Phi}\leq \Phi \leq \hat{\Phi}, \text{ for any }x\in \bigcup_{i=1}^m (U_{r_0}(\xi_i)\setminus B_{\rho_iR_1}(\xi_i)).  \]
		We observe that $-\Delta_g \Phi=o(1)$ in $C_{loc}(\Sigma\setminus\{\xi_1^*,\cdots,\xi_m^*\})$, up to a subsequence, 
		$\Phi\rightarrow \Phi_{\infty} $ in $C^1_{loc}(\Sigma\setminus\{\xi_1^*,\cdots,\xi_m^*\})$. Since $\Phi_{\infty}$ is bounded, it extends to a solution of $-\Delta_g \Phi_{\infty}=0$ in $\Sigma$. We observe that  $\frac 1 {|\Sigma|_g}\int_{\Sigma}\Phi dv_g =- \frac{\int_{\Sigma} Ve^W \phi dv_g }{\int_{\Sigma} Ve^W dv_g }$. By the representation formula, for any  $x\in \Sigma\setminus \{\xi_1^*,\cdots,\xi_m^*\}$, 
		\begin{equation*}
			\begin{array}{lll}
				\Phi_{\infty}(x)&=& \lim_{\rho\rightarrow 0}\left(\frac 1 {|\Sigma|_g}\int_{\Sigma} \Phi dv_g +  \int_{\Sigma} G^g(x,z) (-\Delta_g) \Phi dv_g \right. \\
				&&\left.	+ \int_{\partial \Sigma} G^g(x,z) \partial_{ \nu_g } \Phi ds_g\right)\\
				&=& \mathfrak{c}_0=0.
			\end{array}
		\end{equation*}
		Thus, up to a subsequence,
		\[ \|\Phi\|_{\infty}\leq C( \mathcal{C}+\|\Phi\|_{\rho}+ \|h\|_*)\rightarrow 0, \]
		which contradicts with $\|\phi\|_{\infty}=1.$ \par 
		It remains to prove the claims. 
		\par 
		{\it Proof  of Claim~\ref{c3.1}. }
		For the case $i=1,2,\cdots,k$, $\xi_i\in \intsigma$. We refer to Appendix A. of~\cite{Esposito2014singular}. So we just give the proof for $\xi_i\in \partial\Sigma$, here. 
		Note that 
		${\Phi}_i(z)\in C( \Omega^{\rho}_i)$ and $\|{\Phi}_i\|_{W^{2,p}(\Omega^{\rho}_{i})}\leq C \|\Phi_i\|_{W^{2,p}(\frac 1 {\rho_i} B^{\xi_i}_{2r_0})}.$   By the Sobolev embedding theorem ,
		${\Phi}_i$ is uniformly bounded in $C^{1,\gamma_0}(\Omega^{\rho}_i)$ for some  $\gamma_0\in (0,1)$. By the Arzelà–Ascoli theorem, up to a subsequence, in $C^1_{loc}( \mathbb{R}_i)$
		\[{\Phi}_i\rightarrow {\Phi}_{i,\infty}, \]
		where ${\Phi}_{i,\infty}$ is a bounded solution to 
		$-\Delta {\Phi}_{i,\infty}=\frac 8 {(1+|z|^2)^2}{\Phi}_{i,\infty}$ in $\mathbb{R}_i.$
		Note that \[-\Delta_gPZ_{ij}= \chi_{i} e^{-{\varphi}_{\xi_i} }e^{U_i}Z_{ij} -\frac 1 {|\Sigma|_g}\int_{\Sigma} \chi_{i} e^{-{\varphi}_{\xi_i}}e^{U_i}Z_{ij} dv_g . \]
		By the orthogonal properties,
		\begin{equation*}
			\begin{array}{lll}
				0&= &-\int_{\Sigma} \Delta_g PZ_{ij} \Phi dv_g\\ &=& \int_{B_{2r_0}^{\xi_i}}\chi(|y|) \frac{ 8}{(\rho_i^2 +|y|^2)^2} \frac {4\rho_iy_j}{\rho_i^2 + |y|^2}\Phi\circ y^{-1}_{\xi_i}(y) dy\\
				&&- \frac 1 {|\Sigma|_g}
				\int_{\Sigma} \Phi dv_g \int_{B_{2r_0}^{\xi_i}}	\chi(|y|) \frac{ 8}{(\rho_i^2 +|y|^2)^2} \frac {4\rho_i y_j}{\rho_i^2 + |y|^2} dy \\
				&=&  32 \int_{\mathbb{R}^2_{+}} \frac {z_j}{(1+|z|^2)^3} {\Phi}_i(z) dz -  \frac {32}{|\Sigma|_g}
				\int_{\Sigma} \Phi dv_g  \int_{\mathbb{R}^2_{+}} \frac {z_j}{(1+|z|^2)^3} dz+o(1)\\
				&=& 32 \int_{\mathbb{R}^2_{+}} \frac {z_j}{(1+|z|^2)^3} {\Phi}_i(z) dz+o(1),
			\end{array}
		\end{equation*}
		where $j=1$ and $\int_{\mathbb{R}^2_{+}} \frac {z_j}{(1+|z|^2)^3} dz=0$. Applying Lebesgue's dominated convergence theorem,  we have  $32\int_{\mathbb{R}^2_+} \frac{ z_j}{(1+|z|^2)^3} {\Phi}_{i,\infty} dz=0$. 
		Similarly, 
		\begin{equation*}
			\begin{array}{lll}
				0&= &-\int_{\Sigma} \Delta_g PZ \Phi dv_g\\ &=&\sum_{i=1}^m  \int_{B_{2r_0}^{\xi_i}}\chi(|y|) \frac{ 8}{(\rho_i^2 +|y|^2)^2} \frac {2(\rho_i^2-|y|^2)}{\rho_i^2 + |y|^2}\Phi\circ y^{-1}_{\xi_i}(y) dy\\
				&&- \frac 1 {|\Sigma|_g} \sum_{i=1}^m
				\int_{\Sigma} \Phi dv_g \int_{B_{2r_0}^{\xi_i}}	\chi(|y|) \frac{ 8}{(\rho_i^2 +|y|^2)^2} \frac{2(\rho_i^2-|y|^2)}{\rho_i^2 + |y|^2} dy \\
				&=& 16\sum_{i=1}^m \int_{\mathbb{R}^2_{+}} \frac {1-|z|^2}{(1+|z|^2)^3} {\Phi}_i(z) dz -\frac {16} {|\Sigma|_g}\sum_{i=1}^m  
				\int_{\Sigma} \Phi dv_g  \int_{\RR_i} \frac {1-|z|^2}{(1+|z|^2)^3} dz+o(1)\\
				&=& 16\sum_{i=1}^m \int_{\RR_i} \frac {1-|z|^2}{(1+|z|^2)^3} {\Phi}_i(z) dz+o(1),
			\end{array}
		\end{equation*}
		where we applied the fact $ \int_{\RR_i} \frac{1-|z|^2}{(1+|z|^2)^3} dz=0.$
	Lebesgue's dominated convergence theorem yields that 
		$\sum_{i=1}^m \int_{\RR_i} \frac {1-|z|^2}{(1+|z|^2)^3} {\Phi}_{i,\infty}(z) dz=0.$
		Claim~\ref{c3.1} is concluded.  \\
		{\it Proof of Claim~\ref{c3.2}. } 
		We will apply the following representation formula to get the estimate,
		\begin{equation}~\label{represent}
			\Phi(\xi_i)= \frac{1}{|\Sigma|_g} \int_{\Sigma} \Phi dv_g + \int_{\Sigma} G^g(x,\xi_i)(-\Delta_g)\Phi(x) dv_g + \int_{\partial \Sigma} G^g(x,\xi_i) \partial_{ \nu_g }\Phi(x) ds_g. 
		\end{equation}
		For $\xi_i\in\intsigma$, the calculation is simpler, so here we just show the case for $\xi_i\in \partial \Sigma$.
		Recall that $G^g(x,\xi_i)= -\frac  4 {\varrho(\xi_i)}\chi_i\log |y_{\xi_i}(x)|+H^g(x,\xi_i)$. Applying Lemma~\ref{lem0}, we derive that 
		\begin{equation}~\label{est_kphi}
			\begin{array}{lll}
				\int_{U_{2r_0}(\xi_i)} \mathcal{K}\Phi dv_g
				&=& \int_{\frac 1 {\rho_i} B_{2r_0}^{\xi_i}}  \tilde{\mathcal{K}}_i{\Phi}_i dz \\
				&=&  \int_{\mathbb{R}_i} \frac{ 8}{(1+|z|^2)^2} {\Phi}_{i,\infty} dz+{\mathcal{O}}(\rho) \\
				&=& a_{i0} \int_{\mathbb{R}_i} \frac{ 16(1-|z|^2)}{(1+|z|^2)^3} dz+{\mathcal{O}}(\rho)={\mathcal{O}}(\rho).  
			\end{array}
		\end{equation}
		Considering that $(-\Delta_g+\beta)\Phi= \mathcal{K}\Phi+h$, we have to calculate the following integrals: 
		\begin{equation*}
			\begin{array}{lcl}
				&&\int_{\Sigma} G^g\left(x, \xi_i\right) \mathcal{K}\Phi d v_g\\
				&=&\int_{U_{2r_0}(\xi_i)}\left(-\frac{4}{\varrho(\xi_i)}\chi_i \log |y_{\xi_i}(x)|+H^g\left(x, \xi_i\right)\right)\mathcal{K} \Phi d v 
				_g(x)	\\
				&&+ \sum_{l\neq i} 	\int_{U_{2r_0}(\xi_l)} G^g\left(x, \xi_i\right) \mathcal{K}\Phi d v_g
				+ \sum_{\Sigma\setminus\cup_{i=1}^m U_{2r_0}(\xi_i)} G^g(x,\xi_i) \mathcal{K}\Phi d v_g\\
				&\stackrel{\eqref{sumex1}}{=}& \int_{U_{2r_0}(\xi_i)}\left(-\frac{4}{\varrho(\xi_i)} \chi_i\log |y_{\xi_i}(x)|+H^g\left(\xi_i, \xi_i\right)\right)\mathcal{K} \Phi d v 
				_g(x)	\\
				&&+ \sum_{l\neq i} G^g\left(\xi_l, \xi_i\right) 	\int_{U_{2r_0}(\xi_l)} \mathcal{K}\Phi d v_g +o(1)
				\\
				&=& -\frac{4}{\varrho(\xi_i)}\log \rho_i \int_{U_{2r_0}\left(\xi_i\right)} \mathcal{K}\Phi dv_g+\int_{\mathbb{R}_i}\left(-\frac{4}{\varrho(\xi_i)} \log |z|+H^g\left(\xi_i, \xi_i\right)\right) \frac{8}{\left(1+|z|^2\right)^2} \Phi_{i, \infty} d z \\
				&& +\sum_{l\neq i} G^g\left(\xi_l, \xi_i\right) \int_{U_{2r_0}(\xi_l)}  \mathcal{K}\Phi dv_g+ o(1) \\
				&\stackrel{\eqref{est_kphi}}{=} & 
				4a_{i0}+o(1),
			\end{array}
		\end{equation*}
		where we applied $-\int_{\mathbb{R}_i} \frac{8 \log( |z|)(1-|z|^2)}{(1 +|z|^2)^3} dz =\frac {\varrho(\xi_i)}{2};$ and 
		\begin{equation*}
			\begin{array}{ll}
				\left|\int_{\Sigma} G^g\left(x, \xi_i\right) h d v_g\right| &= \left|\left(\int_{\Sigma\setminus U_{\rho_i}(\xi_i)}+ \int_{U_{\rho_i}(\xi_i)} \right)G^g\left(x, \xi_i\right) h d v_g\right| \\
				&	\leq C|\log \rho_i| \int_{\Sigma}|h| d v_g+\frac{\|h\|_*}{\rho_i^2}\left|\int_{U_{\rho_i}\left(\xi_i\right)} G^g\left(x, \xi_i\right) d v_g\right| \\
				&\leq C|\log\rho_i| \|h\|_* + C\|h\|_* (\sup_{U_{\rho_i}(\xi_i)}|H^g(\cdot ,\xi_i)|+|\log \rho_i|+1)\\
				&\leq C|\log \rho_i|\|h\|_*=o(1).
			\end{array}
		\end{equation*}
		We observe that  $\partial_{\nu_g}\Phi= 0$ on $\partial\Sigma$, so the boundary term is vanish.  
		Combining the estimates above, 
		\[\Phi_i(0)=\Phi(\xi_i)= \frac 1 {|\Sigma|_g}\int_{\Sigma} \Phi dv_g + 4a_{i0}+o(1)= \mathfrak{c}_0+ 4a_{i0}+o(1) \]
		as $\rho\rightarrow 0.$
	\end{proof}
	\begin{altproof}{Proposition~\ref{propolin}}
		We observe that there exists a constant $C>0$ such that $\|\Delta_g PZ_{ij}\|_*\leq C$  for any $i=1,2,\cdots, m$ and $j=0,\cdots, \i(\xi_i),$
		and 
		\[ \left\| (\lambda-\lambda_{k,m}) \frac{Ve^W}{\int_{\Sigma} Ve^W dv_g } \left(\phi- \frac{\int_{\Sigma} Ve^W \phi dv_g }{\int_{\Sigma} Ve^W dv_g }\right)\right\|_*=\mathcal{O}( |\lambda-\lambda_{k,m}|\|\phi\|_{\infty}).\]
		Applying Lemma~\ref{lemlin1}, for $\lambda$ close to $\lambda_{k,m}$, we have 
		\begin{equation}\label{eq:phi_inf_h}
			\|\phi\|_{\infty}\leq C |\log \rho|\left( \|h\|_*+|c_0|+ \sum_{i=1}^m \sum_{j=1}^ {\i(\xi_i)} |c_{ij}|\right),
		\end{equation}
		for any solution $\phi$ to the problem~\eqref{linprob}. 
		We apply $PZ_{ij}$ as a test function to \eqref{linprob}. It follows that 
		\begin{equation}\label{eq:etst_PZ_ij}
			\int_{\Sigma} \phi L(PZ_{ij}) dv_g =\int_{\Sigma} h PZ_{ij} dv_g  -c_0 \sum_{i=1}^m \int_{\Sigma}\Delta_g PZ_{i0} Z_{ij}dv_g- \sum_{l=1}^m \sum_{t=1}^{J_{l}} c_{lt}\int_{\Sigma}\Delta_g PZ_{lt}Z_{ij} dv_g.
		\end{equation}
		From Lemma~\ref{lemb3} and Lemma~\ref{lem4}, for any $i,l=1,2,\cdots,m$, we have the following estimates:
		\begin{eqnarray}
			&	\quad	PZ_{ij} =\chi_i Z_{ij}+\mathcal{O}(\rho). \text{ for any } j=1,\cdots,\i(\xi_i); \label{expz}\\ &PZ_{i0}(x)=\chi_{i}(x) \left(Z_{i0}(x) +2\right) +{\mathcal{O}}(\rho^2|\log \rho |)\label{expz0};\\
			&	\quad\int_{\Sigma} Z_{ij} \Delta_g PZ_{lt} dv_g = -\frac{4\varrho(\xi_i)}{3}\delta_{il}\delta_{jt} +\mathcal{O}(\rho)\label{inproz},
		\end{eqnarray}
		for any	$j=0,\cdots, \i(\xi_i), t=0,\cdots, \i(\xi_l).$
		
		Lemma~\ref{lem0} and~\eqref{sumex1} derive that 
		\begin{equation}~\label{zij}
			\frac{\int_{\Sigma} Ve^WPZ_{ij} dv_g}{\int_{\Sigma} Ve^W dv_g }=\mathcal{O}(\rho), \text{for any }j=1,\cdots, \i(\xi_i).
		\end{equation}
		Applying~\eqref{expz},~\eqref{zij} and~\eqref{int13},
		\begin{equation*}
			\begin{array}{lll}
				L_{\lambda_{k,m}}(PZ_{ij})&=& (-\Delta_g	)PZ_{ij}  - \frac{\lambda_{k,m}  Ve^W}{ \int_{\Sigma} Ve^W dv_g }\left(PZ_{ij}-  \frac{\int_{\Sigma} Ve^W PZ_{ij} dv_g }{ \int_{\Sigma} Ve^W dv_g }\right)\\
				&=& \chi_ie^{U_i}Z_{ij}\left( e^{-\varphi_i\circ y_{\xi_i}^{-1}} -1 \right)+\beta \chi_i Z_{ij}+\mathcal{O}(\rho\sum_{i=1}^m \chi_i e^{U_i}+\rho).
			\end{array}
		\end{equation*}
		It follows that 
		\begin{equation}\label{eq:starPZ_ij}
			\begin{array}{lcl}
				\|L(PZ_{ij})\|_*&\leq& \|L_{\lambda_{k,m}}(PZ_{ij})\|_*+C|\lambda-\lambda_{k,m}|\cdot\|PZ_{ij}\|_{\infty}\\
				&\leq& \mathcal{O}(\rho^{1-\kappa/2}+ \rho^2|\log \rho|)=\mathcal{O}(\rho^{1-\kappa/2}). 
			\end{array}
		\end{equation}
		Similarly, \eqref{sumex1} implies that 
		\begin{equation}
			\label{zi0}
			\frac{\int_{\Sigma} Ve^WPZ_{i0} dv_g}{\int_{\Sigma} Ve^W dv_g }=\frac{2\varrho(\xi_i)}{\lambda_{k,m}}+\mathcal{O}(\rho).
		\end{equation}
		Applying~\eqref{expz0},~\eqref{zi0} and~\eqref{int13}, we have 
		\begin{equation*}
			\begin{array}{lll}
				L_{\lambda_{k,m}}(PZ)&=& \sum_{i=1}^m (-\Delta_g+\beta)PZ_{i0}  - \frac{\lambda_{k,m}  Ve^W}{ \int_{\Sigma} Ve^W dv_g }\left(PZ_{ij}-  \frac{\int_{\Sigma} Ve^W PZ_{ij} dv_g }{ \int_{\Sigma} Ve^W dv_g }\right)\\
				&=& \sum_{i=1}^m\left( \chi_ie^{U_i}\left(e^{-\varphi_i\circ y^{-1}_{\xi_i}} Z_{i0}-PZ_{i0} \right)+ \frac{2\varrho(\xi_i)}{\lambda_{k,m}}\sum_{l=1}^m \chi_le^{U_l}\right)  \\
				&&+\beta\sum_{i=1}^m \chi_i (Z_{i0}+2)+ \mathcal{O}(\rho \sum_{i=1}^m \chi_i e^{U_i}+\rho)\\
				&=&\sum_{i=1}^m \chi_ie^{U_i}\left(e^{-\varphi_i\circ y^{-1}_{\xi_i}} Z_{i0}-PZ_{i0}+2 \right)  +\beta\sum_{i=1}^m \chi_i (Z_{i0}+2)\\
				&&+ \mathcal{O}(\rho \sum_{i=1}^m \chi_i e^{U_i}+\rho). 
			\end{array}
		\end{equation*}
		It follows that 
		\begin{equation}\label{eq:starPZ}
			\begin{array}{lcl}
				\|L(PZ\|_*&\leq& \|L_{\lambda_{k,m}}(PZ)\|_*+C|\lambda-\lambda_{k,m}|\cdot\|PZ\|_{\infty}\\
				&\leq& \mathcal{O}(\rho^{1-\kappa/2}+ \rho^2|\log \rho|)=\mathcal{O}(\rho^{1-\kappa/2}). 
			\end{array}
		\end{equation}
		We take $\alpha_*=1-\kappa/2.$
		The equation~\eqref{eq:etst_PZ_ij} combined with~\eqref{inproz},~\eqref{eq:starPZ} and~\eqref{eq:starPZ_ij} deduces that 
		\begin{eqnarray*}
			|c_0|+	\sum_{l=1}^m \sum_{t=1}^{\i(\xi_i)} |c_{ij}|&\leq& C\|h\|_* + \rho^{\alpha_*}\mathcal{O}(\|\phi\|_{\infty})+  \rho \mathcal{O}\left(\sum_{l=1}^m \sum_{t=1}^{\i(\xi_i)} |c_{ij}| \right)\\
			&\stackrel{\eqref{eq:phi_inf_h}}{\leq}& C \|h\|_*+ \rho^{\alpha_*} |\log \rho| \mathcal{O}\left(\sum_{l=1}^m \sum_{t=1}^{\i(\xi_i)} |c_{ij}| \right). 
		\end{eqnarray*}
		Thus we obtain that 
		\begin{equation}~\label{est1}
			\|\phi\|_{\infty}\leq \mathcal{O}\left(|\log \rho| \|h\|_*\right) \text{ and } |c_0|+\sum_{l=1}^m \sum_{t=1}^{\i(\xi_i)} |c_{ij}|\leq \mathcal{O}(\|h\|_*).
		\end{equation}
		We define $\mathrm{H}$ as a subspace of $\overline{\mathrm{H}}^1$ which is orthogonal to  $\la PZ, PZ_{ij}: i=1,\cdots,m, j=1,\cdots, \i(\xi_i)\ra$, i.e. 
		\[\mathrm{H}:=\left\{\begin{array}{ll}
		\phi\in \overline{\mathrm{H}}^1:& \int_{\Sigma} \phi\Delta_g PZ_{ij} dv_g = 0 \text{ for any } i=1,2,\cdots,m, j=1,\cdots, \i(\xi_i);\\ &\int_{\Sigma} \phi\Delta_g PZ dv_g = 0 
		\end{array} \right\}.\]
		To solve~\eqref{linprob} is equivalent to find $\phi\in \mathrm{H}$ such that 
		\[ \int_{\Sigma}\la\nabla \phi, \nabla\psi\ra _g dv_g +\beta \int_{\Sigma} \phi \psi dv_g = \int_{\Sigma} \left( \frac{\lambda Ve^W }{ \int_{\Sigma} Ve^W dv_g  }\left(\phi-   \frac{\int_{\Sigma} Ve^W\phi dv_g  }{ \int_{\Sigma} Ve^W dv_g  } \right)+h \right)\psi dv_g,   \]
		for any $\psi\in \mathrm{H}.$
		The prior estimate~\eqref{est1} implies that 
		for any $h=0$,  the solution to the problem~\eqref{linprob} must be trivial, i.e. $\phi=0$, $c_0$ and $c_{ij}=0$ (for any $i=1,\cdots,m$ and $j=1,\cdots, \i(\xi_i)$). 
		By Riesz's representation formula and Fredholm's alternative,
		there is a  unique solution $T(h):=\phi \in \overline{\mathrm{H}}^1\text{
			and }\tilde{c}\in \RR^{m+k+1}$ solving the problem~\eqref{linprob}. 
		Then the linear operator
		$T:$ \[\left\{h\in L^{\infty}: \int_{\Sigma} h dv_g \text{ and } \|h\|_*<+\infty\right\} \rightarrow \left\{ \phi \in L^{\infty}(\Sigma): \int_{\Sigma}\phi dv_g \text{ and } \|\phi\|_{\infty}<+\infty\right\},\]
		is continuous. 
		
		Next, we will show that the operator $T$ and coefficients $c_{ij}$ are differentiable concerning $\xi$ and $\rho$. 
		Denote that $\iota=\rho \text{ or } (\xi_i)_j$ for any $i=1,2,\cdots,m,j=1,\cdots, \i(\xi_i)$ and $X= \partial_{\iota } \phi$. 
		Differentiate \eqref{linprob} with $\iota$, formally, 
		\[ L(X)=   L_0-d_0\Delta_gPZ  -\sum_{i=1}^m \sum_{j=1}^{\i(\xi_i)} d_{ij} \Delta_g PZ_{ij},\]
		where \begin{eqnarray*}
			L_0&= & -\partial_{\iota}\left(\frac{\lambda V e^W}{\int_{\Sigma} V e^W d v_g}\right) \phi+\partial_{\iota}\left(\frac{\lambda V e^W}{\left(\int_{\Sigma} V e^W d v_g\right)^2}\right) \int_{\Sigma} V e^W \phi d v_g \\
			& &+\frac{\lambda V e^W}{\left(\int_{\Sigma} V e^W d v_g\right)^2} \int_{\Sigma} V e^W \partial_{\iota} W \phi d v_g-c_0\partial_{\iota}\left(\Delta_g P Z\right) -\sum_{i=1}^m \sum_{j=1}^{\i(\xi_i)}  c_{i j} \partial_{\iota}\left(\Delta_g P Z_{i j}\right),
		\end{eqnarray*}   
		$d_0= \partial_{\iota}c_0$ and $d_{ij}= \partial_{\iota} c_{ij}.$ In addition, the orthogonal properties imply that 
		\[ \int_{\Sigma} X \Delta_g P Z d v_g=-\int_{\Sigma} \phi \partial_{\iota}\left(\Delta_g P Z\right) d v_g\text{ and }\int_{\Sigma} X \Delta_g P Z_{i j} d v_g=-\int_{\Sigma} \phi \partial_{\iota}\left(\Delta_g P Z_{i j}\right) d v_g. \]
		Consider $Y=X+b_0PZ+ \sum_{i=1}^m \sum_{j=1}^{\i(\xi_i)} b_{ij} PZ_{ij}$, where  $b_0\text{ and }b_{ij}$ are chosen to satisfy the orthogonal properties,  
		\[ \int_{\Sigma} Y \Delta_g PZd v_g=0\text{ and }\int_{\Sigma} Y \Delta_g P Z_{i j} d v_g=0.\]
		Since $\|\partial_{\iota}(\Delta_gPZ\|_*+\sum_{i=1}^m \sum_{j=1}^{\i(\xi_i)}\|\partial_{\iota}(\Delta_gPZ_{ij})\|_*\leq \mathcal{O}\left(\frac 1 {\rho}\right)$,  we have $|b_0|+\sum_{i=1}^m \sum_{j=1}^{\i(\xi_i)} |b_{ij}|\leq  \frac C{\rho}\|\phi\|_{\infty}\leq C \frac{|\log \rho|}{\rho} \|h\|_*$, by~\eqref{est1}.
		Then $Y$ solves the problem 
		\begin{equation}
			\begin{cases}
				L(Y)=f -d_0\Delta_gPZ-\sum_{i=1}^m \sum_{j=1}^{\i(\xi_i)} d_{ij} \Delta_g PZ_{ij} & \text{ in }\Sigma\\
				\partial_{ \nu_g } Y=0 & \text{ on } \partial \Sigma\\
				\int_{\Sigma} Y \Delta_g P Z_{i j} d v_g=0, 	\int_{\Sigma} Y \Delta_g P Z d v_g=0 &\text{ for any } i=1,\cdots,m,j=1,\cdots, \i(\xi_i)
			\end{cases},
		\end{equation}
		where $f=L_0 +b_0L(PZ)+ \sum_{i=1}^m \sum_{j=1}^{\i(\xi_i)} b_{ij}L(PZ_{ij}).$
		{ We observe that $\|\mathcal{K}\|_*\leq C$, $\|\partial_{\iota} W\|_{\infty}+ \|\partial_{\iota}\mathcal{K}\|_*\leq \frac {C}{\rho}$ and  $\|\partial_{\iota}\left( \frac{\mathcal{K}}{\int_{\Sigma} Ve^W dv_g }\right)\|_*\leq \frac{C}{\rho}\left(\int_{\Sigma} Ve^W dv_g \right)^{-1}$. } Then 
		\begin{eqnarray*}
			\|f\|_*&\leq& \|L_0\|_*+ |b_0|\|L(PZ)\|_*+ \sum_{i=1}^m \sum_{j=1}^{\i(\xi_i)} |b_{ij}|\|L(PZ_{ij})\|_*\\
			&\leq&  \frac{C}{\rho}\left(\|\phi\|_{\infty}+|c_0|+\sum_{i=1}^m \sum_{j=1}^{\i(\xi_i)}|c_{ij}|\right)\\
			&\stackrel{\eqref{est1}}{\leq}& C \frac{|\log \rho|}{\rho}\|h\|_*.
		\end{eqnarray*}
		Applying~\eqref{est1} again, we have 
		\begin{eqnarray*}
			\|\partial_{\iota} \phi\|_{\infty}&\leq& C|\log \rho| \|f\|_* +|b_0|\|PZ\|_{\infty}  +\sum_{i=1}^m \sum_{j=1}^{\i(\xi_i)} |b_{ij}|\|PZ_{ij}\|_{\infty} 
			\leq C\frac{|\log \rho|^2}{\rho} \|h\|_*. 
		\end{eqnarray*}
		We differentiate $\partial_{\rho}\phi$ with respect to  $\rho$. By the same method to get estimate for $\partial_{\rho}\phi$, we can derive that there exists a constant $C>0$ such that 
		\[\|\partial_{\rho}^2\phi\|_{\infty} \leq C \frac{|\log \rho|^3}{\rho^2}\|h\|_{*}\text{ and }\|\nabla_{\xi}\partial_{\rho}\phi\|_{\infty} \leq C \frac{|\log \rho|^3}{\rho^2}\|h\|_{*}.\]
	\end{altproof}
	\par 
	\subsection{The non-linear problem}
	We establish that the linear operator $L$ exhibits partial invertibility, as shown in Proposition~\ref{propolin}. This enables us to proceed with the analysis of the corresponding non-linear problem.
	\begin{proposition}~\label{nonlinear}
		For any $\delta>0$ sufficiently small, 
		there exists $\rho_0>0$ such that for any $\rho\in (0, \rho_0), \xi \in M_{\xi^*,\delta}$, the problem
		\begin{equation}
			~\label{nonlinearprob}
			\begin{cases}L(\phi)=[R+N(\phi)]-c_0\Delta_gPZ-\sum_{i=1}^m \sum_{j=1}^{\i(\xi_i)} c_{i j} \Delta_g P Z_{i j} & \text { in } \intsigma \\ 
				\partial_{ \nu_g }\phi=0 & \text{ on } \partial \Sigma\\
				\int_{\Sigma} \phi \Delta_g P Z_{i j} d v_g=0, \int_{\Sigma}\phi \Delta_g PZd v_g=0  & \forall i=1, \cdots, m, j=1,\cdots,\i(\xi_i)\end{cases}
		\end{equation}
		has a unique solution $\phi\in \overline{\mathrm{H}}^1\cap W^{2,2}(\Sigma)$ and   $ \tilde{c}\in \RR^{m+k+1}$ where $i=1,2,\cdots,m, j=1,\cdots, \i(\xi_i).$  Moreover, 
		the map $(\rho,\xi)\mapsto(\phi(\rho,\xi), \tilde{c}(\rho,\xi))$ is  $C^1$-differentiable with respect to  $\xi$ and $C^2$-differentiable with respect to  $\rho$  satisfying 
		\begin{eqnarray}~\label{est_phi}
			\|\phi\|_{\infty} \leq C \left(\rho|\log \rho||\nabla \cF_{k,m}(\xi)|_g+\rho^{2-\kappa}|\log\rho|^2\right),\\
			\|\partial_{\rho}\phi\|_{\infty}\leq C \left(|\log\rho|^2 |\nabla\cF_{k,m}(\xi)|_g+\rho^{1-\kappa}|\log\rho|^3\right), \label{non1est_rho}\\ 
			\|\nabla_{\xi}\phi\|_{\infty} \leq C \left(|\log\rho|^2 |\nabla\cF_{k,m}(\xi)|_g+\rho^{1-\kappa}|\log\rho|^3\right), \label{non1est}\\
			\|\partial^2_{\rho}\phi\|\leq C\left(\frac{|\log\rho|^3}{\rho} |\nabla\cF_{k,m}(\xi)|_g+\rho^{-\kappa}|\log\rho|^4\right), \label{non2_rho}\\  \|\partial_{\xi}\partial_{\rho}\phi\|\leq C\left(\frac{|\log\rho|^3}{\rho} |\nabla\cF_{k,m}(\xi)|_g+\rho^{-\kappa}|\log\rho|^4\right) , \label{non2_mix}	
		\end{eqnarray}
		for some constant $C>0$ and $\kappa\in (0,\alpha_0).$
	\end{proposition}
	\begin{proof}
		Following the approach presented in Lemma 4.1 of  \cite{del_pino_singular_2005}, we aim to solve the problem~\eqref{nonlinearprob} by finding a fixed point for the operator $\mathcal{A}(\phi):= T(R+N(\phi))$ on the space
		\[ \mathfrak{F}_{\varpi}:=\left\{ \phi\in C(\Sigma): \int_{\Sigma}\phi dv_g=0,  \|\phi\|_{\infty}\leq \varpi \left(\rho|\log \rho||\nabla \cF_{k,m}(\xi)|_g+\rho^{2-\kappa}|\log\rho|^2 \right)\right\},\]
		for some constant $\varpi>0.$
		For any $\phi\in \mathfrak{F}_{\varpi}$, we can utilize Proposition~\ref{propolin} to deduce that $$\|\mathcal{A}(\phi)\|_{\infty}\leq \|T\|( \|N(\phi)\|_*+ \|R\|_*)\leq C|\log \rho|(\|N(\phi)\|_*+ \|R\|_*).$$
		Given $\|\mathcal{K}\|_*\leq C$ and $|e^s-1|\leq e^{|s|}|s|$, 
		$\|N(\phi)\|_*\leq C\|\phi\|_{\infty}^2$. Consequently, we obtain
		\[ \|\mathcal{A}(\phi)\|_{\infty}\leq  C_0\left(|\log \rho| \|\phi\|_{\infty}^2+\rho|\log \rho| |\nabla \cF_{k,m}(\xi)|_g+\rho^{2-\kappa}|\log\rho|^2\right),\]
		combined with Lemma~\ref{est_R}. 
		
		Take $\varpi=2C_0$, then $ \mathcal{A}(\phi)\in \mathfrak{F}_{\varpi}$, for any $\phi\in \mathfrak{F}_{\varpi}.$
		
		Applying Lemma~\ref{lemb1}, \eqref{eq:B_chi1} and~\eqref{eq:veW_detail}, for any $\phi_1,\phi_2\in \mathfrak{F}_{\varpi}$, denote $\phi_{\theta}=\theta\phi_1+(1-\theta)\phi_2$, for same $\theta\in (0,1)$. 
		\[
		\begin{aligned}
			&\left \| \frac{\lambda V e^{W+\phi_1} }{\int_{\Sigma} V e^{W+\phi_1} d v_g}-\frac{\lambda V e^{W+\phi_2} }{\int_{\Sigma} V e^{W+\phi_2} d v_g}
			- \frac{\lambda V e^{W}(\phi_1-\phi_2)}{\int_{\Sigma} V e^{W} d v_g}+ 
			\frac{\lambda V e^{W}  \int_{\Sigma} V e^{W} (\phi_1-\phi_2) d v_g}{\left(\int_{\Sigma} V e^{W} d v_g\right)^2}\right\|_*\\
			&\leq  \left \|  \frac{\lambda V e^{W} e^{\phi_{\theta}} (\phi_1-\phi_2)}{\int_{\Sigma} V e^{W} e^{\phi_{\theta}} d v_g}-\frac{\lambda V e^{W} (\phi_1-\phi_2)}{\int_{\Sigma} V e^{W} d v_g} \right\|_*\\
			&+ \left \|  - \frac{\lambda V e^{W}  \int_{\Sigma} V e^{W} (\phi_1-\phi_2) d v_g}{\left(\int_{\Sigma} V e^{W} d v_g\right)^2} +
			\frac{\lambda V e^{W} e^{\phi_{\theta}}  \int_{\Sigma} V e^{W} e^{\phi_{\theta}} (\phi_1-\phi_2) d v_g}{\left(\int_{\Sigma} V e^{W} e^{\phi_{\theta}} d v_g\right)^2}   \right\|_*\\
			&\leq C (\|\phi_1\|_{\infty}+ \|\phi_2\|_{\infty} ) \|\phi_1-\phi_2\|_{\infty}.
		\end{aligned}
		\]
		It follows that 
		\begin{eqnarray}~\label{est_N_s}
			\|N(\phi_1)-N(\phi_2)\|_*\leq C_1(\|\phi_1\|_{\infty}+ \|\phi_2\|_{\infty}) \|\phi_1-\phi_2\|_{\infty}. 
		\end{eqnarray}
		Let $\rho'_0>0$ be sufficiently small such that $$2C_1\varpi \left(\rho|\log \rho|^2|\nabla \cF_{k,m}(\xi)|_g+\rho^{2-\kappa}|\log\rho|^3  \right)\leq \frac 1 2,$$ for any $\rho \in (0,\rho'_0).$
		By combining the estimate of $T$ from Proposition~\ref{propolin}, we can deduce that
		\begin{equation}
			\|\mathcal{A}(\phi_1)-\mathcal{A}(\phi_2)\|_{\infty}
			\leq C_1|\log \rho|(\|\phi_1\|_{\infty}+ \|\phi_2\|_{\infty}) \|\phi_1-\phi_2\|_{\infty}\leq \frac 1 2 \|\phi_1-\phi_2\|_{\infty},
		\end{equation}
		Hence, $\mathcal{A}$ is a contraction map from $\mathfrak{F}_{\varpi}$ into itself. There exists a unique fixed point $\phi^{\rho}_{\xi}\in \mathfrak{F}_{\varpi}$. 
		By the $L^p$ theory, we can conclude that  $\phi^{\rho}_{\xi} \in \overline{\mathrm{H}}^1\cap W^{2,2}(\Sigma
		)$. Therefore, 
		$\phi^{\rho}_{\xi}$ solves~\eqref{nonlinearprob}. 
		Consider the map $F(\xi,\phi,\tilde{c}):= \mathcal{A}(\phi)-\phi$. There is a small constant $\rho_0\in (0,\rho'_0)$ which only depends on $\delta$ such that for any fixed $\rho\in (0,\rho_0)$ and $\xi\in M_{\xi^*,\delta}$, we have 
		$F(\xi,\phi^{\rho}_{\xi},\tilde{c}(\rho,\xi))=0$. Applying the  implicit function theorem,, 
		$(\rho,\xi) \mapsto (\phi^{\rho}_{\xi},\tilde{c}(\rho,\xi))$ is $C^1$-mapping with respect to  $\xi$ and $C^2$-mapping with respect to  $\rho$.
		Differentiating $F(\xi,\phi,\tilde{c})=0$ with respect to  $\iota=\rho \text{ or }(\xi_i)_j (\text{for }i=1,2,\cdots,m, j=1,\cdots,\i(\xi_i))$ 
		\[ \partial_{\iota}\phi= (\partial_{\iota} T) ( R+N(\phi))+ T(\partial_{\iota}R+\partial_{\iota} N(\phi)).\]
		By Lemma~\ref{est_R} and \eqref{est_N_s}, 
		\begin{eqnarray*}
			\|\partial_{\iota} T(R+N(\phi))\|_{\infty}
			&\leq& C\frac{|\log \rho|^2}{\rho}( \|R\|_*+ \|N(\phi)\|_*)\\
			&=& C\left(|\log \rho|^2|\nabla\cF_{k,m}(\xi)|_g+ \rho^{1-\kappa}|\log \rho|^3\right).
		\end{eqnarray*}
		\begin{eqnarray*}
			\partial_{\iota} N(\phi)&= & N(\phi) \partial_{\iota} W+\lambda\left(\frac{V e^{W+\phi}}{\int_{\Sigma} V e^{W+\phi} d v_g}-\frac{V e^W}{\int_{\Sigma} V e^W d v_g}\right) \partial_{\iota} \phi \\
			& &-\lambda\left(\frac{V e^{W+\phi} \int_{\Sigma} V e^{W+\phi} \partial_{\iota} W d v_g}{\left(\int_{\Sigma} V e^{W+\phi} d v_g\right)^2}-\frac{V e^W \int_{\Sigma} V e^W \partial_{\iota} W d v_g}{\left(\int_{\Sigma} V e^W d v_g\right)^2}\right. \\
			&& -\frac{V e^W \phi \int_{\Sigma} V e^W \partial_{\iota} W d v_g}{\left(\int_{\Sigma} V e^W d v_g\right)^2}-\frac{V e^W \int_{\Sigma} V e^W \partial_{\iota} W \phi d v_g}{\left(\int_{\Sigma} V e^W d v_g\right)^2} \\
			&& \left.+2 \frac{V e^W\left(\int_{\Sigma} V e^W \partial_{\iota} W d v_g\right)\left(\int_{\Sigma} V e^W \phi d v_g\right)}{\left(\int_{\Sigma} V e^W d v_g\right)^3}\right)\\
			&&-\lambda\left(\frac{V e^{W+\phi} \int_{\Sigma} V e^{W+\phi} \partial_{\iota} \phi d v_g}{\left(\int_{\Sigma} V e^{W+\phi} d v_g\right)^2}-\frac{V e^W \int_{\Sigma} V e^W \partial_{\iota} \phi d v_g}{\left(\int_{\Sigma} V e^W d v_g\right)^2}\right)
		\end{eqnarray*}
		By Lemma~\ref{lem:expansion1st} and Lemma~\ref{expansion_rho},  $\|\partial_{\iota } W\|_{\infty}\leq C/\rho.$
		\begin{eqnarray*}
			\|\partial_{\iota} N(\phi)\|_*&\leq& C(\|\partial_{\iota}W\|_{\infty}\|\phi\|_{\infty}^2+ \|\partial_{\iota}\phi\|_{\infty}\|\phi\|_{\infty})\\
			&\leq & \mathcal{O}\left( \rho|\log\rho|^2|\nabla\cF_{k,m}(\xi)|^2_g +\rho^{3-2\kappa}|\log\rho|^4\right) +o\left( \frac{\|\partial_{\iota}\phi\|_{\infty}}{|\log \rho|}\right). 
		\end{eqnarray*}
		Lemma~\ref{lem0} yields that
		\begin{eqnarray*}
			\partial_{(\xi_i)_j}\left( \int_{\Sigma}\chi_l e^{-\varphi_q\circ y_{\xi_l}^{-1}} e^{U_q} dv_g \right) 
			&=&-\varrho(\xi_l) \frac{ r_0^2\rho_l^2}{(r_0^2+\rho_l^2)^2} \partial_{(\xi_i)_j}\log \tau_l + \mathcal{O}(\rho^2)=\mathcal{O}(\rho^2),
		\end{eqnarray*}
		and 
		\begin{eqnarray*}
			\partial_{\rho}\left( \int_{\Sigma}\chi_l e^{-\varphi_q\circ y_{\xi_l}^{-1}} e^{U_q} dv_g \right) 
			&=&-\varrho(\xi_l) \frac{ 2r_0^2\tau_l\rho}{(r_0^2+\rho_l^2)^2} + \mathcal{O}(\rho)=\mathcal{O}(\rho),
		\end{eqnarray*}
		for any $q=1,\cdots,m.$
		We observe that 
		$e^{-\varphi_i\circ y_{\xi_i}^{-1}(x)}=1+ \mathcal{O}( |y_{\xi_i}(x)|^2)$ and 
		$\partial_{\iota}( \chi_i e^{-\varphi_i\circ y_{\xi_i}^{-1}(x)})= \mathcal{O}( |y_{\xi_i}(x)|).$
		We deduce that 
		\begin{eqnarray}\label{eq:delta_W}
			-\Delta_g \partial_{\iota} W= \sum_{i=1}^m \chi_i e^{U_i} \partial_{\iota} U_i +\mathcal{O}(\rho^{1-\kappa}),
		\end{eqnarray} 
		where $\mathcal{O}$ is  estimated in $\|\cdot\|_*$-norm. 
		We observe that 	\begin{equation*}
			\begin{array}{lcl}
				\partial_{\iota}R &=&\Delta_g \partial_{\iota}W + \frac{\lambda Ve^W}{\int_{\Sigma}Ve^W dv_g} \left(\partial_{\iota}W- \frac{\int_{\Sigma} Ve^W \partial_{\iota}W dv_g}{\int_{\Sigma} Ve^W dv_g} \right).
			\end{array}
		\end{equation*}
		By~\eqref{1w},~\eqref{eq:int_ve_pa_xi_normal},~\eqref{sumex1} and~\eqref{eq:delta_W}, we have  for $\iota=(\xi_i)_j$
		\begin{equation}
			\|\partial_{\iota}R\|_{*}\leq \mathcal{O}\left( \sum_{i=1}^m |\nabla\log (\tau_i\circ y_{\xi_i}^{-1})(0)|+\rho^{1-\kappa} \right). 
		\end{equation}
		By~\eqref{w1strho},~\eqref{eq:int_vew_pa_rho_normal},~\eqref{sumex1} and~\eqref{eq:delta_W}, we have  for $\iota=\rho$
		\begin{equation}
			\|\partial_{\iota}R\|_{*}\leq \mathcal{O}\left( \sum_{i=1}^m |\nabla\log (\tau_i\circ y_{\xi_i}^{-1})(0)|+\rho^{1-\kappa} |\log\rho|\right)
		\end{equation}
		Combining all the estimates above, 
		\[ \|\partial_{\iota}\phi\|_{\infty}=
		\mathcal{O}\left(|\log\rho|^2 |\nabla\cF_{k,m}(\xi)|_g+\rho^{1-\kappa}|\log\rho|^3\right). \]
		Thus the estimates~\eqref{non1est} and~\eqref{non1est_rho} hold. Similarly, we can obtain 
		\[ \|\partial_{\rho}^2\phi\|_{\infty}=\mathcal{ O}\left(\frac{|\log\rho|^3}{\rho} |\nabla\cF_{k,m}(\xi)|_g+\rho^{-\kappa}|\log\rho|^4\right) \]
		and 
		\[ \|\partial_{\rho}^2\phi\|_{\infty}=\mathcal{ O}\left(\frac{|\log\rho|^3}{\rho} |\nabla\cF_{k,m}(\xi)|_g+\rho^{-\kappa}|\log\rho|^4\right). \]
	\end{proof}
	\section{Variational reduction and the expansion of the energy functional}~\label{expansion}
	The function $W+\phi^{\rho}_{\xi}$ is a solution of~\eqref{linear} if $c_0 \text{ and }c_{ij}=0$ for any $i=1,2,\cdots, m$ and $j=1,\cdots, \i(\xi_i)$. The following proposition shows that 
	it is equivalent to find critical points of the reduced energy functional $E_{\lambda}(\rho,\xi):= J_{\lambda}(W+\phi^{\rho}_{\xi}).$
	\begin{proposition}~~\label{pro_equi}
		There exists $\rho_0>0$ such that $(\rho,\xi) \in (0,\rho_0)\times( \intsigma^k\times(\partial\Sigma)^{m-k}\setminus \Delta)$ is a critical point of $E_{\lambda}$  if and only if $u=W+\phi^{\rho}_{\xi}$ is a solution of~\eqref{eq:main_eq} with respect to  $\lambda$.
	\end{proposition}
	\begin{proof}
		Assume that $(\xi,\rho)$ is a critical point of the reduced map ${E}_{\lambda}$. There exists $\delta>0$ such that  $\xi$ is in the interior of $ M_{\xi^*,\delta}$ for some $\xi^*\in \intsigma^k\times(\partial\Sigma)^{m-k}\setminus \Delta $ and $\delta>0$. Then for any $\iota= (\xi_{i'})_{j'}$ for $i'=1,\cdots,m, j'=1,\cdots, \i(\xi_{i'})$ and $\rho$,
		\begin{equation}~\label{r1}
			\begin{array}{lll}
				\partial_{\iota}{E}_{\lambda}(\rho, \xi)&=&\int_{\Sigma} \left(L(\phi)-N(\phi)-R \right) \partial_{\iota}\left( W+\phi\right) dv_g(x)\\
				&=& \int_{\Sigma} (-c_0\Delta_gPZ-\sum_{i=1}^m \sum_{j=1}^{\i(\xi_i)} c_{ij}\Delta_gP Z_{ij} )\partial_{\iota}\left( W+\phi\right) dv_g =0.
			\end{array}
		\end{equation}
		Note that the partial derivative of $PU_{i}$ with respect to  $(\xi_{i'})_{j'}$ can be expressed as the sum of two terms:
		\begin{eqnarray*}
			\partial_{(\xi_{i'})_{j'}} PU_{i}= \partial_{(\xi_{i'})_{j'}} PU_{\rho, \xi_i}|_{\rho=\rho_i}+ \frac 1 2 (\rho\partial_{\rho} PU_{\rho, \xi_i})|_{\rho=\rho_i} \partial_{(\xi_{i'})_{j'}}\log \tau_i.
		\end{eqnarray*}
		Lemma~\ref{lem:expansion1st}  combining \eqref{expz} yields that 
		$ PZ_{ij}= \mathcal{O}(\rho)+ \rho_i\partial_{(\xi_i)_j} PU_{\rho, \xi_i}|_{\rho=\rho_i}. $
		Applying~\eqref{inproz}, we have  
		\[ \int_{\Sigma} \left(-c_0\Delta_gPZ-\sum_{i=1}^m \sum_{j=1}^{\i(\xi_i)} c_{ij}\Delta_gP Z_{ij} \right)\partial_{(\xi_{i'})_{j'}} W dv_g = \mathcal{O}\left(|c_0|+\sum_{i=1}^m\sum_{j=1}^{\i(\xi_i)}|c_{ij}|\right)+ c_{i'j'} \frac{4\varrho(xi_{i'})} {3\rho_{i'}},\]
		for any $i'=1,\cdots,m$ and $j'=1,\cdots, \i(\xi_{i'})$. Lemma~\ref{expansion_rho} combining \eqref{expz} yields that
		\begin{equation*}
			\begin{array}{lcl}
				\partial_{\rho} PU_{i}&=&\left. \sqrt{\tau_i}\partial_{\rho} PU_{\rho, \xi_i}|_{\rho=\rho_i}=\sqrt{\tau_i}\left(\chi_i\left( \partial_{\rho}U_{\rho,\xi_i}-\frac 2{\rho}\right) -4\rho F_{\xi_i}+\frac 1 {\rho}h^1_{\rho,\xi_i}+o(\rho)\right)\right|_{\rho=\rho_i}\\
				&=&-\frac 1 {\rho
				}PZ_{i0}+\mathcal{O}\left(\rho|\log\rho|\right).
			\end{array}
		\end{equation*}
		Applying~\eqref{inproz}, we have  
		\[ \int_{\Sigma} \left(-c_0\Delta_gPZ-\sum_{i=1}^m \sum_{j=1}^{\i(\xi_i)} c_{ij}\Delta_gP Z_{ij} \right)\partial_{\rho}W dv_g = \mathcal{O}\left(|c_0|+\sum_{i=1}^m\sum_{j=1}^{\i(\xi_i)}|c_{ij}|\right)+ c_{0} \frac{4\lambda_{k,m}}{3\rho}.\]
		
		Since $\int_{\Sigma} \phi\Delta_gPZ_{ij} dv_g=0$, 
		$\int_{\Sigma} \Delta_gPZ_{ij}\partial_{\iota}\phi dv_g 
		= \int_{\Sigma} \partial_{\iota}\phi \Delta_g\partial_{\iota}PZ_{ij}dv_g.$
		By a direct calculation,  $\| \partial_{\iota}PZ_{ij}\|\leq \| \partial_{\iota}Z_{ij}\|\leq \mathcal{O}(\frac 1{\rho})$. 
		Proposition~\ref{nonlinearprob} deduces that
		\[ \left| \int_{\Sigma} \Delta_gPZ_{ij}\partial_{\iota}\phi dv_g \right|\leq \| \phi\|\| \partial_{\iota}PZ_{ij}\|\leq o\left(\frac 1 {\rho}\right).\]
		Combining the estimates above, for $i'=1,2,\cdots,m, j=1,\cdots, \i(\xi_i),$
		\[c_{i'j'} \frac{3\varrho(\xi_{i'})}{4} +o\left(|c_0|+ \sum_{i=1}^m \sum_{j=1}^{\i(\xi_i)} |c_{ij}|\right)=0,\]
		and 
		\[c_{0} \frac{4\lambda_{k,m}}{3}+o\left(|c_0|+ \sum_{i=1}^m \sum_{j=1}^{\i(\xi_i)} |c_{ij}|\right)=0.\]
		Thus there exists $\rho_0>0$
		sufficiently small, for any $\rho\in (0,\rho_0)$, $\tilde{c}\equiv 0.$
		Consequently, $u=W+ \phi$ solves~\eqref{eq:main_eq} with  respect to $\lambda$.

		Assume that  for $(\xi,\rho)$, there exists a solution of~\eqref{eq:main_eq}, with the form  $\sum_{i=1}^m PU_i +\phi^{\rho}_{\xi}$, it is obvious that \eqref{r1} is true, which implies that $(\rho,\xi)$ is a critical point of the reduced map ${E}_{\lambda}$. 
	\end{proof}
	Next, we will give the asymptotic expansions for the reduced functional $E_{\lambda}$ and its derivatives.
	\begin{proposition}~\label{pro_expansion_Elamda}
		Assume \eqref{xi}-\eqref{lamda1}. The following expansions  hold for any $i=1,2,\cdots, m$ and $j=1,\cdots, \i(\xi_i)$:
		\[ 	\begin{array}{lcl}
			E_\lambda(\rho, \xi)
			&=& -\lambda_{k,m}-\lambda\log\left(\frac{\lambda_{k,m}}{8}\right)-\frac 1 2 \cF_{k,m}(\xi)+2(\lambda-\lambda_{k,m})\log \rho  -\mathcal{A}_1(\xi)\rho\\
			&&+\mathcal{A}_2(\xi)\rho^2\log\rho-\mathcal{B}(\xi)\rho^2+ O\left(\rho^2|\log\rho||\nabla\cF_{k,m}(\xi)|^2_g\right)+o(\rho^2), \\
			\partial_{(\xi_i)_j} E_{\lambda}(\rho,\xi)&=&-\frac 12 \partial_{(\xi_i)_j} \cF_{k,m}+\mathcal{O}(\rho),\\
			\partial_{\rho} E_{\lambda}(\rho,\xi)
			&=& \frac{2(\lambda-\lambda_{k,m})}{\rho}-\mathcal{A}_1(\xi)+\frac{\mathcal{A}^2_1(\xi)}{\lambda_{k,m}}\rho +\mathcal{A}_2(\xi)\rho+2\mathcal{A}_2(\xi)\rho\log\rho-2 \mathcal{B}(\xi)\rho\\
			&&+\mathcal{O}\left( \rho|\log\rho|^2|\nabla\cF_{k,m}(\xi)|_g^2\right)+o(\rho),\\
			\partial^2_{\rho} E_{\lambda}(\rho,\xi)
			&=&- \frac{2(\lambda-\lambda_{k,m})}{\rho^2}+\frac{\mathcal{A}_1^2(\xi)}{\lambda_{k,m}}+3\mathcal{A}_2(\xi)+2\mathcal{A}_2\log\rho-2\mathcal{B}(\xi)\\
			&&+\mathcal{O}\left( |\log\rho|^3|\nabla\cF_{k,m}(\xi)|_g^2\right)+o(1),\\
			\partial_{(\xi_i)_j}\partial_{\rho} E_{\lambda}(\rho,\xi)&=&  \mathcal{O}\left(\rho^{1-\kappa}|\log\rho|^4+ |\log\rho|^3|\nabla\cF_{k,m}(\xi)|_g^2\right),
		\end{array}\]
		as $\rho \rightarrow 0^{+}$. Moreover, $E_{\lambda}(\rho,\xi)$ is $C^1$-differentiable with respect to $\xi$ and $C^2$-differentiable with respect to $\rho$.
	\end{proposition}
	\begin{proof}
		Since $\int_{\Sigma} \phi d v_g=0$,
		$$DJ_{\lambda}(W)(\phi)=\int_{\Sigma}(-\Delta_gW+\beta W- \frac{\lambda V e^W}{\int_{\Sigma} V e^W dv_g }) \phi dv_g =\int_{\Sigma} -R \phi dv_g $$ and $$D^2 J_\lambda(W)(\phi,\phi)=\int_{\Sigma} L(\phi)\phi dv_g =\int_{\Sigma}(R+ N(\phi))\phi dv_g,$$ where $\phi$ is obtained by Proposition~\ref{propolin}. Then 
		$$
		\begin{array}{lll}
			J_\lambda(W+\phi)-J_\lambda(W)&=&D J_\lambda(W)(\phi)+\frac 1 2 {D^2 J_\lambda(W)(\phi,\phi)}\\
			&&+\int_0^1 \int_0^1\left(D^2 J_\lambda(W+t s \phi)-D^2 J_\lambda(W)\right)(\phi,\phi) t d s d t \\
			&=&-\frac{1}{2} \int_{\Sigma} R \phi d v_g+\frac{1}{2} \int_{\Sigma} N(\phi) \phi d v_g\\
			&&+\int_0^1 \int_0^1\left(D^2 J_\lambda(W+t s \phi)-D^2 J_\lambda(W)\right)(\phi,\phi) t d s d t
		\end{array}
		$$
		and 
		\[  D J_\lambda(W)(\phi)+D^2 J_\lambda(W)(\phi,\phi)=\int_{\Sigma} N(\phi) \phi d v_g\]
		hold. Since $\frac{1}{2} \int_{\Sigma} V e^W d v_g \leq \int_{\Sigma} V e^{W+t s \phi} d v_g \leq 2 \int_{\Sigma} V e^W d v_g$ and $\left|e^{W+t s \phi}-e^W\right| \leq C e^W\|\phi\|_{\infty}$, 
		$$
		\begin{array}{lcl}
			& &\left|D J_\lambda(W)(\phi)+D^2 J_\lambda(W)(\phi,\phi)\right|+\left|\int_0^1 t d t \int_0^1 d s\left[D^2 J_\lambda(W+t s \phi)-D^2 J_\lambda(W)\right](\phi,\phi)\right| \\
			& =&O\left(\|N(\phi)\|_*\|\phi\|_{\infty}+\|\phi\|_{\infty}^3\right)\stackrel{\eqref{est_N_s}}{=}O\left(\|\phi\|_{\infty}^3\right).
		\end{array}
		$$
		Lemma~\ref{est_R} and~\eqref{est_phi} yield 
		$$
		\left|J_\lambda(W+\phi)-J_\lambda(W)\right|=O\left(\|R\|_*\|\phi\|_{\infty}+\|\phi\|_{\infty}^3\right)\stackrel{}{=}O\left(\rho^2|\log\rho||\nabla\cF_{k,m}(\xi)|^2_g+ \rho^{3-\kappa}|\log\rho|^2\right). 
		$$
		By \eqref{expansion_JW}, 
		\[ 	\begin{aligned}
			E_\lambda(\rho, \xi)&=J_{\lambda}(W)+O\left(\rho^2|\log\rho||\nabla\cF_{k,m}(\xi)|^2_g+ \rho^{3-\kappa}|\log\rho|^2\right)\\
			&= -\lambda_{k,m}-\lambda\log\left(\frac{\lambda_{k,m}}{8}\right)-\frac 1 2 \cF_{k,m}(\xi)+2(\lambda-\lambda_{k,m})\log \rho  -\mathcal{A}_1(\xi)\rho+\mathcal{A}_2(\xi)\rho^2\log\rho\\
			&-\mathcal{B}(\xi)\rho^2+ O\left(\rho^2|\log\rho||\nabla\cF_{k,m}(\xi)|^2_g\right)+o(\rho^2). 
		\end{aligned}\]
		Differentiating with respect to $\iota=(\xi_i)_j\text{ or }\rho, i=1, \cdots, m, j=1,\cdots,\i(\xi_i)$, 
		$$
		\begin{aligned}
			\partial_\iota\left(J_\lambda(W+\phi)-J_\lambda(W)\right)= & -\frac{1}{2} \int_{\Sigma}\left[\partial_\iota R \phi+R \partial_\iota \phi\right] d v_g+\frac{1}{2} \int_{\Sigma}\left(\partial_\iota[N(\phi)] \phi+N(\phi) \partial_\iota \phi\right) d v_g \\
			& +\int_0^1 t d t \int_0^1 d s \partial_\iota\left\{\left[D^2 J_\lambda(W+t s \phi)-D^2 J_\lambda(W)\right](\phi,\phi)\right\}. 
		\end{aligned}
		$$
		By a direct calculation, we have 
		$$
		\begin{array}{lll}
			&& \left|\int_0^1 t d t \int_0^1 d s \partial_\iota\left\{\left[D^2 J_\lambda(W+t s \phi)-D^2 J_\lambda(W)\right](\phi,\phi)\right\}\right|\\&=&\mathcal{O}\left(\|\phi\|_{\infty}^2\left\|\partial_\iota \phi\right\|_{\infty}+\|\phi\|_{\infty}^3\left\|\partial_\iota W\right\|_{\infty}\right). 
		\end{array}
		$$ 
		Then  by Proposition~\ref{nonlinear},
		\[
		\begin{array}{lll}
			&& \partial_\iota\left(J_\lambda(W+\phi)-J_\lambda(W)\right) \\&=&\mathcal{O}\left(\|R\|_*|\partial_{\iota}\phi|_{\infty}+\|\partial_{\iota} R\|_*|\phi|_{\infty }+ \|N(\phi)\|_*|\partial_{\iota}\phi|_{\infty}+ \|\partial_{\iota} N(\phi)\|_*|\phi|_{\infty}\right)  \\
			&& +\mathcal{O}\left(\|\phi\|_{\infty}^2\left\|\partial_\iota \phi\right\|_{\infty}+\|\phi\|_{\infty}^3\left\|\partial_\iota W\right\|_{\infty}\right)\\
			&=&\mathcal{O}\left((\rho^2|\log\rho||\nabla\cF_{k,m}(\xi)|^2_g +\rho^{3-\kappa}|\log\rho|^2) \frac{|\log\rho|}{\rho}\right).
		\end{array}
		\]
		Thus applying~\eqref{expansion_JW_C_1} and \eqref{expansion_JW_rho1} we have for $\iota=(\xi_i)_j$
		\begin{eqnarray*}
			&&\partial_{\iota} E_{\lambda}(\rho,\xi)= \partial_{\iota}J_\lambda(W)+\mathcal{O}\left((\rho^2|\log\rho||\nabla\cF_{k,m}(\xi)|^2_g +\rho^{3-\kappa}|\log\rho|^2) \frac{|\log\rho|}{\rho}\right)\\
			&=&  -\frac 12 \partial_{\iota} \cF_{k,m}+\mathcal{O}(\rho),
		\end{eqnarray*}
		and for $\iota=\rho$
	\begin{eqnarray*}
			\partial_{\iota} E_{\lambda}(\rho,\xi)&=& \partial_{\iota}J_{\lambda}(W)+\mathcal{O}\left((\rho^2|\log\rho|^2|\nabla\cF_{k,m}(\xi)|^2_g +\rho^{3-\kappa}|\log\rho|^2) \frac{|\log\rho|}{\rho}\right)\\
		&=& \frac{2(\lambda-\lambda_{k,m})}{\rho}-\mathcal{A}_1(\xi)+\frac{\mathcal{A}^2_1(\xi)}{\lambda_{k,m}}\rho +\mathcal{A}_2(\xi)\rho+2\mathcal{A}_2(\xi)\rho\log\rho-2 \mathcal{B}(\xi)\rho\\
		&&+\mathcal{O}\left( \rho|\log\rho|^2|\nabla\cF_{k,m}(\xi)|_g^2\right)+o(\rho). 
	\end{eqnarray*}
		Similarly, by Proposition~\ref{nonlinear}, we can derive that 
		\[\partial^2_\rho\left(J_\lambda(W+\phi)-J_\lambda(W)\right) 
		=\mathcal{O}\left((\rho^2|\log\rho||\nabla\cF_{k,m}(\xi)|^2_g +\rho^{3-\kappa}|\log\rho|^2) \frac{|\log\rho|^2}{\rho^2}\right),
		\]
		and 
		\[\partial_{(\xi_i)_j}\partial_\rho\left(J_\lambda(W+\phi)-J_\lambda(W)\right) 
		=\mathcal{O}\left((\rho^2|\log\rho||\nabla\cF_{k,m}(\xi)|^2_g +\rho^{3-\kappa}|\log\rho|^2) \frac{|\log\rho|^2}{\rho^2}\right).
		\]
		Combining with	\eqref{expansion_JW_rho2} and~\eqref{expansion_JW_mix}, it follows that 
		\[ \begin{array}{lcl}
			\partial^2_{\rho} E_{\lambda}(\rho,\xi)&=& \partial^2_{\rho}J_\lambda(W)+\mathcal{O}\left((\rho^2|\log\rho||\nabla\cF_{k,m}(\xi)|^2_g +\rho^{3-\kappa}|\log\rho|^2) \frac{|\log\rho|^2}{\rho^2}\right)\\
			&=& -\frac{2(\lambda-\lambda_{k,m})}{\rho^2}+\frac{\mathcal{A}_1^2(\xi)}{\lambda_{k,m}}+3\mathcal{A}_2(\xi)+2\mathcal{A}_2\log\rho-2\mathcal{B}(\xi)\\
			&&+\mathcal{O}\left( |\log\rho|^3|\nabla\cF_{k,m}(\xi)|_g^2\right)+o(1),
		\end{array}
		\]
		and 
		\[ \begin{array}{lcl}
			\partial_{(\xi_i)_j}\partial_{\rho} E_{\lambda}(\rho,\xi)&=& \partial_{(\xi_i)_j}\partial_{\rho}J_\lambda(W)+\mathcal{O}\left((\rho^2|\log\rho||\nabla\cF_{k,m}(\xi)|^2_g +\rho^{3-\kappa}|\log\rho|^2) \frac{|\log\rho|^2}{\rho^2}\right)\\
			&=& \mathcal{O}\left(\rho^{1-\kappa}|\log\rho|^4+ |\log\rho|^3|\nabla\cF_{k,m}(\xi)|_g^2\right).
		\end{array}
		\]
	\end{proof}
	\section{The proof of the main theorem}\label{proof_main_thm}
	\begin{altproof}{Theorem~\ref{main_thm}}
		Here, we just construct the solution for $\lambda>\lambda_{k,m}$ and we can deduce it similarly for  $\lambda<\lambda_{k,m}$. 
		Suppose that $\xi^*$ is a $C^1$-stable critical point of $\cF_{k,m}$ and  $\delta>0$ sufficiently small  such that for any $\xi \in M_{\xi^*,\delta}(\subset U)$ there exists a refined isothermal coordinate $(y_{\xi}, U(\xi))$ in Claim~\ref{claim:0} such that $\xi$ satisfies  $\mathfrak{a}_1\text{ or } \mathfrak{a}_2$  in $M_{\xi^*,\delta}$, and  $U(\xi_i)\cap U(\xi_j)=\emptyset$ for any $i\neq j$ and $U(\xi_i)\cap \partial\Sigma=\emptyset$ for $i=1,\cdots, k.$

		Proposition~\ref{pro_expansion_Elamda} implies that 
		\[ \begin{array}{lcl}
			\left.	\frac{\rho\partial_{\rho} E_{\lambda}(\rho,\xi)}{\lambda-\lambda_{k,m}}\right|_{\rho=\mu\sqrt{\lambda-\lambda_{k,m}}}
			&=& 2-\frac{\mu\mathcal{A}_1(\xi)}{\sqrt{\lambda-\lambda_{k,m}}}+\frac{\mu^2\mathcal{A}^2_1(\xi)}{\lambda_{k,m}}+(\mu^2+2\mu^2\log\mu+\mu^2\log(\lambda-\lambda_{k,m}))\mathcal{A}_2(\xi)\\
			&&-2\mu^2 \mathcal{B}(\xi)+\mathcal{O}\left(\mu^2 |\log(\lambda-\lambda_{k,m})|^2|\nabla\cF_{k,m}(\xi)|_g^2\right)+o(\mu^2),\\
			\left. 	\frac{\rho^2\partial^2_{\rho} E_{\lambda}(\rho,\xi)}{\lambda-\lambda_{k,m}}\right|_{\rho=\mu\sqrt{\lambda-\lambda_{k,m}}}&=&-2 +\frac{\mu^2\mathcal{A}_1^2(\xi)}{\lambda_{k,m}}+(3\mu^2
			+2\mu^2\log\mu+\mu^2\log(\lambda-\lambda_{k,m}))\mathcal{A}_2(\xi)\\
			&&-2\mu^2\mathcal{B}(\xi)+\mathcal{O}\left(\mu^2 ( |\log(\lambda-\lambda_{k,m})|^3)|\nabla\cF_{k,m}(\xi)|_g^2\right)+o(\mu^2).
		\end{array}.\]
		
		We define that 
		\[ \mathcal{D}_{\lambda}:=\left\{\xi\in U: |\nabla\cF_{k,m}(\xi)|_g\leq \frac{\sqrt{2}}{|\log(\lambda-\lambda_{k,m})|^3} \right\}. \]
		 \\ 
		{\it Case 1. 
			$ \cA_1(\xi)\equiv 0$ and  $\cA_2(\xi)>0$ in $M_{\xi^*,\delta}$.}\\
		Suppose that $  a_1\leq |\cA_2(\xi)| \leq  a_0$ for any $\xi\in M_{\xi^*,\delta}$ for some constants $a_0, a_1>0.$
		Let  $m_0=\sqrt{ \frac 1 {2a_0|\log(\lambda-\lambda_{k,m})|}}>0$ and 
		$M_0= \sqrt{\frac 1 { a_1}} <\infty.$
		We define $I_{\lambda}=[m_0,M_0]$. 
		For any $\lambda$ sufficiently close to $\lambda_{k,m}$ and $\xi\in \mathcal{D}_{\lambda}$,  we have  as 
		\[ 	\begin{array}{ccl}
		\left.	\frac{(\rho\partial_{\rho} E_{\lambda})(\mu\sqrt{\lambda-\lambda_{k,m}},\xi)}{\lambda-\lambda_{k,m}}\right|_{\mu=m_0}
		&	\geq &2 
		-  m_0^2|\cA_2(\xi)|
		|\log(\lambda-\lambda_{k,m})|+
		o(1)\geq  1,\\
		\left.	\frac{(\rho\partial_{\rho} E_{\lambda})(\mu\sqrt{\lambda-\lambda_{k,m}},\xi)}{\lambda-\lambda_{k,m}}\right|_{\mu=M_0}
		&\leq& 2- M_0^2|\log(\lambda-\lambda_{k,m})|\mathcal{A}_2(\xi)(1+o(1))-2M_0^2\mathcal{B}(\xi)+o(1)\\
		&\leq& -1\\
		\frac{(\rho^2\partial^2_{\rho} E_{\lambda})(\mu\sqrt{\lambda-\lambda_{k,m}},\xi)}{\lambda-\lambda_{k,m}}
		&\leq &-2-\mu^2|\log(\lambda-\lambda_{k,m})|\mathcal{A}_2(\xi)(1+o(1))\\
		&&-2\mu^2\mathcal{B}(\xi) +o(1)\\
		&\leq & -1,
	\end{array}\]
		applying the restriction on $\mathcal{A}_1(\xi),\mathcal{A}_2(\xi)$ and $\mathcal{B}(\xi).$\\
		{\it Case 2. $\cA_1(\xi)=\cA_2(\xi)\equiv 0$ and $\cB(\xi)>0$ in $M_{\xi^*,\delta}$.}\\
		Suppose that $b_1\leq |\cB(\xi)|\leq b_0$ for any $\xi\in M_{\xi^*,\delta}$. 
		Let $ m_0=\sqrt{\frac 1 {4b_0|\log(\lambda-\lambda_{k,m})|}}$ and $M_0= \sqrt{\frac 1 {b_1}}<\infty.$
		It follows that 
			for any $\lambda$ sufficiently close to $\lambda_{k,m}$ and $\xi\in \mathcal{D}_{\lambda}$,  we have  as 
				\[ 	\begin{array}{ccl}
				\left.	\frac{(\rho\partial_{\rho} E_{\lambda})(\mu\sqrt{\lambda-\lambda_{k,m}},\xi)}{\lambda-\lambda_{k,m}}\right|_{\mu=m_0}
				&	\geq &2 
				- m_0^2|\cB(\xi)|+
				o(1)\geq  1,\\
				\left.	\frac{(\rho\partial_{\rho} E_{\lambda})(\mu\sqrt{\lambda-\lambda_{k,m}},\xi)}{\lambda-\lambda_{k,m}}\right|_{\mu=M_0}
				&\leq& 2- M_0^2|\cB(\xi)|+o(1)\\
				&\leq& -1\\
				\frac{(\rho^2\partial^2_{\rho} E_{\lambda})(\mu\sqrt{\lambda-\lambda_{k,m}},\xi)}{\lambda-\lambda_{k,m}}
				&\leq &-2-2\mu^2\mathcal{B}(\xi) +o(1)\\
				&\leq & -1.
			\end{array}\]
		In both cases, we find an interval $I_{\lambda}=[m_0, M_0]$ such that  
		\begin{eqnarray*}
				\left.	\frac{(\rho\partial_{\rho} E_{\lambda})(\mu\sqrt{\lambda-\lambda_{k,m}},\xi)}{\lambda-\lambda_{k,m}}\right|_{\mu=m_0}
			&	\geq&  1,\\
			\left.	\frac{(\rho\partial_{\rho} E_{\lambda})(\mu\sqrt{\lambda-\lambda_{k,m}},\xi)}{\lambda-\lambda_{k,m}}\right|_{\mu=M_0}
			&\leq& -1\\
			\frac{(\rho^2\partial^2_{\rho} E_{\lambda})(\mu\sqrt{\lambda-\lambda_{k,m}},\xi)}{\lambda-\lambda_{k,m}}
			&\leq & -1.
		\end{eqnarray*}
	Using the  intermediate value theorem and the monotonic property,  there exists $\varepsilon>0$ such that for any $\lambda\in(\lambda_{k,m},\lambda_{k,m}+\varepsilon)$ and $\xi\in \mathcal{D}_\lambda$ there is a unique solution $\mu:=\mu(\lambda,\xi)\in I_{\lambda}$ solving  $(\partial_{\rho}E_{\lambda})(\mu\sqrt{\lambda-\lambda_{k,m}},\xi)=0.$ The implicit function theorem implies that 
		$\xi\in \mathcal{D}_{\lambda}\mapsto \rho(\lambda,\xi):=\mu(\lambda,\xi)\sqrt{\lambda-\lambda_{k,m}}$ is $C^1$-differentiable with respect to $\xi$ and 
		\[ \partial_{\xi}\rho(\lambda,\xi)=-\frac{\partial_{\xi}\partial_{\delta}E_{\lambda}(\rho(\lambda,\xi),\xi)}{\partial_{\rho}^2 E_{\lambda}(\rho(\lambda,\xi),\xi)}. \]
		Proposition~\ref{pro_expansion_Elamda} implies that $$\partial_{\xi}\partial_{\rho}E_{\lambda}(\rho,\xi)=\mathcal{O}(|\log \rho|^3|\nabla\cF_{k,m}(\xi)|^2_g+\rho^{1-\kappa}|\log\rho|^4).$$ 
		Considering that  $|\partial^2_{\rho}E_{\lambda}(\rho(\lambda,\xi),\xi)|\geq \frac 1 {\mu^2}$ and $(\mu,\xi)\in I_{\lambda}\times\mathcal{D}_{\lambda}$, 
		\[ \partial_{\xi}\rho(\lambda,\xi)=\mathcal{O}\left(|\log(\lambda-\lambda_{k,m})|^{-3}\right). \]
		We will apply the cut-off function $\chi$ (see Section~\ref{sec_green}) to extend $\rho(\lambda,\xi)$ to be a $C^1$-function on $M_{\xi^*,\delta}$ for $\xi$. 
		Let 
		\begin{equation}\label{eq:de_rho}
			\begin{array}{lcl}
				\tilde{\rho}(\lambda,\xi)&:=& \chi\left(|\log(\lambda-\lambda_{k,m})|^6
				|\nabla\cF_{k,m}(\xi)|^2\right)\rho(\lambda,\xi)+(\lambda-\lambda_{k,m})^{\frac 1 2}\\
				&&\left(1- \chi\left(|\log(\lambda-\lambda_{k,m})|^6
				|\nabla\cF_{k,m}(\xi)|^2\right)\right).
			\end{array}
		\end{equation}
		Denote $\tilde{E}_{\lambda}(\xi)= E_{\lambda}(\tilde{\rho}(\lambda,\xi),\xi): I_{\lambda}\times M_{\xi^*,\delta}\rightarrow \RR.$
		Proposition~\ref{pro_expansion_Elamda} yields that 
		\[ \tilde{E}_{\lambda}(\xi):= -\lambda_{k,m}-\lambda\log\left(\frac{\lambda_{k,m}}{8}\right)-\frac 1 2 \cF_{k,m}(\xi)+\mathcal{O}\left(|\lambda-\lambda_{k,m}| |\log(\lambda-\lambda_{k,m})|\right)\]
		and 
		\begin{eqnarray*}
			\partial_{\xi}\tilde{E}_{\lambda}(\xi)&=&
			\partial_{\xi}E_{\lambda}(\tilde{\rho},\xi)+\partial_{\rho}E_\lambda(\rho,\xi)|_{\rho=\tilde{\rho}(\lambda,\xi)}\partial_{\xi}\tilde{\rho}(\lambda,\xi)\\
			&=&-\frac 12 \partial_{\xi}\cF_{k,m}(\xi)+\mathcal{O}\left(\sqrt{\lambda-\lambda_{k,m}}\cdot|\log(\lambda-\lambda_{k,m})|^2\right) ,
		\end{eqnarray*}
		which is uniformly convergent with respect to $\xi\in M_{\xi^*,\delta}$ as $\lambda\rightarrow\lambda_{k,m}^+.$
		By the assumption $\xi^*$ is a stable critical point of $\cF_{k,m}$, for any $\lambda$ in a small right neighborhood of $\lambda_{k,m}$, there exists $\xi_{\lambda}\in M_{\xi^*,\delta}$ to be a critical point of $\tilde{E}_{\lambda}(\xi)+\lambda_{k,m}+\lambda\log\left(\frac{\lambda_{k,m}}{8}\right)$. Consequently, $\xi_{\lambda}$ is also a critical point of $\tilde{E}_{\lambda}$.  
		Then, $$|\nabla \cF_{k,m}(\xi_{\lambda})|=\mathcal{O}\left(\sqrt{\lambda-\lambda_{k,m}}\cdot|\log(\lambda-\lambda_{k,m})|^2\right),$$  which implies that $\xi_{\lambda}\in \mathcal{D}_{\lambda}$. Hence, we have $\tilde{\rho}(\lambda,\xi_{\lambda})=\rho(\lambda,\xi_{\lambda})$, $\tilde{E}_{\lambda}(\xi_{\lambda}
		)=E_{\lambda}(\rho(\lambda,\xi_{\lambda}),\xi_{\lambda})$, $\partial_{\rho}E_{\lambda}(\rho(\lambda,\xi_{\lambda}),\xi_{\lambda}))=0 $ and $\partial_{\xi}E_{\lambda}(\rho(\lambda,\xi_{\lambda}),\xi_{\lambda}))=0.$ Proposition~\ref{pro_equi} yields that $u_{\lambda}:= W+\phi^{\rho(\lambda,\xi_{\lambda})}_{\xi_{\lambda}}$ is a solution of ~\eqref{eq:main_eq} with respect to $\lambda$.
		
		 For any $\delta>0$ sufficiently small and some $\lambda_{\delta
		}$ sufficiently close to $\lambda_{k,m}$, we can construct $\xi_{\lambda_{\delta}}\in M_{\xi^*,\delta}$, $\rho_{\delta}:=\rho(\lambda_{\delta},\xi_{\lambda_{\delta}})$ such that $\partial_{\xi}E_{\lambda_{\delta}}(\rho(\lambda_{\delta},\xi_{\lambda_{\delta}}),\xi_{\lambda_{\delta}})=0.$
		By Proposition~\ref{nonlinear} and Proposition~\ref{pro_equi},  there exists a solution $u_{\lambda}$ that solves the problem \eqref{eq:main_eq} for $\lambda$ in week sense. 
		$ \xi_{\lambda_{\epsilon}}=\left(\xi_{\lambda_{\epsilon},1}, \cdots,\xi_{\lambda_{\epsilon},m}\right)\rightarrow \xi^*=(\xi_1^*,\cdots,\xi^*_m)$
		as $\delta\rightarrow 0.$
		Then we  define 
		$u_{\delta}=\sum_{i=1}^m P U_{i,\delta}+\phi_{\xi_{\lambda_\delta}}^{\rho(\lambda_{\delta},\xi_{\lambda_\delta})}$, where $U_{i,\delta}=U_{\rho(\lambda_{\delta},\xi_{\lambda_\delta}) \sqrt{\tau_{i}(\xi_{\lambda_{\delta}})}, \xi_{\lambda_{\delta},i}}.$
		
		According to Proposition~\ref{pro_equi}, $u_{\rho_n}$ solves the problem~\eqref{eq:main_eq} in weak sense for some $\lambda_{\delta}$ satisfying that $|\lambda_{\delta}-\lambda_{k,m}|\leq C \rho^2(\lambda_\delta,\xi_{\lambda_\delta})|\log \rho(\lambda_\delta,\xi_{\lambda_\delta}) |$. 
		Claim that 
		$$
		\int_{\Sigma} \frac{ \lambda_\delta Ve^{u_\delta}}{\int_{\Sigma} Ve^{u_\delta} dv_g} \Psi dv_g  \rightarrow  \sum_{i=1}^m \varrho(\xi_i) \Psi\left(\xi_{i}\right)  \quad \text { as } \delta\rightarrow 0, \forall \Psi \in C(\Sigma) .
		$$
		In fact, by the inequality $\left|\mathrm{e}^{s}-1\right| \leqslant \mathrm{e}^{|s|}|s|$ for any $s \in \mathbb{R}$, the estimate \eqref{sumex1} yield that
		\begin{eqnarray}
			\int_{\Sigma} \frac{ \lambda_\delta Ve^{u_\delta}}{\int_{\Sigma} Ve^{u_\delta} dv_g} \Psi dv_g  
			& =&\sum_{i=1}^m \varrho(\xi_i) \Psi\left(\xi_{i}\right)+o(1), 
		\end{eqnarray} 
		as $\delta\rightarrow 0.$
		Therefore, $u_\epsilon$ is a sequence of blow-up solutions of \eqref{eq:main_eq} along $\lambda _\delta\rightarrow\lambda_{k,m}$. The proof is concluded. 
	\end{altproof}

	\newpage 
	\section{Appendix}\label{appendix}
	\subsection*{Appendix A.}
	\begin{lemma}~\label{lem:re_green}
		For any fixed $\xi\in \Sigma$ and $\alpha\in (0,1)$, 
		$H^g(x,\xi)$ is $C^{2,\alpha}$ in $\Sigma$ with respect to  $x$. Moreover, $ H^g(x,\xi)$ is uniformly bounded in $C^{2,\alpha}$ for any $\xi$ in any compact subset of $\intsigma$ or $\partial\Sigma$.\\
	\end{lemma} 
	\begin{proof}
		%Since $H^g(x,\xi)$ solves the following problem~\eqref{eq:eqR}.
		We apply the isothermal coordinate $(y_{\xi}, U(\xi))$ introduced in Section~\ref{prelim}.
		By the transformation law for $\Delta_g$ under a conformal transformation, $\Delta_{\tilde{g}}=e^{-\varphi}\Delta_g$ for any $\tilde{g}=e^{\varphi} g$. It follows that  
		$ \Delta_g\left( \log\frac 1{ |y_{\xi}(x)|}\right) =e^{-{\varphi}_{\xi}(y)} \left. \Delta \log\frac 1 {|y|}\right|_{y= y_{\xi}(x)}= -\frac{\varrho(\xi)} 4 \delta_{\xi},$
		where $\delta_{\xi}$ is the Dirac mass in $\xi\in \Sigma$.  
		For any $x\in U(\xi)\cap \partial \Sigma$, let $y=y_{\xi}(x)$
		\begin{equation*}
			\begin{aligned}
				\partial_{ \nu_g } \log|y_{\xi}(x)| &
				\stackrel{\eqref{eq:out_normal_derivatives}}{=} - e^{-\frac 1 2 {\varphi}_{\xi}(y)}\frac {\partial}{\partial y_2} \log |y|= - e^{-\frac 1 2 {\varphi}_{\xi}(y)}\frac{y_2}{|y|^2}\equiv 0. 
			\end{aligned}
		\end{equation*}
		Clearly, $\partial_{\nu_g} \chi(|y_{\xi}(x)|)=0$ for $x \in \partial\Sigma \cap  U_{r_\xi/4}(\xi)$.
		It follows that that $\partial_{ \nu_g } H^g(\cdot,\xi)$ is smooth on  $\partial \Sigma$. 
		And $\Delta_gH^g(\cdot,\xi)$ is bounded in $C^{\alpha}(\Sigma)$, for any $p\geq 1$. Then by the Schauder estimate in Lemma 5 of~\cite{yang2021125440}, 
		there is a unique solution  $H^g(x,\xi)\in C^{2,\alpha}(\Sigma)$ solves \eqref{eq:eqR}. Given that  $-\Delta_g H^g(\cdot,\xi)$ is uniformly bounded in $C^{\alpha}(\Sigma)$ for any $\xi$ in any compact subset of $\intsigma$ or $\partial\Sigma$, $ H^g(\cdot,\xi)$ is uniformly bounded in $C^{2,\alpha}$, via the Schauder estimates.
	\end{proof}

	\subsection*{Appendix B}
	For any $\xi$ in a compact subset of $\intsigma$ or $\partial \Sigma$ we take that $r_{\xi}= 4 r_0$, for some constant $r_0>0$. Next, we will obtain some important asymptotic behaviors of the approximation solutions.
	
	\begin{lemma}~\label{lem0}
		For any $r,\rho>0$ , we have 
		\begin{eqnarray}
			\int _{\B_r} \frac{\rho^2}{                                                                 ( \rho^2+ |y|^2)^2} dy = 2\int_{\B^+_r} \frac{\rho^2}{ ( \rho^2+ |y|^2)^2} dy= \pi- \frac{\pi\rho^2}{r^2}+ \frac{\pi\rho^4}{(r^2+\rho^2)r^2};
			\label{int01}\\
			\int_{\B_r}     \frac{2\rho^4}{(\rho^2+|y|^2 )^3 } dy =2	\int_{\B^+_r}     \frac{2\rho^4}{(\rho^2+|y|^2 )^3 } dy =
			\pi-\frac{ \pi \rho^4 }{(r^2+\rho^2)^2}; \label{int02}\\
			\int_{B} \frac{ \rho^3y_j}{ (\rho^2+|y|^2)^3} dy=\begin{cases}
				0 
				\text{ if } j=2 \text{ and }B= \B_r \text{ or } j=1\\
				\frac 12+\mathcal{O}(\rho^4) \text{ if } j=2, B=\B_r^+
			\end{cases};
			\label{int13}\\
			\quad	\int_{\B_{r}} \frac{ \rho^2\log( \frac{\rho^2+ |y|^2}{\rho^2} )}{ (\rho^2+|y|^2)^2} dy=	2\int_{\B^+_{r}} \frac{\rho^2\log( \frac{\rho^2+ |y|^2}{\rho^2} )}{ (\rho^2+|y|^2)^2} dy= \pi + \frac{\pi\rho^2\log(\rho^2)}{r^2}+ \mathcal{O}(\rho^2), 	\label{int14}
		\end{eqnarray}
		as $\rho\rightarrow 0,$ where $\mathcal{O}(\cdot)$ relies on $r.$
	\end{lemma}
	\begin{proof}
		The proof can be derived from straightforward computation and applying symmetry properties.
	\end{proof}
	\begin{cor}~\label{cor0}
		Let $f \in C^{2, \alpha}(\Sigma)$ (possibly depending on $\xi$), $0<\alpha<1$,  $P_2(f)$ be the second-order Taylor expansion of $f(x)$ at $\xi$, i.e.
		$$
		P_2 f(x):=f(\xi)+\left\langle\nabla\left(f \circ y_{\xi}^{-1}\right)(0), y_{\xi}(x)\right\rangle+\frac{1}{2}\left\langle D^2\left(f \circ y_{\xi}^{-1}\right)(0) y_{\xi}(x), y_{\xi}(x)\right\rangle. 
		$$ 
		Let $\mathfrak{f}_2(\xi)=0$ if $\xi\in \intsigma$ and $ \mathfrak{f}_2(\xi)=\frac{\partial}{\partial y_2} f\circ y_{\xi}^{-1}(0)$ if $\xi\in \partial\Sigma$. Then 
		the following expansions hold as $\rho \rightarrow 0$ :
		\begin{equation}~\label{eqspan}
			\begin{array}{lll}
				&&	\int_{\Sigma} \chi_{\xi} e^{-\varphi_{\xi}\circ y_{\xi}^{-1}} f(x) e^{U_{\rho,\xi}} dv_g\\
				&=& \varrho(\xi)f(\xi)+\varrho(\xi)\mathfrak{f}_2(\xi)\rho+ \left( 4 f(\xi)\int_{\Sigma}\frac 1 {r_0} \chi^{\prime}(\frac {|y_{\xi}|}{r_0})e^{-\varphi_{\xi}\circ y_{\xi}^{-1}}\frac {1} {|y_{\xi}|^3}dv_g\right.\\
				&&\left. +  4\mathfrak{f}_2(\xi) \int_{\Sigma} \frac 1 {r_0} \chi^{\prime}(\frac {|y_{\xi}|}{r_0})e^{-\varphi_{\xi}\circ y_{\xi}^{-1}}\frac {(y_{\xi})_2 } {|y_{\xi}|^3}dv_g -\frac {\varrho(\xi)} 4\Delta_g f(\xi)(2\log \rho+1)\right. \\
				&& \left. -2\Delta_g f(\xi) \int_{\Sigma} \frac 1 {r_0} \chi^{\prime}(\frac {|y_{\xi}|}{r_0})e^{-\varphi_{\xi}\circ y_{\xi}^{-1}}\frac {\log|y_{\xi}| } {|y_{\xi}|}dv_g  + 8 \int_{\Sigma} \chi_{\xi} e^{-\varphi_{\xi}\circ y_{\xi}^{-1}} \frac{f-P_2(f)}{|y_{\xi}|^4} dv_g \right)\rho^2
				+o(\rho^2);
			\end{array}
		\end{equation}
		\begin{equation}\label{eqspan_1}
			\begin{array}{lll}
				\int_{\Sigma} \chi_{\xi} e^{-\varphi_{\xi}\circ y_{\xi}^{-1}} f(x) e^{U_{\rho, \xi}} \frac{d v_g}{\rho^2+\left|y_{\xi}(x)\right|^2}&=&\frac{\varrho(\xi) f(\xi)}{2\rho^2}+\frac{\varrho(\xi)}{4\rho}\mathfrak{f}_2(\xi)     \\
				& &+\frac{\varrho(\xi)}{8}\Delta_g  f(\xi)+ \mathcal{O}(\rho^{\alpha});
			\end{array}
		\end{equation}
		for any $a\in \RR$, 
		\begin{equation}\label{eqspan_2}
			\begin{array}{lll}
				\int_{\Sigma} \chi_{\xi} e^{-\varphi_{\xi}\circ y_{\xi}^{-1}} f(x) e^{U_{\rho, \xi}} \frac{( a \rho^2-|y_{\xi}|^2)d v_g}{(\rho^2+\left|y_{\xi}(x)\right|^2)^2}&=& \frac{\varrho(\xi)(2a-1)}{6\rho^2}f(\xi) +\frac{(a-1)\varrho(\xi)}{8\rho} \mathfrak{f}_2(\xi)\\
				&&+\frac{\varrho(\xi)(a-2)}{24}\Delta_g f(\xi)+ \mathcal{O}(\rho^{\alpha}). 
			\end{array}
		\end{equation}
	\end{cor}
	\begin{proof}
		Applying the Taylor expansion, we have 
		\begin{equation*}
			\begin{array}{lll}
				&&\int_{\Sigma} \chi_{\xi} e^{-\varphi_{\xi}\circ y_{\xi}^{-1}} f e^{U_{\rho,\xi}} dv_g=
				\int_{U_{2r_0}(\xi)} \chi_{\xi} e^{-\varphi_{\xi}\circ y_{\xi}^{-1}} f e^{U_{\rho,\xi}} dv_g\\
				&=&\int_{\Sigma} \chi_{\xi} e^{-\varphi_{\xi}\circ y_{\xi}^{-1}} P_2(f)(x) e^{U_{\rho,\xi}} dv_g + \int_{\Sigma} \chi_{\xi} e^{-\varphi_{\xi}\circ y_{\xi}^{-1}} (f-P_2(f)) e^{U_{\rho,\xi}} dv_g.  
			\end{array}
		\end{equation*}
		Since $|f(x)-P_2f(x)|\leq C |y_{\xi}(x)|^{2+\alpha}$ for same $\alpha\in (0,1)$, Lebesgue's dominated convergence theorem yields that 
		\begin{equation}\label{eq:est_D_re}
			\begin{array}{lll}
				\int_{\Sigma} \chi_{\xi} e^{-\varphi_{\xi}\circ y_{\xi}^{-1}} (f-P_2(f)) e^{U_{\rho,\xi}} dv_g
				&=& 8\rho^2 \int_{\Sigma}\chi_{\xi} e^{-\varphi_{\xi}\circ y_{\xi}^{-1}} \frac{f-P_2(f)}{|y_{\xi}|^4} dv_g+o(\rho^2). 
			\end{array}
		\end{equation}
		Applying~\eqref{int01} and integrating by part, we have 
		\begin{equation}\label{eq:est_D_0}
			\begin{array}{lll}
				&&	f(\xi)\int_{\Sigma} \chi_{\xi} e^{-\varphi_{\xi}\circ y_{\xi}^{-1}} e^{U_{\rho,\xi}} dv_g \\&=& f(\xi)\int_{ B_{2r_0}^{\xi}} \frac {8\rho^2}{(\rho^2+|y|^2)^2} dy + f(\xi)\int_{ B_{2r_0}^{\xi}} (\chi(|y|/r_0)-1)\frac {8\rho^2}{(\rho^2+|y|^2)^2} dy \\
				&=&  \left(\varrho(\xi)f(\xi)- \frac{\varrho(\xi)f(\xi)\rho^2}{4r_0^2}\right)+\left( \frac{\varrho(\xi)f(\xi)\rho^2}{4r_0^2}+\frac{\varrho(\xi)f(\xi)\rho^2}{r_0} \int_0^{\infty} \chi^{\prime}(\frac s {r_0})s^{-2}ds\right)+
				\mathcal{O}(\rho^4)\\
				&=& \varrho(\xi)f(\xi)+\frac{\varrho(\xi)\rho^2}{r_0} f(\xi)\int_0^{\infty} \chi^{\prime}(\frac s {r_0})s^{-2}ds+
				\mathcal{O}(\rho^4).
			\end{array}
		\end{equation}
		By a straightforward calculation, for $\xi\in \partial\Sigma$, i.e. $  B_{2r_0}^{\xi}=\B_{2r_0}^+$
		\begin{equation}
			\label{eq:inte_34}
			\begin{array}{lll}
				\int_{ B_{2r_0}^{\xi}} \frac{8\rho^2 y_2}{(\rho^2+|y|^2)^2}&=& 8\rho \arctan(2r_0/\rho)- \frac{4\rho^2}{r_0}+\mathcal{O}(\rho^4)\\
				&=& 4\pi \rho -8\rho\int_{2r_0/\rho}^{\infty}\frac 1 {(s+1)^2} s + 8\rho\int_{2r_0/\rho}^{\infty}( \frac 1 {(s+1)^2}- \frac 1 {s^2+1}) ds - \frac{4\rho^2}{r_0}+\mathcal{O}(\rho^4)\\
				&=& 4\pi \rho-\frac{8\rho^2}{r_0}+\mathcal{O}(\rho^3). 
			\end{array}
		\end{equation}
		Similarly,  for $\xi\in \partial\Sigma$, we use~\eqref{eq:inte_34} and integrating by part to derive the following estimate: 
		\begin{equation}\label{eq:est_D_1}
			\begin{array}{lcl}
				&&\int_{\Sigma} \chi_{\xi} e^{-\varphi_{\xi}\circ y_{\xi}^{-1}} \la \nabla f\circ y_{\xi}^{-1}(0), y_{\xi}(x)\ra e^{U_{\rho,\xi}} dv_g \\&=& \frac{\partial}{\partial y_2}f\circ y_{\xi}^{-1}(0)\int_{ B_{2r_0}^{\xi}} \frac {8\rho^2y_2}{(\rho^2+|y|^2)^2} dy + \frac{\partial}{\partial y_2}f\circ y_{\xi}^{-1}(0)\int_{ B_{2r_0}^{\xi}} (\chi(|y|/r_0)-1)\frac {8\rho^2y_2}{(\rho^2+|y|^2)^2} dy \\
				&=& \varrho(\xi)\frac{\partial}{\partial y_2}f\circ y_{\xi}^{-1}(0)\rho-\frac{ \varrho(\xi)}{\pi r_0}\frac{\partial}{\partial y_2}f\circ y_{\xi}^{-1}(0)\rho^2\\
				&&+ \frac{\partial}{\partial y_2}f\circ y_{\xi}^{-1}(0)\int_{ B_{2r_0}^{\xi}} (\chi(|y|/r_0)-1)\frac {8\rho^2y_2}{(\rho^2+|y|^2)^2} dy+
				\mathcal{O}(\rho^3)\\
				&=& \varrho(\xi)\frac{\partial}{\partial y_2}f\circ y_{\xi}^{-1}(0)\rho-\frac{ \varrho(\xi)}{\pi r_0}\frac{\partial}{\partial y_2}f\circ y_{\xi}^{-1}(0)\rho^2\\
				&&+\frac{\partial}{\partial y_2}f\circ y_{\xi}^{-1}(0)\left(\frac{\varrho(\xi)\rho^2}{\pi r_0} +\frac{\varrho(\xi)\rho^2}{\pi r_0}  \int_0^{\infty} \chi'(\frac s {r_0})s^{-1} ds\right)+ \mathcal{O}(\rho^3)\\
				&=& \varrho(\xi)\frac{\partial}{\partial y_2}f\circ y_{\xi}^{-1}(0)\rho+ \frac{\varrho(\xi)\rho^2}{\pi r_0} \frac{\partial}{\partial y_2}f\circ y_{\xi}^{-1}(0) \int_0^{\infty} \chi'(\frac s {r_0})s^{-1} ds+\mathcal{O}(\rho^3).
			\end{array}
		\end{equation}
		For $\xi\in \intsigma$, we have $\int_{\Sigma} \chi_{\xi} e^{-\varphi_{\xi}\circ y_{\xi}^{-1}} \la \nabla f\circ y_{\xi}^{-1}(0), y_{\xi}(x)\ra e^{U_{\rho,\xi}} dv_g =0$ by the symmetric property. 
		In view of that $\int_{B^{\xi}_{2r_0}} \frac{y_1y_2}{(\rho^2+|y|^2)^2} dy=0$ and $\frac 1 2 \int_{B^{\xi}_{2r_0}} \frac{y_1^2-y_2^2}{(\rho^2+|y|^2)^2} dy=0,$ 
		\[\frac 1 2 \int_{\Sigma} \chi_{\xi} e^{-\varphi_{\xi}\circ y_{\xi}^{-1}} \la D^2 f\circ y_{\xi}^{-1}(0) y_{\xi}(x), y_{\xi}(x)\ra e^{U_{\rho,\xi}} dv_g
		=\frac 1 4 \Delta f\circ y_{\xi}^{-1}(0) \int_{B^{\xi}_{2r_0}}\chi(|y|/r_0)  \frac{ 8\rho^2 |y|^2}{(\rho^2+|y|^2)} dy. \]
		By integrating by part again, we have 
		\begin{equation}\label{eq:est_D_2}
			\begin{array}{lll}
				&&\frac 1 4 \Delta f\circ y_{\xi}^{-1}(0) \int_{B^{\xi}_{2r_0}}\chi(|y|/r_0)  \frac{ 8\rho^2 |y|^2}{(\rho^2+|y|^2)} dy\\&=& -\frac {\varrho(\xi)} 4\Delta f\circ y_{\xi}^{-1}(0)\left(2\rho^2\log \rho+\rho^2 + \frac{2\rho^2}{r_0}\int_0^{\infty} \chi'(\frac s {r_0}) \log s ds\right)\\
				&&+\mathcal{O}(\rho^4).
			\end{array}
		\end{equation}
		Combining the estimates~\eqref{eq:est_D_re},~\eqref{eq:est_D_0},~\eqref{eq:est_D_1} and \eqref{eq:est_D_2}, 
		\begin{equation*}
			\begin{array}{lll}
				&&	\int_{\Sigma} \chi_{\xi} e^{-\varphi_{\xi}\circ y_{\xi}^{-1}} f(x) e^{U_{\rho,\xi}} dv_g\\
				&=& \varrho(\xi)f(\xi)+\varrho(\xi)\mathfrak{f}_2(\xi)\rho+ \left( 4 f(\xi)\int_{\Sigma}\frac 1 {r_0} \chi^{\prime}(\frac {|y_{\xi}|}{r_0})e^{-\varphi_{\xi}\circ y_{\xi}^{-1}}\frac {1} {|y_{\xi}|^3}dv_g\right.\\
				&&\left. +  4\frac{\partial}{\partial y_2}f\circ y_{\xi}^{-1}(0) \int_{\Sigma} \frac 1 {r_0} \chi^{\prime}(\frac {|y_{\xi}|}{r_0})e^{-\varphi_{\xi}\circ y_{\xi}^{-1}}\frac {(y_{\xi})_2 } {|y_{\xi}|^3}dv_g\right. \\
				&& \left. -\frac {\varrho(\xi)} 4\Delta f\circ y_{\xi}^{-1}(0)(2\log \rho+1)-2\Delta f\circ y_{\xi}^{-1}(0) \int_{\Sigma} \frac 1 {r_0} \chi^{\prime}(\frac {|y_{\xi}|}{r_0})e^{-\varphi_{\xi}\circ y_{\xi}^{-1}}\frac {\log|y_{\xi}| } {|y_{\xi}|}dv_g \right.\\
				&&\left. + 8 \int_{\Sigma} \chi_{\xi} e^{-\varphi_{\xi}\circ y_{\xi}^{-1}} \frac{f-P_2(f)}{|y_{\xi}|^4} dv_g \right)\rho^2
				+o(\rho^2),
			\end{array}
		\end{equation*}
		where $\mathfrak{f}_2(\xi)=0$ if $\xi\in \intsigma$ and $ \mathfrak{f}_2(\xi)=\frac{\partial}{\partial y_2} f\circ y_{\xi}^{-1}(0)$ if $\xi\in \partial\Sigma$. Since $ \Delta_g f(\xi)= \Delta f\circ y_{\xi}^{-1}(0)$,~\eqref{eqspan} is concluded. 
		Lemma~\ref{lem0} yields that 
		\begin{equation*}
			\begin{array}{lll}
				&&\int_{\Sigma} \chi_{\xi} e^{-\varphi_{\xi}\circ y_{\xi}^{-1}} f(x) e^{U_{\rho, \xi}} \frac{d v_g}{\rho^2+|y_{\xi}|^2}= 	\int_{U_{2r_0}(\xi)} \chi_{\xi} e^{-\varphi_{\xi}\circ y_{\xi}^{-1}} f(x) e^{U_{\rho, \xi}} \frac{d v_g}{\rho^2+|y_{\xi}|^2}\\
				&=& \int_{U_{2r_0}(\xi)} \chi_{\xi} e^{-\varphi_{\xi}\circ y_{\xi}^{-1}} P_2f(x) e^{U_{\rho, \xi}} \frac{d v_g}{\rho^2+\left|y_{\xi}(x)\right|^2}+ \int_{U_{2r_0}(\xi)} \chi_{\xi} e^{-\varphi_{\xi}\circ y_{\xi}^{-1}} (f-P_2f) e^{U_{\rho, \xi}} \frac{d v_g}{\rho^2+|y_{\xi}|^2} \\
				&=&\frac 8 {\rho^2} f(\xi) \int_{\frac 1 {\rho} B_{2r_0}^{\xi}}\chi(|y|/r_0) \frac{1}{(1+|y|^2)^3} dy+ \frac{8}{\rho
				}\frac{\partial}{\partial y_2} f\circ y_{\xi}^{-1}(0) \int_{\frac 1 {\rho} B_{2r_0}^{\xi}}\chi(|y|/r_0) \frac{y_2}{(1+|y|^2)^3} dy\\
				&& + \frac 1 4 \Delta  f\circ y_{\xi}^{-1}(0) \int_{\frac 1 {\rho} B_{2r_0}^{\xi}}\chi(|y|/r_0) \frac{8 |y|^2}{(1+|y|^2)^3} dy + \mathcal{O}(\rho^{\alpha}) \\
				&=& \frac{\varrho(\xi) f(\xi)}{2\rho^2}+\frac{\varrho(\xi)\mathfrak{f}_2(\xi)}{4\rho}  \frac{\partial}{\partial y_2} f\circ y_{\xi}^{-1}(0) +\frac{\varrho(\xi)}{8}\Delta  f\circ y_{\xi}^{-1}(0) + \mathcal{O}(\rho^{\alpha}), 
			\end{array}
		\end{equation*}
		for any $\alpha\in (0,1).$
		We observe that 
		\begin{equation*}
			\begin{array}{lll}
				\int_{ \frac 1{\rho
					} B^{\xi}_{2r_0}} \frac 8 {(1+|y|^2)^4} dy&=& \frac {\varrho(\xi)} 3\left(1- \frac{\rho^6}{(4r_0^2+\rho^2)^3}\right),\\
				\int_{ \frac 1{\rho
					} B^{\xi}_{2r_0}} \frac {8y_2} {(1+|y|^2)^4} dy&=&\begin{cases}
					0& \xi\in\intsigma\\
					\frac \pi 2 +\mathcal{O}(\rho^2)& \xi\in \partial\Sigma
				\end{cases}. 
			\end{array}
		\end{equation*}
		It follows that  
		\begin{equation*}
			\begin{array}{lll}
				&&	\int_{\Sigma} \chi_{\xi} e^{-\varphi_{\xi}\circ y_{\xi}^{-1}} f e^{U_{\rho, \xi}} \frac{( a \rho^2-|y_{\xi}|^2)d v_g}{(\rho^2+|y_{\xi}|^2)^2}=	\int_{U_{2r_0}(\xi)} \chi_{\xi} e^{-\varphi_{\xi}\circ y_{\xi}^{-1}} f e^{U_{\rho, \xi}} \frac{( a \rho^2-|y_{\xi}|^2)d v_g}{(\rho^2+|y_{\xi}|^2)^2}\\
				&=& \int_{U_{2r_0}(\xi)} \chi_{\xi} e^{-\varphi_{\xi}\circ y_{\xi}^{-1}} P_2f e^{U_{\rho, \xi}} \frac{( a \rho^2-|y_{\xi}|^2)d v_g}{(\rho^2+|y_{\xi}|^2)^2}\\
				&&+ \int_{U_{2r_0}(\xi)} \chi_{\xi} e^{-\varphi_{\xi}\circ y_{\xi}^{-1}} (f-P_2f) e^{U_{\rho, \xi}} \frac{( a \rho^2-|y_{\xi}|^2)d v_g}{(\rho^2+|y_{\xi}|^2)^2}\\
				&=&\frac 8 {\rho^2} f(\xi) \int_{\frac 1 {\rho} B_{2r_0}^{\xi}}\chi(|y|/r_0) \frac{a-|y|^2}{(1+|y|^2)^4} dy+ \frac{8}{\rho
				}\frac{\partial}{\partial y_2} f\circ y_{\xi}^{-1}(0) \int_{\frac 1 {\rho} B_{2r_0}^{\xi}}\chi(|y|/r_0) \frac{(a-|y|^2)y_2}{(1+|y|^2)^4} dy\\
				&& + \frac 1 4 \Delta  f\circ y_{\xi}^{-1}(0) \int_{\frac 1 {\rho} B_{2r_0}^{\xi}}\chi(|y|/r_0) \frac{8(a-|y|^2)|y|^2}{(1+|y|^2)^4} dy + \mathcal{O}(\rho^{\alpha})\\
				&=& \frac{\varrho(\xi)(2a-1)}{6\rho^2}f(\xi) +\frac{(a-1)\varrho(\xi)}{8\rho} \mathfrak{f}_2(\xi)+\frac{\varrho(\xi)(a-2)}{24}\Delta f\circ y_{\xi}^{-1}(0)+ \mathcal{O}(\rho^{\alpha}),
			\end{array}
		\end{equation*}	for any $\alpha\in (0,1).$
	\end{proof}
	
	The following lemma shows the asymptotic expansion of $PU_{\rho,\xi}$  as $\rho\rightarrow 0$.
	Let $$f_{\xi}=\frac{1}{|y_{\xi}|^2}\Delta_g\chi_{\xi}+2 \left\la\nabla \chi_{\xi},\nabla \frac 1{|y_{\xi}|^2}\right\ra_g+\frac 2 {|\Sigma|_g} \int_{\Sigma}\frac 1 {r_0} \chi^{\prime}\left(\frac{|y_{\xi}|}{r_0}\right)\frac{e^{-\varphi_{\xi}\circ y_{\xi}^{-1}} }{ |y_{\xi}|^3} dv_g$$ and  $F_{\xi}$  be the unique solution to the following equations:
	\begin{equation}\label{eq:F_xi}
		\begin{cases}
			-\Delta_g F_{\xi}= f_{\xi}& \text{ in }\Sigma\\
			\partial_{\nu_g}F_{\xi}=0 & \text{ on }\partial\Sigma\\
			\int_{\Sigma} F_{\xi} dv_g=0
		\end{cases}. 
	\end{equation}
	$F_{\xi}$ is well-defined. 
	By integrating \eqref{eq:delta_f_xi} over $\Sigma$: 
	\[ \int_{\Sigma} f_{\xi}=- \frac 1 {2\rho^2}\int_{\Sigma} \Delta_g{\eta}_{\rho,\xi} + \mathcal{O}(\rho^2).\]
	Applying the divergence theorem and leveraging~\eqref{eq:f_i_b}, we deduce that
	\[ \int_{\Sigma} \Delta_g {\eta}_{\rho,\xi}= \int_{\Sigma} \partial_{\nu_g}{\eta}_{\rho,\xi} =\mathcal{O}(\rho^4).\]
	Consequently, we derive $\int_{\Sigma} f_{\xi}=0.$
	\begin{lemma}~\label{lemb1}
		The function $P U_{\delta, \xi}$ has the following expansion, 
		$$
		P U_{\rho, \xi}=\chi_{\xi}\left(U_{\rho, \xi}(x)-\log \left(8 \rho^2\right)\right)+\varrho(\xi) H^g(x, \xi)+h_{\rho,\xi}-2\rho^2 F_{\xi} +\mathcal{O}(\rho^3|\log \rho|)\quad(\rho\rightarrow 0),
		$$
		in $C(\Sigma)$, where $F_{\xi}$ is defined by~\eqref{eq:F_xi} and $h_{\rho,\xi}$  is a constant defined by 
		\[ h_{\rho,\xi}=	-\frac{\varrho(\xi)\rho^2\log \rho}{2|\Sigma|_g}+ \frac{2\rho^2}{|\Sigma|_g} \left(\frac{\varrho(\xi)}{8}+\int_{\Sigma} \chi_{\xi} \frac{1- e^{-\varphi_{\xi}\circ y_{\xi}^{-1}}}{|y_{\xi}|^2} dv_g -\int_{\Sigma} \frac 1 {r_0} \chi^{\prime}\left(\frac{|y_{\xi}|} {r_0}\right) \frac{ \log|y_{\xi}|}{|y_{\xi}|}dv_g \right). \] 
		Moreover, the convergent is locally uniform for $\xi$ in  $\intsigma$ and $\partial \Sigma$. 
		In particular, 
		$$
		P U_{\rho, \xi}=\varrho(\xi)G^g(x, \xi)-2 \rho^2\frac {\chi_{\xi}} {|y_{\xi}|^2}+h_{\rho,\xi}-2\rho^2F_{\xi}+\mathcal{O}(\rho^3|\log \rho|)\quad(\rho\rightarrow 0), 
		$$
		in $C_{loc}(\Sigma\setminus\{\xi\})$.
	\end{lemma}
	\begin{proof}
		Let $\eta_{\rho,\xi}(x)= PU_{\rho,\xi}-\chi_{\xi}(U_{\rho,\xi}-\log 8\rho^2)  -\varrho(\xi) H^g(x,\xi)$.
		\begin{equation*}
			\int_{\Sigma}\eta_{\rho,\xi} dv_g = 2\int_{\Sigma}\chi_{\xi}\log\left(1+ \frac{\rho^2}{|y_{\xi}|^2}\right) dv_g. 
		\end{equation*}
		We divide the integral into two parts: 
		\[ 	2\int_{\Sigma}  \chi_{\xi}\log\left(1+ \frac{\rho^2}{|y_{\xi}|^2}\right) dv_g=2\left(\int_{B^{\xi}_{r_0}} +\int_{B^{\xi}_{2r_0}\setminus B^{\xi}_{r_0}}\right) \chi(|y|/r_0)e^{\varphi_{\xi}(y)} \log \left(1+ \frac{\rho^2 }{|y|^2}\right)dy. \]
		Since $\log \left(1+ \frac{\rho^2 }{|y|^2}\right)= \frac{\rho^2}{|y|^2}+\mathcal{O}(\rho^4)$  for any $r_0\leq |y|\leq 2r_0$, 
		\begin{equation}\label{eq:int_res}
			\begin{array}{lll}
				&&2\int_{B^{\xi}_{2r_0}\setminus B^{\xi}_{r_0}} \chi(|y|/r_0) e^{\varphi_{\xi}(y)} \log \left(1+ \frac{\rho^2 }{|y|^2}\right)dy
				=2\rho^2 \int_{B^{\xi}_{2r_0}\setminus B^{\xi}_{r_0}} \chi(|y|/r_0) e^{\varphi_{\xi}(y)} |y|^{-2}dy +\mathcal{O}(\rho^4)\\
				&=& 2\rho^2 \int_{B^{\xi}_{2r_0}\setminus B^{\xi}_{r_0}} \chi(|y|/r_0) (e^{\varphi_{\xi}(y)} -1)\frac 1 {|y|^2}dy-\frac{\varrho(\xi)\rho^2}{2}\log r_0 -\frac{\varrho(\xi)\rho^2}{2} \int_0^{\infty}\frac 1 {r_0} \chi^{\prime}(\frac s {r_0}) \log s ds+\mathcal{O}(\rho^4) 
			\end{array}
		\end{equation}
		We observe that 
		\begin{equation}\label{eq:int_0_main}
			\begin{array}{lll}
				\int_{B^{\xi}_{r_0}} \log \left(1+ \frac{\rho^2 }{|y|^2}\right)dy&=& \frac{\varrho(\xi)\rho^2}{4} \int_0^{\infty} \left(\log(1+s^{-2})- {(1+s^2)^{-1}}\right)s ds\\
				&& + \frac{\varrho(\xi)\rho^2}{8}\log\left(1+ \frac{r_0^2}{\rho^2}\right)+ \mathcal{O}(\rho^4)
				\\
				&=& \frac{\varrho(\xi)\rho^2}{8}  + \frac{\varrho(\xi)\rho^2}{8}\log\left(1+ \frac{r_0^2}{\rho^2}\right)+ \mathcal{O}(\rho^4). 
			\end{array}
		\end{equation}
		Here, the last inequality applied the integral $\int_0^{\infty} \left(\log(1+s^{-2})- {(1+s^2)^{-1}}\right)s ds=\frac 1 2. $
		By $e^{\varphi_{\xi}(y)}=\begin{cases}
			1+ \mathcal{O}(|y|^2) & \text{ for }\xi\in \intsigma \\
			1-2k_g(\xi)y_2+ \mathcal{O}(|y|^2)  & \text{ for }\xi\in \partial\Sigma
		\end{cases}$, we can deduce that 
		\begin{equation}\label{eq:int_main_r1}
			\begin{array}{lll}
				&&	\int_{B^{\xi}_{r_0}} \left( \log \left(1+ \frac{\rho^2 }{|y|^2}\right)-\frac {\rho^2}{|y|^2}\right) (e^{\varphi_{\xi}(y)}-1) dy \\
				&=& \rho^3\mathcal{O}\left( \int_{\B_{r_0/\rho}} \left| \log \left(1+ \frac{1 }{|y|^2}\right)-\frac {1}{|y|^2}\right| |y|dy \right)\\
				&=&\rho^3 \mathcal{O}\left(\left( \int_{\B_{r_0/\rho}\setminus\B_1}+ \int_{\B_1} \right)\left| \log \left(1+ \frac{1 }{|y|^2}\right)-\frac {1}{|y|^2}\right| |y|dy \right)\\
				&=&\mathcal{O}(\rho^3|\log \rho|)
			\end{array}
		\end{equation}
		Combining estimates~\eqref{eq:int_res},~\eqref{eq:int_0_main} and~\eqref{eq:int_main_r1}, it follows that 
		\begin{equation}
			\begin{array}{lll}
				&&\int_{\Sigma}  \eta_{\rho,\xi} dv_g=  \frac{\varrho(\xi)\rho^2}{4} + \frac{\varrho(\xi)\rho^2}{4}\log\left(1+ \frac{r_0^2}{\rho^2}\right)+2\rho^2 \int_{B^{\xi}_{2r_0}\setminus B^{\xi}_{r_0}} \chi(|y|/r_0) (e^{\varphi_{\xi}(y)} -1)\frac 1 {|y|^2}dy\\
				&&-\frac{\varrho(\xi)\rho^2}{2}\log r_0 -\frac{\varrho(\xi)\rho^2}{2} \int_0^{\infty}\frac 1 {r_0} \chi^{\prime}(\frac s {r_0}) \log s ds+\mathcal{O}(\rho^3|\log \rho|)\\
				&=& -\frac{\varrho(\xi)\rho^2\log \rho}{2}+ 2\rho^2 \left(\frac{\varrho(\xi)}{8}+\int_{\Sigma} \chi_{\xi} \frac{1- e^{-\varphi_{\xi}\circ y_{\xi}^{-1}}}{|y_{\xi}|^2} dv_g -\int_{\Sigma} \frac 1 {r_0} \chi^{\prime}\left(\frac{|y_{\xi}|} {r_0}\right) \frac{ \log|y_{\xi}|}{|y_{\xi}|}dv_g \right)+\mathcal{O}(\rho^3 |\log \rho|).
			\end{array}
		\end{equation}	
		If $\xi\in \intsigma$,  $\partial _{ \nu_g }\eta_{\rho,\xi}=2\log \left(1+\frac{\rho^2}{|y_{\xi}(x)|^2}\right) \partial_{\nu_g} \chi_{\xi}  +2 \chi_{\xi} \partial_{ \nu_g }\log \left(1+\frac{\rho^2}{|y_{\xi}(x)|^2}\right) \equiv 0$ on $\partial \Sigma$.  
		If $\xi\in \partial \Sigma$,  \eqref{eq:out_normal_derivatives} yields that 
		$  \partial_{\nu_g} |y_{\xi}(x)|^2=  -e^{-\frac 1 2\varphi_{\xi}\circ y_{\xi}^{-1}(y) }  \frac{\partial |y|^2}{\partial y_2} |_{y=y_{\xi}(x)} \equiv0$ on $U(\xi)\cap \partial\Sigma$. 
		For any $x\in \partial \Sigma$, 
		\begin{equation} \label{eq:f_i_b}
			\begin{array}{lll}
				\partial_{ \nu_g }\eta_{\rho,\xi}(x)
				&=&2\frac{\rho^2}{|y_{\xi}(x)|^2}\partial_{ \nu_g}\chi_{\xi} +2\chi_{\xi} \frac{\rho^2}{\rho^2+ |y_{\xi}(x)|^2} \frac{ \partial_{\nu_g} |y_{\xi}(x)|^2}{|y_{\xi}(x)|^2}+\mathcal{O}(\rho^4) \\
				&=& \mathcal{O}(\rho^4),\quad\text{ as } \rho\rightarrow 0.
		\end{array}\end{equation}
		By \eqref{eqspan}, we have 
		\begin{equation}
			\label{eq:int_av}
			\begin{array}{lll}
				\int_{\Sigma} \chi_{\xi}(-\Delta_g)U_{\rho,\xi}dv_g&=&\int_{\Sigma}\chi_{\xi} e^{-\varphi_{\xi}\circ y_{\xi}^{-1}} e^{U_{\rho,\xi}} dv_g\\
				&=&\varrho(\xi)+4 \rho^2\int_{\Sigma}\frac 1 {r_0} \chi^{\prime}\left(\frac {|y_{\xi}|} {r_0}\right)\frac {e^{-\varphi_{\xi} \circ y_{\xi}^{-1}} } {|y_{\xi}|^3} dv_g+\mathcal{O}(\rho^4). 
			\end{array}
		\end{equation}
		We observe that for any $x\in U(\xi)$, $-\Delta_g U_{\rho,\xi}= e^{-{\varphi}_{\xi}(y)} \Delta u_{\rho,0}|_{ y=y_{\xi}(x)} =e^{-\varphi_{\xi}\circ y_{\xi}^{-1}} e^{U_{\rho,\xi}}$.  It follows that
		\begin{equation} 	~\label{eq:delta_f_xi} \begin{array}{lcl}
				&&	-\Delta_g \eta_{\rho,\xi}\stackrel{\eqref{eq:projection}\text{ and }\eqref{eq:eqR}}{=}-\Delta _g\left( PU_{\rho,\xi}-\chi_{\xi}(U_{\rho,\xi}-\log  (8\rho^2))\right)
				+\varrho(\xi)\Delta_g H^g(\cdot,\xi)    \\
				&=&-2\rho^2\left(\frac{1}{|y_{\xi}|^2}\Delta_g\chi_{\xi}+2 \left\la\nabla \chi_{\xi},\nabla \frac 1{|y_{\xi}|^2}\right\ra_g\right)+ \frac 1 {|\Sigma|_g}\left( \varrho(\xi)- \int_{\Sigma} \chi_{\xi} e^{-\varphi_{\xi}\circ y_{\xi}^{-1}} e^{U_{\rho,\xi}} dv_g \right)\\
				&&+\mathcal{O}(\rho^4)\\
				&\stackrel{\eqref{eq:int_av}}{=}& -2\rho^2 f_{\xi}+\mathcal{O}(\rho^4).
			\end{array}
		\end{equation}
		
		By $L^p$-estimate (see~\cite{Agmon1959,Wehrheim2004}, for instance)  and Sobolev embedding theorem  theorem , we have  
		\[ \eta_{\rho,\xi}=\overline{\eta_{\rho,\xi}}-2\rho^2 F_{\xi}+\mathcal{O}(\rho^4)= h_{\rho,\xi}-2\rho^2F_{\xi}+\mathcal{O}(\rho^3|\log \rho|), \]
		in $C(\Sigma)$. 
		The proof is concluded. 
	\end{proof} 
	\begin{lemma}~\label{lem:expansion1st} 
		\begin{equation*}
			\partial_{\xi_j}PU_{\rho,\xi}= \chi_{\xi} \partial_{\xi_j} U_{\rho,\xi} +\varrho(\xi) \partial_{\xi_j} H^g(x,\xi)-4\log |y_{\xi}(x)|\partial_{\xi_j}\chi_{\xi}  +{\mathcal{O}}(\rho^2|\log\rho|),
		\end{equation*}
		in $C(\Sigma)$,
		where $j=1,2$ if $\xi\in\intsigma$; $j=1$ if $\xi\in\partial\Sigma.$
		And 
		\[ 	\partial_{\xi_j}PU_{\rho,\xi}=\varrho(\xi) \partial_{\xi_j} G^g(x,\xi) +{\mathcal{O}}(\rho^2|\log\rho|),\]
		in $C_{loc}(\Sigma\setminus\{\xi\})$,
		which is uniformly convergent for $\xi$ in any compact subset of $\intsigma$ or $\xi\in \partial\Sigma$.
	\end{lemma}
	\begin{proof}
		$\eta_{\rho,\xi}^0= \partial_{\xi_j}PU_{\rho,\xi}- \chi_{\xi} \partial_{\xi_j} U_{\rho,\xi} -\varrho(\xi) \partial_{\xi_j} H^g(x,\xi)+4\log |y_{\xi}(x)|\partial_{\xi_j}\chi_{\xi}$.
		We observe that for $x\in U(\xi)$
		\begin{equation}\label{eq:extension_pa_y}
			\partial_{\xi_{j'}} y_{\xi}(x)_j= -\delta_{j'j}+ \mathcal{O}(|y_{\xi}(x)|). 
		\end{equation}
		If $\xi\in\intsigma$, $\partial_{\nu_g} \eta_{\rho,\xi}^0\equiv 0. $ If $\xi\in \partial\Sigma$, for any $x=y_{\xi}^{-1}(y)\in U(\xi)$
		\begin{equation*}
			\begin{array}{lll}
				\partial_{ \nu_g}\eta_{\rho,\xi}^0(x)&=& -\partial_{ \nu_g}\left( \chi_{\xi}(x) \partial_{\xi_j} U_{\rho,\xi}(x) +\varrho(\xi) \partial_{\xi_j} H^g(x,\xi)-4\partial_{\xi_j}\chi_{\xi}(x) \log |y_{\xi}(x)|\right)\\
				&=&0,
			\end{array}
		\end{equation*}
		where we applied that $\partial_{\nu_g}|y_{\xi}(x)|^2=0$ for any $x\in U(\xi)\cap\partial\Sigma$.
		\begin{equation*}
			\begin{array}{lcl}
				\int_{\Sigma} \eta_{\rho,\xi}^0dv_g &
				=& -\int_{\Sigma}\chi_{\xi} \partial_{\xi_j}\log\left(\frac{|y_{\xi}|^4}{(\rho^2+|y_{\xi}|^2)^2}\right)dv_g \\
				&\stackrel{\eqref{eq:extension_pa_y}}{=}& 2 \int_{\Sigma} \chi_{\xi} \frac{2\rho^2 y_{\xi}(x)_j}{\rho^2+|y_{\xi}(x)|^2}dv_g(x) + \mathcal{O}\left( \int_{\Sigma} \chi_{\xi} \frac{\rho^2 }{\rho^2+|y_{\xi}(x)|^2} dv_g(x) \right) \\
				&=& \mathcal{O}(\rho^2|\log \rho|),
			\end{array}
		\end{equation*}
		where we applied the fact that $\int_{B_{r_0}^{\xi}} \frac{y_j}{\rho^2+|y|^2} dy=0$, for $j=1,\cdots, \i(\xi).$ 
		It follows d
		\begin{equation*}
			\begin{array}{lcl}
				-\Delta_g\eta_{\rho,\xi}^0&=&  \mathcal{O}(\rho^2). 
			\end{array}
		\end{equation*}
		By $L^p$-estimate (see~\cite{Agmon1959,Wehrheim2004}, for instance)  and Sobolev embedding theorem  theorem , we have  
		\[ \eta^0_{\rho,\xi}=\mathcal{O}(\rho^2|\log\rho|)\]
		in $C(\Sigma)$. 
	\end{proof}

	\begin{lemma}~\label{expansion_rho} As $\rho\rightarrow 0,$
		\begin{equation} \label{eq:rho_p_rho1}
			\rho\partial_{\rho}PU_{\rho,\xi}=\chi_{\xi}(\rho \partial_{\rho}U_{\rho,\xi}-2) -4\rho^2F_{\xi}+h^1_{\rho,\xi}+o(\rho^2),
		\end{equation}
		and 
		\begin{equation}\label{eq:rho_p_rho2}
			\rho^2	\partial^2_{\rho}PU_{\rho,\xi}= \chi_{\xi}(2+\rho^2 \partial^2_{\rho} U_{\rho,\xi} )-4\rho^2F_{\xi
			}+h^2_{\rho,\xi}+o(\rho^2),
		\end{equation}
		in $C(\Sigma)$, where 
		\[h^1_{\rho,\xi}=-\frac{\varrho(\xi)}{|\Sigma|_g}\rho^2\log\rho+\frac{4}{|\Sigma|_g}\left( \int_{\Sigma} \chi_{\xi}\frac{1-e^{-\varphi_{\xi}\circ y_{\xi}^{-1}}}{|y_{\xi}|^2}dv_g-\int_{\Sigma} \frac 1 {r_0}\chi^{\prime}\left(\frac {|y_{\xi}|} {r_0}\right)\frac{\log |y_{\xi}|}{|y_\xi|} dv_g\right)\rho^2\]	and 
		\[ h^2_{\rho,\xi}=	\frac{-\varrho(\xi)\rho^2\log\rho}{|\Sigma|_g}+\frac {4}{|\Sigma|_g}\left( \int_{\Sigma} \chi_{\xi}\frac{1-e^{-\varphi_{\xi}\circ y_{\xi}^{-1}}}{|y_{\xi}|^2}dv_g-\int_{\Sigma} \frac 1 {r_0}\chi^{\prime}\left(\frac {|y_{\xi}|} {r_0}\right)\frac{\log |y_{\xi}|}{|y_\xi|} dv_g- \frac{\varrho(\xi)}{4}\right)\rho^2.\]
		Moreover,	the convergences are uniform for $\xi$ in any compact subset of $\intsigma$ or $\xi\in \partial\Sigma$.
	\end{lemma}
	\begin{proof}
		We observe that for any $x\in U(\xi)$, we have $\rho\partial_{\rho} U_{\rho,\xi}-2=\frac{-4\rho^2}{\rho^2+|y_{\xi}(x)|^2}.$
		Let $$\eta_{\rho,\xi}^1= \rho\partial_{\rho}PU_{\rho,\xi}-\chi_{\xi}(\rho \partial_{\rho}U_{\rho,\xi}-2).$$
		\begin{equation*}
			\begin{array}{lll}
				\int_{\Sigma} \eta_{\rho,\xi}^1 dv_g &=& \int_{\Sigma}\frac{4\rho^2\chi_{\xi}}{\rho^2+|y_{\xi}|^2}dv_g \\ 
				&=&\int_{B^{\xi}_{2r_0}}\chi(|y|/r_0)\frac{4\rho^2}{\rho^2+|y|^2}dy + 4\rho^2\int_{\Sigma} \chi_{\xi}\frac{1-e^{-\varphi_{\xi}\circ y_{\xi}^{-1}}}{|y_{\xi}|^2}dv_g+\mathcal{O}(\rho^3)
			\end{array}
		\end{equation*}
		By a direct calculation,  \[\int_{B^{\xi}_{2r_0}}\frac{4\rho^2}{\rho^2+|y|^2}dy=-\varrho(\xi)\rho^2\log \rho+\varrho(\xi)\log(2r_0)\rho^2+\mathcal{O}(\rho^4).\]
		By integrating by part, 
		\begin{equation*}
			\begin{array}{lll}
				\int_{B^{\xi}_{2r_0}}\chi(|y|/r_0)\frac{4\rho^2}{\rho^2+|y|^2}dy&=&\varrho(\xi)\rho^2 \int_{r_0}^{2r_0} \chi(\frac s {r_0})\frac{s}{\rho^2+s^2}ds\\
				&=& -\varrho(\xi)\log(2r_0)\rho^2-\varrho(\xi)\rho^2\int_0^{\infty} \frac 1 {r_0}\chi^{\prime}(\frac s {r_0})\log s ds  +\mathcal{O}(\rho^4)
			\end{array}
		\end{equation*}
		Thus
		\begin{equation}
			\label{eq:int_eta1}
			\begin{array}{lll}
				\int_{\Sigma}
				&&\frac{4\rho^2\chi_{\xi}}{\rho^2+|y_{\xi}|^2}dv_g\\
				&
				=&-\varrho(\xi)\rho^2\log\rho+4\left( \int_{\Sigma} \chi_{\xi}\frac{1-e^{-\varphi_{\xi}\circ y_{\xi}^{-1}}}{|y_{\xi}|^2}dv_g-\int_{\Sigma} \frac 1 {r_0}\chi^{\prime}\left(\frac {|y_{\xi}|} {r_0}\right)\frac{\log |y_{\xi}|}{|y_\xi|} dv_g\right)\rho^2+\mathcal{O}(\rho^4).
			\end{array}
		\end{equation}
		For $x\in \partial\Sigma$, we have 
		\[ \partial_{ \nu_g}\eta_{\rho,\xi}^1= \frac{4\rho^2}{|y_{\xi}(x)|^2}\partial_{ \nu_g}\chi_{\xi}+\chi_{\xi} \partial_{\nu_g}\frac{4\rho^2}{\rho^2+|y_{\xi}(x)|^2}+\mathcal{O}(\rho^4). \]
		If $\xi\in \intsigma$, $\partial_{ \nu_g}\eta_{\rho,\xi}^1\equiv0$; if $\xi\in \partial\Sigma$, $\partial_{\nu_g}\frac{4\rho^2}{\rho^2+|y_{\xi}(x)|^2}= \frac{-\partial}{\partial y_2} \frac{4\rho^2}{\rho^2+|y|^2}|_{y=y_{\xi}(x)}\equiv 0$  and 
		$\partial_{\nu_g}\chi\left(\frac {|y_{\xi}(x)|} {r_0}\right)=-\chi'\left(\frac{|y|}{r_0}\right)\frac{y_2}{|y|}|_{y=y_{\xi}(x)}\equiv 0$ for any $x\in U(\xi)\cap \partial\Sigma.$
		So $$\partial_{ \nu_g }\eta_{\rho,\xi}^1=\mathcal{O}(\rho^4),$$ on the boundary $\partial\Sigma$. 
		Applying~\eqref{eq:projection}, we deduce that 
		\begin{equation*}
			\begin{array}{lll}
				-\Delta_g \eta_{\rho,\xi}^1&=& -\Delta _g\left( \rho\partial_{\rho}PU_{\rho,\xi}-\chi_{\xi}(\rho \partial_{\rho}U_{\rho,\xi}-2)  \right)\\
				&=&-4\rho^2\left((\Delta_g\chi_{\xi})\frac{1}{|y_{\xi}(x)|^2}+2 \left\la\nabla \chi_{\xi},\nabla \frac 1{|y_{\xi}(x)|^2}\right\ra_g\right)+\frac{1}{|\Sigma|_g} \int_{\Sigma}\chi_{\xi}\rho\partial_{\rho} \Delta_g U_{\rho,\xi}dv_g+\mathcal{O}(\rho^4). 
			\end{array}
		\end{equation*}
		Considering $-\Delta_g U_{\rho,\xi}= e^{-\varphi_{\xi}\circ y_{\xi}^{-1}} e^{U_{\rho,\xi}}$ in $U(\xi)$, we can derive that  from Corollary~\ref{cor0}
		\begin{equation*}
			\begin{array}{lcl}
				\int_{\Sigma}\chi_{\xi}\rho\partial_{\rho} \Delta_g U_{\rho,\xi}dv_g&=&- \int_{\Sigma} \chi_{\xi} e^{-\varphi_{\xi}\circ y_{\xi}^{-1}}  e^{U_{\rho,\xi}}\rho\partial_{\rho} U_{\rho,\xi} dv_g \\
				&=& -\int_{B^{\xi}} \chi(|y|/r_0)\frac{8\rho^2}{(\rho^2+|y|^2)^2}\left(2-\frac{4\rho^2}{\rho^2+|y|^2}\right)dy\\
				&=&-8\rho^2\int_{\Sigma}\frac 1 {r_0} \chi^{\prime}\left(\frac{|y_{\xi}|}{r_0}\right)\frac{e^{-\varphi_{\xi}\circ y_{\xi}^{-1}} }{ |y_{\xi}|^3 } dv_g+o(\rho^2). 
			\end{array}
		\end{equation*}
		Hence, $	-\Delta_g \eta_{\rho,\xi}^1=-4\rho^2f_{\xi}+o(\rho^2).$
		By $L^p$-estimate (see~\cite{Agmon1959,Wehrheim2004}, for instance)  and Sobolev embedding theorem  theorem , we have 
		\[ \|\eta^1_{\rho,\xi}+4\rho^2F_{\xi}-\overline{\eta^1_{\rho,\xi}}\|_{W^{1+s,p}}\leq C\left(\|-\Delta_g\eta^1_{\rho,\xi}\|_{L^p(\Sigma)}+  \|\partial_{\nu_g}\eta^1_{\rho,\xi}\|_{L^p(\partial\Sigma)}\right)\leq o(\rho^2).\]
		Hence,
		\[ \rho\partial_{\rho}PU_{\rho,\xi}=\chi_{\xi}(\rho \partial_{\rho}U_{\rho,\xi}-2) -4\rho^2F_{\xi}+h^1_{\rho,\xi}+o(\rho^2),\]
		where $h^1_{\rho,\xi}=-\frac{\varrho(\xi)}{|\Sigma|_g}\rho^2\log\rho+\frac{4}{|\Sigma|_g}\left( \int_{\Sigma} \chi_{\xi}\frac{1-e^{-\varphi_{\xi}\circ y_{\xi}^{-1}}}{|y_{\xi}|^2}dv_g-\int_{\Sigma} \frac 1 {r_0}\chi^{\prime}\left(\frac {|y_{\xi}|} {r_0}\right)\frac{\log |y_{\xi}|}{|y_\xi|} dv_g\right)\rho^2.$
		The estimate \eqref{eq:rho_p_rho1} is concluded.
		
		Let $\eta_{\rho,\xi}^2=	\rho^2	\partial^2_{\rho}PU_{\rho,\xi}- \chi_{\xi}(2+\rho^2 \partial^2_{\rho} U_{\rho,\xi} ). $
		We observe that $ 2+\rho^2\partial^2 U_{\rho,\xi}= \frac{4\rho^2(\rho^2-|y_{\xi}(x)|^2)}{(\rho^2+|y_{\xi}(x)|^2)^2}$ in $U(\xi)$. 
		By~\eqref{eqspan} and~\eqref{eq:int_eta1}, we can obtain
		\begin{eqnarray*}
			\int_{\Sigma}\eta_{\rho,\xi}^2 dv_g&=& -\int_{\Sigma} \chi_{\xi} \frac{4\rho^2(\rho^2-|y_{\xi}(x)|^2)}{(\rho^2+|y_{\xi}(x)|^2)^2} dv_g\\
			&=& \int_{\Sigma}\frac{4\rho^2\chi_{\xi}}{\rho^2+|y_{\xi}|^2}dv_g-\rho^2 \int_{\Sigma}\frac{8\rho^2\chi_{\xi}}{(\rho^2+|y_{\xi}|^2)^2} dv_g\\
			&=& -\varrho(\xi)\rho^2\log\rho+4\left( \int_{\Sigma} \chi_{\xi}\frac{1-e^{-\varphi_{\xi}\circ y_{\xi}^{-1}}}{|y_{\xi}|^2}dv_g-\int_{\Sigma} \frac 1 {r_0}\chi^{\prime}\left(\frac {|y_{\xi}|} {r_0}\right)\frac{\log |y_{\xi}|}{|y_\xi|} dv_g- \frac{\varrho(\xi)}{4}\right)\rho^2\\
			&&
			+ \mathcal{O}(\rho^3|\log \rho|).
		\end{eqnarray*}
		Let $h^2_{\rho,\xi}=\frac{-\varrho(\xi)\rho^2\log\rho}{|\Sigma|_g}+\frac {4}{|\Sigma|_g}\left( \int_{\Sigma} \chi_{\xi}\frac{1-e^{-\varphi_{\xi}\circ y_{\xi}^{-1}}}{|y_{\xi}|^2}dv_g-\int_{\Sigma} \frac 1 {r_0}\chi^{\prime}\left(\frac {|y_{\xi}|} {r_0}\right)\frac{\log |y_{\xi}|}{|y_\xi|} dv_g- \frac{\varrho(\xi)}{4}\right)\rho^2.$
		Analogue to calculation of $\partial_{\nu_g}\eta_{\rho,\xi}^1$, we can obtain $
		\partial_{ \nu_g}\eta_{\rho,\xi}^2= \mathcal{O}(\rho^4)$. 
		\begin{equation}
			\begin{array}{lll}
				-\Delta_g\eta_{\rho,\xi}^2&=&-\Delta _g\left( \rho^2\partial^2_{\rho}PU_{\rho,\xi}-\chi_{\xi}(2+\rho^2\partial^2_{\rho}U_{\rho,\xi})  \right)\\
				&=&-4\rho^2\left((\Delta_g\chi_{\xi})\frac{1}{|y_{\xi}|^2}+2 \left\la\nabla \chi_{\xi},\nabla \frac 1{|y_{\xi}|^2}\right\ra_g\right)+\frac{1}{|\Sigma|_g} \int_{\Sigma}\chi_{\xi}\rho^2\partial^2_{\rho} \Delta_g U_{\rho,\xi}dv_g+\mathcal{O}(\rho^4). 
			\end{array}
		\end{equation}
		Corollary~\ref{cor0} yields that 
		\begin{equation*}
			\begin{array}{lcl}
				\rho^2\partial^2_{\rho}	\int_{\Sigma}\chi_{\xi}\Delta_gU_{\rho,\xi}dv_g&=& 
				-	\int_{B_{2r_0}^{\xi}}\chi(|y|/r_0) \left(\frac{16\rho^2}{(\rho^2+|y|^2)^2}- \frac{160\rho^4}{(\rho^2+|y|^2)^3}+ \frac{192\rho^6}{(\rho^2+|y|^2)^4}\right) dy\\
				&=&
				-	2\varrho(\xi)+ 10\varrho(\xi)-8\varrho(\xi)-8 \rho^2\int_{\Sigma}\frac 1 {r_0} \chi^{\prime}\left(\frac {|y_{\xi}|} {r_0}\right)\frac {e^{-\varphi_{\xi} \circ y_{\xi}^{-1}} } {|y_{\xi}|^3} dv_g+o(\rho^2)\\
				&=& -8 \rho^2\int_{\Sigma}\frac 1 {r_0} \chi^{\prime}\left(\frac {|y_{\xi}|} {r_0}\right)\frac {e^{-\varphi_{\xi} \circ y_{\xi}^{-1}} } {|y_{\xi}|^3} dv_g+ o(\rho^2).
			\end{array}
		\end{equation*}
		Thus, $-\Delta_g\eta_{\rho,\xi}^2= -4\rho^2f_{\xi}+o(\rho^2)$. 
		By $L^p$-estimate (see~\cite{Agmon1959,Wehrheim2004}, for instance)  and Sobolev embedding theorem  theorem ,
		\[ \|\eta^2_{\rho,\xi}+4\rho^2F_{\xi}-\overline{\eta^2_{\rho,\xi}}\|_{C(\Sigma)}\leq o(\rho^{2}).\]

		The estimate \eqref{eq:rho_p_rho2} is concluded.
	\end{proof}
	\begin{lemma}~\label{lem:expansion_mix} 
		\begin{equation*}
			\partial_{\xi_j}\partial_{\rho}PU_{\rho,\xi}= \chi_{\xi} \partial_{\xi_j}\partial_{\rho}U_{\rho,\xi}  +{\mathcal{O}}(\rho|\log\rho|),
		\end{equation*}
		in $C(\Sigma)$,
		where $j=1,2$ if $\xi\in\intsigma$; $j=1$ if $\xi\in\partial\Sigma.$
		And 
		\[ 	\partial_{\xi_j} \partial_{\rho}PU_{\rho,\xi}=\varrho(\xi) \partial_{\xi_j} \partial_{\rho}G^g(x,\xi) +{\mathcal{O}}(\rho|\log\rho|),\]
		in $C_{loc}(\Sigma\setminus\{\xi\})$,
		which is uniformly convergent for $\xi$ in any compact subset of $\intsigma$ or $\xi\in \partial\Sigma$.
	\end{lemma}
	\begin{proof}
		$\eta_{\rho,\xi}^3= \partial_{\xi_j}\partial_{\rho}PU_{\rho,\xi}- \chi_{\xi} \partial_{\xi_j}\partial_{\rho} U_{\rho,\xi}$.
		If $\xi\in\intsigma$, $\partial_{\nu_g} \eta_{\rho,\xi}^3\equiv 0. $ If $\xi\in \partial\Sigma$, for any $x=y_{\xi}^{-1}(y)\in U(\xi)$
		\begin{equation*}
			\begin{array}{lll}
				\partial_{ \nu_g}\eta_{\rho,\xi}^0(x)&=& -\partial_{ \nu_g}\left( \chi_{\xi}(x) \partial_{\xi_j}\partial_{\rho} U_{\rho,\xi}(x)\right)=0,
			\end{array}
		\end{equation*}
		where we applied that $\partial_{\nu_g}|y_{\xi}(x)|^2=0$ for any $x\in U(\xi)\cap\partial\Sigma$.
		\begin{equation*}
			\begin{array}{lcl}
				\int_{\Sigma} \eta_{\rho,\xi}^3dv_g &
				=& -\int_{\Sigma}\chi_{\xi}\frac {4\rho\partial_{\xi_j}|y_{\xi}(x)|^2}{(\rho^2+|y_{\xi}(x)|^2)^2}dv_g \\
				&\stackrel{\eqref{eq:extension_pa_y}}{=}&\frac 1 \rho\left( -\int_{\Sigma} \chi_{\xi} \frac{4\rho^2 y_{\xi}(x)_j}{(\rho^2+|y_{\xi}(x)|^2)^2}dv_g(x) + \mathcal{O}\left( \int_{\Sigma} \chi_{\xi} \frac{\rho^2|y_{\xi}(x)|^2 }{(\rho^2+|y_{\xi}(x)|^2)^2} dv_g(x) \right) \right)\\
				&=& \mathcal{O}(\rho|\log \rho|),
			\end{array}
		\end{equation*}
		where we applied the fact that $\int_{B_{r_0}^{\xi}} \frac{y_j}{(\rho^2+|y|^2)} dy=0$. 
		\begin{equation*}
			\begin{array}{lcl}
				-\Delta_g\eta_{\rho,\xi}^3&=&-\partial_{\xi_j}\chi_{\xi}\partial_{\rho}\Delta_g U_{\rho,\xi}+ \overline{\partial_{\xi_j}\chi_{\xi}\partial_{\rho}\Delta_g U_{\rho,\xi}} \\
				&=& \mathcal{O}(\rho). 
			\end{array}
		\end{equation*}
		By $L^p$-estimate and Sobolev embedding theorem  theorem , we have  
		\[ \eta^3_{\rho,\xi}=\mathcal{O}(\rho|\log\rho|).\]
		in $C(\Sigma)$.  
	\end{proof}
	
	%%%%%%%%%%%%%%%%%%%%%%%%
	Analogue to the proof of Lemma~\ref{lemb1}, we can get the expansions of $PZ_{ij}$ (see Section~\ref{sec_linear}) for $i=1,2,\cdots, m$ and $j=0,1,\cdots, \i(\xi_i)$. 
	\begin{lemma}~\label{lemb3}
		{\it
			\[PZ_{i0}(x)=\chi_{i}(x) \left(Z_{i0}(x) +2\right) +{\mathcal{O}}(\rho^2|\log \rho |) = 4\chi_i(x)\frac{\rho_i^2}{ \rho_i^2 +|y_{\xi_i}(x)|^2}+{ \mathcal{O}} (\rho^2|\log \rho |), \] 
			in $C(\Sigma)$  as $\rho\rightarrow 0$. 
			And 
			\[ PZ_{i0}(x)= \mathcal{O}(\rho^2|\log \rho|), \]
			in $C_{loc}(\Sigma \setminus\{\xi\})$ as $\rho\rightarrow 0$.
			\[PZ_{ij}(x)
			= \chi_{i}(x)Z_{ij}(x)+ {\mathcal{O}}(\rho),
			\] 
			in $C(\Sigma)$ as $\rho\rightarrow 0$, and 
			\[PZ_{ij}(x)=\mathcal{O}(\rho), \]
			in $C_{loc}(\Sigma\setminus\{\xi_i\})$ as $\rho\rightarrow 0$, where   $j=1,\cdots, \i(\xi_i)$. 
			In addition, the convergences above are 
			uniform for $\xi_i$ in any compact subset of $\intsigma$ or  $\xi_i\in \partial \Sigma$. }
		
	\end{lemma}
	\begin{proof}
		Recall that for any $x\in U_{2r_0}(\xi_i)$, 
		$ Z_{i0}(x)=2
		\frac{(\rho_i^2- |y_{\xi_i}(x)|^2)}{\rho_i^2+ |y_{\xi_i}(x)|^2}$ and $\chi_i=\chi(|y_{\xi}(x)|/r_{0})$.
		Let $$\eta_{i} = PZ_{i0}-\chi_i\left( Z_{i0}+2\right). $$
		If $\xi_i\in\intsigma$, $\partial_{\nu_g} \eta_i\equiv 0$ on the boundary $\partial\Sigma.$
		If $\xi_i\in \partial\Sigma$, 
		for $x\in\partial \Sigma$, 
		\begin{equation*}
			\begin{array}{lll}
				\partial_{ \nu_g } \eta_i &=& -(Z_{i0}+2)\partial_{ \nu_g } \chi_i+ \chi_i \partial_{ \nu_g } Z_{i0}\\
				&=& -  \chi_i(x)\frac{ 8\rho^2_i| y_{\xi_i}(x)|^2}{( | y_{\xi_i}(x)|^2+\rho_i^2)^2 } \partial_{ \nu_g } \log | y_{\xi_i}(x)|+\mathcal{O}(\rho^2),
			\end{array}
		\end{equation*}
		where we applied $\partial_{ \nu_g } \chi_{\xi}=0$ on  $U_{r_0}(\xi)$. For any $x\in \partial\Sigma\cap U_{2r_0}(\xi_i)$, 
		\begin{equation*}
			\begin{array}{lll}
				\partial_{ \nu_g } \log {|y_{\xi_i}(x)|}&=& \left. -\frac{\partial}{\partial y_2} \log|y| \right|_{y=y_{\xi}(x)}=\left.  -\frac{ 2y_2}{|y|^2}\right|_{y=y_{\xi}(x)}=0.
			\end{array}
		\end{equation*}
		Hence, $ \partial_{ \nu_g }\eta_i=\mathcal{O}(\rho^2)$  in $\partial \Sigma.$
		\begin{equation*}
			\begin{array}{ll}
				\left|	\int_{\Sigma} \chi_{i} \left( Z_{i0}(x) +2\right) dv_g\right| &=4
				\int_{\Sigma}\chi_{i} \frac{ \rho^2_i}{ |y_{\xi_i}(x)|^2 +\rho_i^2 } dv_g(x)\\
				&= 2 \rho^2_i \int_{ B_{2r_0}^{\xi}} \chi(|y|) \frac{1}{|y|^2+\rho_i^2} e^{{\varphi}_{\xi_i} (y)} dy\\
				&\leq \mathcal{O}\left(2\rho^2_i \int_{ B_{r_0}^{\xi}} \frac{1}{|y|^2+\rho_i^2}  dy\right) +\mathcal{O}(\rho)\\
				&= \mathcal{O}(\rho^2|\log \rho|).
			\end{array}
		\end{equation*}
		By Lemma~\ref{lem0},
		\begin{equation*}
			\begin{array}{lcl}
				-\Delta_g \eta_{i}(x)&=&  -\Delta_g\left(  PZ_{i0} - \chi_{i}\left( Z_{i0}+2\right)   \right)\\
				&=&
				\Delta_g \chi_{i} \left(Z_{i0}+2\right) +2 \left\la \nabla \chi_{i}, \nabla Z_{i0} \right\ra_g + \overline{ \chi_i\Delta_g Z_{i0}}  \\
				&=& \frac 1 {|\Sigma|_g }\int_{B^{\xi_i }_{r_0} } 2 \Delta\frac{\rho_i^2- |y|^2}{\rho_i^2+ |y|^2} dy  +  \mathcal{O}(\rho^2)\\
				&=&\frac 1 {|\Sigma|_g }\int_{B_{r_0}^{\xi_i}} \frac{16\rho^2_i(|y|^2-\rho_i^2)}{(\rho_i^2+|y|^2)^3} dy + \mathcal{O}(\rho^2)
				\\
				&\stackrel{\eqref{int01}\text{ and } \eqref{int02}}{=}&  \mathcal{O}(\rho^2).
			\end{array}
		\end{equation*}
		By the $L^p$-estimate and the Sobolev embedding theorem  theorem , 
		$\eta_{i}= {\mathcal{O}}(\rho^2|\log \rho|)$ in $C(\Sigma).$
		%%%%%%%%%%%%%%%%%%%%%
		
		Recall that for any $x\in U_{2r_0}(\xi_i)$, 
		$ Z_{ij}(x)=
		\frac{4 \rho_i y_{\xi_i}(x)_j }{\rho_i^2+ |y_{\xi_i}(x)|^2}$. 
		Let $\eta^1_i= \chi_{i}(x)Z_{ij}(x).$
		If $\xi_i\in\intsigma$, $\partial_{\nu_g} \eta^1_i\equiv 0$ on the boundary $\partial\Sigma.$
		If $\xi_i\in \partial\Sigma$, 
		for $x\in\partial \Sigma$, 
		\begin{equation*}
			\begin{array}{lll}
				\partial_{ \nu_g } \eta^1_i &=& -Z_{ij}\partial_{ \nu_g } \chi_i+ \chi_i \partial_{ \nu_g } Z_{ij}\\
				&=&\mathcal{O}(\rho),
			\end{array}
		\end{equation*}
		where we applied $\partial_{ \nu_g } \chi_{\xi}=0$ on  $U_{r_0}(\xi)$. 
		
		\begin{equation*}
			\begin{array}{ll}
				\int_{\Sigma} \chi_{i}  Z_{ij}(x) dv_g &=
				\int_{\Sigma}\chi_{i} \frac{4 \rho_i y_{\xi_i}(x)_j }{\rho_i^2+ |y_{\xi_i}(x)|^2} dv_g(x)\\
				&= 4 \rho_i \int_{ B_{r_0}^{\xi_i}}  \frac{y_j}{|y|^2+\rho_i^2} e^{{\varphi}_{\xi_i} (y)} dy +\mathcal{O}(\rho)\\
				&= \mathcal{O}(\rho).
			\end{array}
		\end{equation*}
		By~\eqref{int13}, for any $i=1,\cdots,m$ and $j=1,\cdots, \i(\xi_i)$, we have 
		\begin{equation*}
			\begin{array}{lcl}
				-\Delta_g \eta^1_{i}(x)&=&  -\Delta_g\left(  PZ_{ij} - \chi_{i} Z_{ij}\right)\\
				&=&
				Z_{ij}\Delta_g \chi_{i}  +2 \left\la \nabla \chi_{i}, \nabla Z_{ij} \right\ra_g + \overline{ \chi_i\Delta_g Z_{ij}}  \\
				&=& \frac 1 {|\Sigma|_g }\int_{B^{\xi_i }_{r_0} } \Delta\frac{4 \rho_i y_j }{\rho_i^2+ |y|^2} dy  +  \mathcal{O}(\rho)\\
				&=&\frac 1 {|\Sigma|_g }\int_{B_{r_0}^{\xi_i}} \frac{32\rho^3_iy_j}{(\rho_i^2+|y|^2)^3} dy + \mathcal{O}(\rho)
				\\
				&=&  \mathcal{O}(\rho).
			\end{array}
		\end{equation*}
		
		By the $L^p$-estimate and the Sobolev embedding theorem  theorem , 
		$\eta^1_{i}= \mathcal{O}(\rho)$ in $C(\Sigma)$. 
	\end{proof}
	
	The asymptotic ``orthogonality'' properties of $PZ_{ij}$ as $\rho\rightarrow 0$. 
	\begin{lemma}~\label{lem4}
		For any $i,l=1,2,\cdots,m, j=0,1,\cdots, \i(\xi_i)$ and $t=0,1,\cdots, \i(\xi_l)$, 
		\[	\langle PZ_{ij}, PZ_{lt}\rangle_g= \frac{4\varrho(\xi_i)}{3}\delta_{il}\delta_{jt}+\mathcal{O}(\rho).\]  
	\end{lemma}
	\begin{proof}
		Applying Lemma~\ref{lemb3} with the symmetric property, we deduce that  
		\begin{equation*} \begin{array}{lll}
				\langle PZ_{ij}, PZ_{lt}\rangle_g&=&- \int_{\Sigma} \chi_{i}(x)\Delta_g Z_{ij}(x) P Z_{lt}dv_g(x)\\
				&=&\begin{cases}
					\delta_{il}\int_{B^{\xi_i}_{r_0}}\frac{8\rho_i^2}{( \rho_i^2+|y|^2)^2}z_{j}\left(\frac{|y|}{\rho_i}\right)\left(  z_{t}\left(\frac{|y|}{\rho_i}\right)+ \mathcal{O}(\rho)\right)    +\mathcal{O}(\rho), &t\neq 0\\
					\delta_{il}\int_{B^{\xi_i}_{r_0}}\frac{8\rho_i^2}{( \rho_i^2+|y|^2)^2}z_{j}\left(\frac{|y|}{\rho_i}\right)\left(  z_{t}\left(\frac{|y|}{\rho_i}\right)+2+ \mathcal{O}(\rho)\right)    +\mathcal{O}(\rho), &t=0
				\end{cases} \\
				&=&\frac{\varrho(\xi_i)}{\pi}D_j\delta_{il}\delta_{jt}  + \mathcal{O}(\rho),
		\end{array}\end{equation*}
		where $\delta_{il}$'s are the Kronecker symbols and  $D_0= 4\int_{\RR^2} \frac{(1-|y|^2)^2}{(1+|y|^2)^4}dy, D_1=D_2=8 \int_{\RR^2} \frac{|y|^2}{(1+|y|^2)^4}dy$.
		Calculating the integrals, we derive that $D_j=\frac{4\pi}{3}$ for any $j=0,\cdots,2.$
	\end{proof}
	%	\cleardoublepage
	\newpage
	
	% Reference
	\bibliographystyle{plain} 
	\bibliography{name} 

\begin{thebibliography}{10}

\bibitem{Agmon1959}
Shmuel Agmon, Abraham Douglis, and Louis Nirenberg.
\newblock Estimates near the boundary for solutions of elliptic partial differential equations satisfying general boundary conditions. {I}.
\newblock {\em Comm. Pure Appl. Math.}, 12:623--727, 1959.

\bibitem{Ahmedou-Bartsch-Fiernkranz:2023}
Mohameden Ahmedou, Thomas Bartsch, and Tim Fiernkranz.
\newblock Equilibria of vortex type {H}amiltonians on closed surfaces.
\newblock {\em Topol. Methods Nonlinear Anal.}, 61(1):239--256, 2023.

\bibitem{Ahmedou_Ben_Morse2023}
Mohameden Ahmedou and Mohamed Ben~Ayed.
\newblock Morse inequalities at infinity for a resonant mean field equation.
\newblock {\em Commun. Contemp. Math.}, 25(01):2150054, 2023.

\bibitem{ahmedou_resonant_2017}
Mohameden Ahmedou, Mohamed Ben~Ayed, and Marcello Lucia.
\newblock On a resonant mean field type equation: {A} ``critical point at {Infinity}" approach.
\newblock {\em Discrete Contin. Dyn. Syst.}, 37(4):1789--1818, 2017.

\bibitem{baraket_construction_1997}
Sami Baraket and Frank Pacard.
\newblock Construction of singular limits for a semilinear elliptic equation in dimension 2.
\newblock {\em Calc. Var. Partial Differential Equations}, 6(1):1--38, December 1997.

\bibitem{Bartolucci2020}
Daniele Bartolucci, Changfeng Gui, Yeyao Hu, Aleks Jevnikar, and Wen Yang.
\newblock Mean field equations on tori: existence and uniqueness of evenly symmetric blow-up solutions.
\newblock {\em Discrete Contin. Dyn. Syst.}, 40(6):3093--3116, 2020.

\bibitem{Bartsch2017TheMP}
Thomas Bartsch, Anna~Maria Micheletti, and Angela Pistoia.
\newblock The morse property for functions of {K}irchhoff-{R}outh path type.
\newblock {\em Discrete Contin. Dyn. Syst. - S}, 12(7):1867--1877, 2019.

\bibitem{Battaglia2018}
Luca Battaglia.
\newblock A general existence result for stationary solutions to the {Keller}-{Segel} system.
\newblock {\em Discrete Contin. Dyn. Syst.}, 39(2):905--926, 2019.

\bibitem{bers1957riemann}
Lipman Bers.
\newblock {\em {R}iemann Surfaces}.
\newblock Courant Institute Lecture Notes. Courant Institute of Mathematical Sciences, New York University, New York, 1957.
\newblock Lecture notes from 1957-58.

\bibitem{brezis_uniform_1991}
Ha\"im Brezis and Frank Merle.
\newblock Uniform estimates and blow–up behavior for solutions of {$-\Delta u= V e^u $} in two dimensions.
\newblock {\em Comm. Partial Differential Equations}, 16(8-9):1223--1253, 1991.

\bibitem{caffarelli1995vortex}
Luis~A Caffarelli and Yisong Yang.
\newblock Vortex condensation in the {C}hern-{S}imons-{H}iggs model: an existence theorem.
\newblock {\em Comm. Math. Phys.}, 168(2):321--336, 1995.

\bibitem{caglioti1995special2}
Emanuele Caglioti, Pierre-Louis Lions, Carlo Marchioro, and M~Pulvirenti.
\newblock A special class of stationary flows for two-dimensional {E}uler equations: a statistical mechanics description. part {I}{I}.
\newblock {\em Comm. Math. Phys.}, 174(2):229--260, 1995.

\bibitem{caglioti1992special}
Emanuele Caglioti, Pierre-Louis Lions, Carlo Marchioro, and Mario Pulvirenti.
\newblock A special class of stationary flows for two-dimensional euler equations: a statistical mechanics description.
\newblock {\em Comm. Math. Phys.}, 143(3):501--525, 1992.

\bibitem{Chang_chen_lin2003}
Sun-Yung~A. Chang, Chiun-Chuan Chen, and Chang-Shou Lin.
\newblock Extremal functions for a mean field equation in two dimension.
\newblock In {\em Lectures on partial differential equations}, volume~2 of {\em New Stud. Adv. Math.}, pages 61--93. Int. Press, Somerville, MA, 2003.

\bibitem{Chang1993}
Sun-Yung~A. Chang, Matthew~J. Gursky, and Paul~C. Yang.
\newblock The scalar curvature equation on {$2$}- and {$3$}-spheres.
\newblock {\em Calc. Var. Partial Differential Equations}, 1(2):205--229, 1993.

\bibitem{chang1988conformal}
Sun-Yung~A Chang and Paul~C Yang.
\newblock Conformal deformation of metrics on { $S^{2} $}.
\newblock {\em J. Differential Geom.}, 27(2):259--296, 1988.

\bibitem{chang1987prescribing}
Sun-Yung~Alice Chang and Paul~C Yang.
\newblock Prescribing gaussian curvature on {$S^2$}.
\newblock {\em Acta Math.}, 159:215--259, 1987.

\bibitem{chanillo1994rotational}
Sagun Chanillo and Michael Kiessling.
\newblock Rotational symmetry of solutions of some nonlinear problems in statistical mechanics and in geometry.
\newblock {\em Comm. Math. Phys.}, 160(2):217--238, 1994.

\bibitem{chen2002sharp}
Chiun-Chuan Chen and Chang-Shou Lin.
\newblock Sharp estimates for solutions of multi-bubbles in compact {R}iemann surfaces.
\newblock {\em Comm. Pure Appl. Math.}, 55(6):728--771, 2002.

\bibitem{chen1987scalar}
Wenxiong Chen and Weiyue Ding.
\newblock Scalar curvatures on {$S^2$}.
\newblock {\em Trans. Amer. Math. Soc.}, 303(1):365--382, Sep 1987.

\bibitem{chern1955}
Shiing-shen Chern.
\newblock An elementary proof of the existence of isothermal parameters on a surface.
\newblock {\em Proc. Amer. Math. Soc.}, 6:771--782, 1955.

\bibitem{childress1984chemotactic}
Stephen Childress.
\newblock Chemotactic collapse in two dimensions.
\newblock In Willi J{\"a}ger and James~D. Murray, editors, {\em Modelling of Patterns in Space and Time}, pages 61--66, Berlin, Heidelberg, 1984. Springer Berlin Heidelberg.

\bibitem{ching1993nirenberg}
Chang~Kung Ching and Liujia Quan.
\newblock On {N}irenberg's problem.
\newblock {\em Internat. J. Math.}, 4(1):35--58, 1993.

\bibitem{del_pino_singular_2005}
Manuel del Pino, Michal Kowalczyk, and Monica Musso.
\newblock Singular limits in {Liouville}-type equations.
\newblock {\em Calc. Var. Partial Differential Equations.}, 24(1):47--81, September 2005.

\bibitem{ding2001self}
Weiyue Ding, J{\"u}rgen Jost, Jiayu Li, Xiaowei Peng, and Guofang Wang.
\newblock Self duality equations for {G}inzburg--{L}andau and {S}eiberg--{W}itten type functionals with 6th order potentials.
\newblock {\em Commun. Math. Phys.}, 217:383--407, 2001.

\bibitem{Ding1997TheDE}
Weiyue Ding, J\"urgen Jost, Jiayu Li, and Guofang Wang.
\newblock The differential equation {$\Delta u=8\pi-8\pi he^u$} on a compact {R}iemann surface.
\newblock {\em Asian J. Math.}, 1(2):230--248, 1997.

\bibitem{ding1999multiplicity}
Weiyue Ding, Jiayu Li, Guofang Wang, and Changyou Zhou.
\newblock Multiplicity results for the two-vortex {C}hern-{S}imons {H}iggs model on the two-sphere.
\newblock {\em Comment. Math. Helv.}, 74:118--142, 1999.

\bibitem{djadli2008existenceq_curvature}
Zindine Djadli and Andrea Malchiodi.
\newblock Existence of conformal metrics with constant {Q}-curvature.
\newblock {\em Ann. Math.}, pages 813--858, 2008.

\bibitem{Esposito2014singular}
Pierpaolo Esposito and Pablo Figueroa.
\newblock Singular mean field equations on compact {R}iemann surfaces.
\newblock {\em Nonlinear Anal.}, 111:33--65, 2014.

\bibitem{Esposito2005}
Pierpaolo Esposito, Massimo Grossi, and Angela Pistoia.
\newblock On the existence of blowing-up solutions for a mean field equation.
\newblock {\em Ann. Inst. Henri Poincaré C, Anal. Non Linéaire}, 22(2):227--257, 2005.

\bibitem{figueroa2022bubbling}
Pablo Figueroa.
\newblock Bubbling solutions for mean field equations with variable intensities on compact {R}iemann surfaces.
\newblock {\em J. Anal. Math.}, 152:507--555, 2024.

\bibitem{Hartman1955}
Philip Hartman and Aurel Wintner.
\newblock On uniform {D}ini conditions in the theory of linear partial differential equations of elliptic type.
\newblock {\em Amer. J. Math.}, 77:329--354, 1955.

\bibitem{BartschHuSubmitted}
Zhengni Hu and Thomas Bartsch.
\newblock The morse property of limit functions appearing in mean field equations on surfaces with boundary.
\newblock {\em J. Geom. Anal.}, 34:220, 2024.

\bibitem{kazdan1974curvature}
Jerry~L Kazdan and Frank~W Warner.
\newblock Curvature functions for compact {$2$}-manifolds.
\newblock {\em Ann. of Math. (2)}, 99:14--47, 1974.

\bibitem{keller1970}
Evelyn~F. Keller and Lee~A. Segel.
\newblock Initiation of slime mold aggregation viewed as an instability.
\newblock {\em J. Theoret. Biol.}, 26(3):399--415, 1970.

\bibitem{kiessling1993statistical}
Michael K.-H. Kiessling.
\newblock Statistical mechanics of classical particles with logarithmic interactions.
\newblock {\em Comm. Pure Appl. Math.}, 46(1):27--56, 1993.

\bibitem{lisunyang2023}
Jiayu Li, Linlin Sun, and Yunyan Yang.
\newblock The boundary value problem for the mean field equation on a compact {R}iemann surface.
\newblock {\em Sci. China Math.}, 66, 2023.

\bibitem{li1999existence}
Jiayu Li, Guofang Wang, Weiyue Ding, and J{\"u}rgen Jost.
\newblock Existence results for mean field equations.
\newblock {\em Ann. Inst. H. Poincar\'e{} C Anal. Non Lin\'eaire}, 16(5):653--666, 1999.

\bibitem{Li1997OnAS}
Yanyan Li.
\newblock On a singularly perturbed elliptic equation.
\newblock {\em Adv. Differential}, 2:955--980, 1997.

\bibitem{lin1941motion}
Chia-Ch'iao Lin.
\newblock On the motion of vortices in two dimensions. i. existence of the kirchhoff-routh function.
\newblock {\em Proc. Natl. Acad. Sci.}, 27(11):570--575, 1941.

\bibitem{ma_convergence_2001}
Li~Ma and J.~Wei.
\newblock Convergence for a liouville equation.
\newblock {\em Comment. Math. Helv.}, 76:506--514, 2001.

\bibitem{moser1973nonlinear}
J{\"u}rgen Moser.
\newblock On a nonlinear problem in differential geometry.
\newblock In {\em Dynamical systems}, pages 273--280. Elsevier, 1973.

\bibitem{Nagasaki1990AsymptoticAF}
Ken’ichi Nagasaki and Takashi Suzuki.
\newblock Asymptotic analysis for two-dimensional elliptic eigenvalue problems with exponentially dominated nonlinearities.
\newblock {\em Asymptotic Anal.}, 3:173--188, 1990.

\bibitem{nolasco1998sharp}
Margherita Nolasco and Gabriella Tarantello.
\newblock On a sharp {S}obolev-type inequality on two-dimensional compact manifolds.
\newblock {\em Arch. Ration. Mech. Anal.}, 145(2):161--195, 1998.

\bibitem{nolasco1999double}
Margherita Nolasco and Gabriella Tarantello.
\newblock Double vortex condensates in the {C}hern--{S}imons--{H}iggs theory.
\newblock {\em Calc. Var. Partial Differential Equations}, 9(1):31--94, 1999.

\bibitem{pino_collapsing_2006}
Manuel~del Pino and Juncheng Wei.
\newblock Collapsing steady states of the {Keller}–{Segel} system.
\newblock {\em Nonlinearity}, 19(3):661--684, January 2006.
\newblock Publisher: IOP Publishing.

\bibitem{Senba2000some}
Takasi Senba and Takashi Suzuki.
\newblock Some structures of the solution set for a stationary system of chemotaxis.
\newblock {\em Adv. Math. Sci. Appl.}, 10:191--224, 2000.

\bibitem{suzuki_two_1992}
Takashi Suzuki.
\newblock Two {dimensional} {Emden}-{Fowler} {equation} with {exponential} {nonlinearity}.
\newblock In N.~G. Lloyd, W.~M. Ni, L.~A. Peletier, and J.~Serrin, editors, {\em Nonlinear {Diffusion} {Equations} and {Their} {Equilibrium} {States}, 3: {Proceedings} from a {Conference} held {August} 20–29, 1989 in {Gregynog}, {Wales}}, pages 493--512. Birkhäuser Boston, Boston, MA, 1992.

\bibitem{tarantello1996multiple}
Gabriella Tarantello.
\newblock Multiple condensate solutions for the {C}hern-{S}imons-{H}iggs theory.
\newblock {\em J. Math. Phys.}, 37(8):3769--3796, 1996.

\bibitem{Vekua1955}
Ilia~Nestorovich Vekua.
\newblock The problem of reduction to canonical form of differential forms of elliptic type and the generalized {C}auchy-{R}iemann system.
\newblock {\em Dokl. Akad. Nauk SSSR (N.S.)}, 100:197--200, 1955.

\bibitem{Wang2002SteadySS}
Guofang Wang and Juncheng Wei.
\newblock Steady state solutions of a reaction-diffusion system modeling chemotaxis.
\newblock {\em Math. Nachr.}, 233/234:221--236, 2002.

\bibitem{Wehrheim2004}
Katrin Wehrheim.
\newblock {\em Uhlenbeck Compactness}.
\newblock European Mathematical Society, 2004.

\bibitem{yang2021125440}
Yunyan Yang and Jie Zhou.
\newblock Blow-up analysis involving isothermal coordinates on the boundary of compact {R}iemann surface.
\newblock {\em J. Math. Anal. Appl.}, 504(2):125440, 2021.

\end{thebibliography}
	\vspace{2mm}\noindent
	{\sc  Mohameden Ahmedou}\\
	Mathematisches Institut,
Universit\"{a}t Giessen,
Arndtstr.\ 2,
35392 Giessen, Germany\\
	\href{mailto:Mohameden.Ahmedou@math.uni-giessen.de}{Mohameden.Ahmedou@math.uni-giessen.de}
	\\
	{\sc  Thomas Bartsch}\\
	Mathematisches Institut,
	Universit\"{a}t Giessen,
	Arndtstr.\ 2,
	35392 Giessen, Germany\\
	\href{mailto:Thomas.Bartsch@math.uni-giessen.de}{Thomas.Bartsch@math.uni-giessen.de}
	\\
	{\sc Zhengni Hu}\\
School of Mathematical Sciences,
 Shanghai Jiao Tong University,
  800 Dongchuan RD, Minhang District, Shanghai, 200240, China
	\\
	\href{mailto:Zhengni_Hu2021@outlook.com}{zhengni\_hu2021@outlook.com}
	\\
	
	%\include{Hu_main_result}
	%\include{Hu_prelim}
	%\include{Hu_reduction}
	%\include{Hu_expansion}
	%\include{Hu_proof_main_result}
	%	\cleardoublepage
	%\include{Hu_appendix_full}
	
	%\cleardoublepage
	%\newpage
	%Reference

\end{document}